\documentclass[11pt,a4paper]{article}

\usepackage{amscd}
\usepackage{amsfonts,amsmath,amssymb,amsthm}
\usepackage{array}
\usepackage{graphicx}
\usepackage[all]{xy}
\usepackage[utf8]{inputenc}
\usepackage[english]{babel}
\usepackage[T1]{fontenc}
\usepackage{lmodern}
\usepackage{tikz}
\usetikzlibrary{decorations}
\usetikzlibrary{decorations.markings}
\usepackage{microtype}
\usepackage{float}
\usepackage{enumitem}
\setlist{noitemsep}
\usepackage[totalwidth=160mm, totalheight=247mm]{geometry}

\theoremstyle{plain}
\newtheorem{theorem}{Theorem}[section]

\newtheorem{lemma}[theorem]{Lemma}
\newtheorem{proposition}[theorem]{Proposition}
\theoremstyle{definition}
\newtheorem{remark}[theorem]{Remark}
\newtheorem{definition}[theorem]{Definition}
\newtheorem{example}[theorem]{Example}

\newtheorem*{theorem*}{Theorem}

\DeclareTextFontCommand{\tdef}{\itshape\bfseries}

\DeclareMathOperator{\supp}{supp}

\newcommand{\wreath}{\text{Wreath}}

\newcommand{\core}{\text{Core}}
\newcommand{\area}{\text{Area}}
\newcommand{\B}{\mathsf{B}}

\newcommand{\Tstep}{{\mathcal T_\textsf{step}}}
\newcommand{\Tfrac}{{\mathcal T_\textsf{frac}}}
\newcommand{\Ttop}{{\mathcal T_\textsf{top}}}
\newcommand{\Sstep}{{\Sigma_\textsf{step}}}

\newcommand{\Ms}{{\mathbf M}_s}
\newcommand{\EOS}{{{\mathbf E}_1^*}}
\newcommand{\EOSS}{{{\mathbf E}}}

\newcommand{\myvcenter}[1]{\ensuremath{\vcenter{\hbox{#1}}}}
\newcommand{\T}{\mathcal{T}}
\newcommand{\type}{\text{type}}
\newcommand{\PP}{\mathcal{PP}}
\newcommand{\bt}{\mathsf{b}}

\numberwithin{equation}{section}

\makeatletter
\newcommand{\Biggg}{\bBigg@{4}}
\makeatother

\newcommand{\svect}[3]{%
\Big(\begin{smallmatrix}%
#1 \\ #2 \\ #3%
\end{smallmatrix}\Big)%
}

\newcommand{\tileA}[9]{%
\fill[fill=#9,draw=black,thick,shift={#8}]
       (0,0,0)  node[above right] {\tiny $#2$}
    -- (0,1,0)  node[above] {\tiny $#3$}
    -- (-1,1,0) node[above left] {\tiny $#4$}
    -- (-1,1,1) node[below left] {\tiny $#5$}
    -- (-1,0,1) node[below] {\tiny $#6$}
    -- (0,0,1)  node[below right] {\tiny $#7$}
    -- cycle;
\node[shift={#8}] at (-1,0,0) {$#1$};
}

\def\tileC#1#2#3#4{\def\blaz{#1}\def\blaa{#2}\def\blab{#3}\def\blac{#4}\tileCbuff}
\def\tileCbuff#1#2#3#4#5#6#7#8#9{%
\fill[fill=#9,draw=black,thick,shift={#8}] (0,0,0)
       (0,0,0)   node[below] {\tiny $\blaa$}
    -- (0,-1,0)  node[below] {\tiny $\blab$}
    -- (1,-1,0)  node[below right ] {\tiny $\blac$}
    -- (1,-1,-1) node[above right] {\tiny $#1$}
    -- (1,0,-1)  node[above] {\tiny $#2$}
    -- (0,0,-1)  node[above] {\tiny $#3$}
    -- (0,1,-1)  node[above] {\tiny $#4$}
    -- (-1,1,-1) node[above left] {\tiny $#5$}
    -- (-1,1,0)  node[below left] {\tiny $#6$}
    -- (-1,0,0)  node[below] {\tiny $#7$}
    -- cycle;
\node[shift={#8}] at (1,0,0) {$\blaz$};
}

\def\tuileC#1#2#3#4{\def\blaa{#1}\def\blab{#2}\def\blac{#3}\def\blad{#4}\tuileCbuff}
\def\tuileCbuff#1#2#3#4#5#6#7#8#9{%
\fill[fill=\blac,draw=black,thick,shift={\blab}]
         (0,1,1) node[above left] {\tiny $\blad$}
    -- ++(0,0,1) node[below left] {\tiny $#1$}
    -- ++(0,-1,0) node[below] {\tiny $#2$}
    -- ++(1,0,0) node[below right] {\tiny $#3$}
    -- ++(0,0,-2) node[above right] {\tiny $#4$}
    -- ++(0,1,0) node[below right] {\tiny $#5$}
    -- ++(0,0,-1) node[above right] {\tiny $#6$}
    -- ++(0,1,0) node[above] {\tiny $#7$}
    -- ++(-1,0,0) node[above left] {\tiny $#8$}
    -- ++(0,0,2) node[below left] {\tiny $#9$}
    -- cycle;
\node[shift={\blab}] at (0,0.5,0) {$\blaa$};
}

\newcommand{\tuileB}[9]{%
\fill[fill=#3,draw=black,thick,shift={#2}]
         (2,0,0) node[above] {\tiny $#4$}
    -- ++(-1,0,0) node[above left] {\tiny $#5$}
    -- ++(0,0,2) node[below left] {\tiny $#6$}
    -- ++(0,-1,0) node[below] {\tiny $#7$}
    -- ++(1,0,0) node[below right] {\tiny $#8$}
    -- ++(0,0,-2) node[above right] {\tiny $#9$}
    -- cycle;
\node[shift={#2}] at (1,-1,0.5) {$#1$};
}

\newcommand{\tuileA}[9]{%
\fill[fill=#3,draw=black,thick,shift={#2}]
         (2,1,-1) node[above] {\tiny $#4$}
    -- ++(-1,0,0) node[above left] {\tiny $#5$}
    -- ++(0,0,1) node[below left] {\tiny $#6$}
    -- ++(0,-1,0) node[below] {\tiny $#7$}
    -- ++(1,0,0) node[below right] {\tiny $#8$}
    -- ++(0,0,-1) node[above right] {\tiny $#9$}
    -- cycle;
\node[shift={#2}] at (1,0,-1) {$#1$};
}

\newcommand{\sqtA}[1]{%
\fill[fill=white,draw=black,thick,shift={#1}]
(0, 0) -- (1, 0) -- (1, 1) -- (2, 1) -- (2, 2) -- (1, 2) -- (1, 3) -- (0, 3) -- (0, 2) -- (0, 1) -- cycle;
\node[shift={#1}] at (0.5,1.5) {$A$};
}

\newcommand{\sqtB}[1]{%
\fill[fill=white,draw=black,thick,shift={#1}]
(0, 0) -- (1, 0) -- (1, 1) -- (1, 2) -- (0, 2) -- (0, 1) -- cycle;
\node[shift={#1}] at (0.5,1.5) {$B$};
}

\newcommand{\sqtC}[1]{%
\fill[fill=white,draw=black,thick,shift={#1}]
(0, 0) -- (1, 0) -- (1, 1) -- (2, 1) -- (2, 2) -- (1, 2) -- (0, 2) -- (0, 2) -- cycle;
\node[shift={#1}] at (0.5,1.5) {$C$};
}

\begin{document}

\title{Topological substitution for the aperiodic Rauzy fractal tiling}

\date{}

\author{
Nicolas B\'edaride\footnote{
    Aix Marseille Univ, CNRS, Centrale Marseille, I2M, Marseille, France.
    Email: \texttt{nicolas.bedaride@univ-amu.fr}}
\and
Arnaud Hilion\footnote{
    Aix Marseille Univ, CNRS, Centrale Marseille, I2M, Marseille, France.
    Email: \texttt{arnaud.hilion@univ-amu.fr}}
\and
Timo Jolivet\footnote{
    Aix Marseille Univ, CNRS, LIF, Marseille, France.
    Email: \texttt{timo.jolivet@lif.univ-mrs.fr}}
    }

\maketitle


\begin{abstract}
We consider two families of planar self-similar tilings of different nature:
the tilings consisting of translated copies
of the fractal sets defined by an iterated function system,
and the tilings obtained as a geometrical realization of a topological substitution
(an object of purely combinatorial nature, defined in \cite{Beda.Hil.10}).
We establish a link between the two families in a specific case,
by defining an explicit topological substitution and by proving that
it generates the same tilings as those associated with the Tribonacci Rauzy fractal.
\end{abstract}

\section{Introduction}

\subsection{Main result and motivation}
\label{sec:motiv}
Self-similar tilings of the plane are characterized by the existence of
a common subdivision rule for each tile, such that the tiling obtained by subdivising each tile
is the same as the original one, up to a contraction.
These tilings have been introduced by Thurston \cite{Thu89}
and they are studied in several fields including dynamical systems and theoretical physics, see \cite{BG13}.
A particular class of self-similar tilings arises from \emph{substitutions},
which are ``inflation rules'' describing how to replace a geometrical shape
by a union of other geometrical shapes (within a finite set of basic shapes).
Among these, an important class consists of the planar tilings
by the so-called Rauzy fractals associated with some one-dimensional substitutions.
These fractals are used to provide geometrical interpretations of substitution dynamical systems.
They also provide an interesting class of aperiodic self-similar tilings of the plane, see~\cite{Pyt.02,CANT}.

The aim of this article is to establish a formal link between two self-similar tilings
constructed from two different approaches:
\begin{itemize}
\item
Using an \emph{iterated function system (IFS)},
that is, specifying the shapes and the positions of the tiles
with planar set equations (using contracting linear maps),
which define the tiles as unions of smaller copies of other tiles.
In particular, an IFS does make use of the Euclidean metric of the plane.
\item
Using a \emph{topological substitution},
that is, specifying which tiles are allowed to be neighbors,
and how the neighboring relations are transferred when we
``inflate'' the tiles by substitution to construct the tiling.
With this kind of substitution, there is no use in anyway of a the Euclidean metric: the tiles do not have a metric shape (they are just topological discs).
\end{itemize}
In other words, we tackle the following question:
\begin{quote}
\emph{Given a tiling defined by an IFS,
is there a topological substitution
which generates an equivalent tiling?
If yes, how can we construct it?
In other words, when is it possible to describe the geometry of a self-similar tiling (geometrical constraints)
by using a purely combinatorial rule (combinatorial constraints) ?
}
\end{quote}
In this article we answer this question for the tilings of the plane by translated copies of the Rauzy fractals
associated with the Tribonacci substitution (which are defined by an IFS).
We define a particular topological substitution $\sigma$ (Figure~\ref{tribo-noms},~p.~\pageref{tribo-noms})
and we prove that the Tribonacci fractal tiling $\Tfrac$
and the tiling $\Ttop$ generated by the topological substitution are equivalent in a strong way.
More precisely:

\begin{itemize}
\item Associated with the Tribonacci substitution $s: 1\mapsto 12$, $2\mapsto 13$, $3\mapsto 1$,
there is a dual substitution $\EOSS$ (see Section~\ref{sect:dual_sub})
which acts on facets in $\mathbb R^3$.
Iteration of this dual substitution gives rise to a stepped surface $\Sstep$ (a surface which is a union of facets),
that is included in the $1$-neighbourhood of some (linear) plane $\mathcal P$ in $\mathbb R^3$.
Projecting the stepped surface $\Sstep$ (and its facets) on $\mathcal P$ gives rise to a tiling $\Tstep$ of $\mathcal P$.
It is known \cite{ABI02,CANT} that the tiling $\Tfrac$ is strongly related to a tiling $\Tstep$. 

\item The topological substitution $\sigma$ can be iterated on a tile $C$, giving rises to a 2-dimensional CW-complex $\sigma^\infty(C)$ homeomorphic to a plane, see Section~\ref{sect:deftoptribo}. However, this complex is not embedded a priori in a plane, even if it turns out that $\sigma^\infty(C)$ can be effectively realized as a tiling $\Ttop$ of the plane, see Proposition~\ref{prop:topotiling}. To locate a tile $T$ in $\sigma^\infty(C)$ relatively to another one $T'$, we build a vector (an {\em ``position''}) $\omega_0(T,T')\in\mathbb Z^3$: by construction, this vector depends a priori on the choice of a combinatorial path from $T$ to $T'$ in $\sigma^\infty(C)$, and we have to prove that in fact it is independent of the path, see Section~\ref{sec:position}.

\item Since it is already explained in the literature how to relate $\Tfrac$ and $\Sstep$, and since  we explain how $\Ttop$ is build from $\sigma^\infty(C)$, the main result of the paper is Theorem~\ref{theo:Psi} that states an explicit formula which define a bijection $\Psi$ between tiles in $\sigma^\infty(C)$ and facets in $\Sstep$: we reproduce it just below.

\end{itemize}

\begin{theorem*}
The map $\Psi$ defined, for every tile $T$ of $\sigma^\infty(C)$, by:
\begin{equation}\label{eq:Psi}
\Psi(T)=[\mathbf M_s^3(\omega_0(T,C)+\mathbf u_{\type(T)}),\theta(\type(T))]^*
\end{equation}
is a bijection from the set of tiles of $\sigma^\infty(C)$ to the set of facets of $\Sigma_{step}$.
\end{theorem*}

The notation used to state this theorem will be introduced along the paper. But we want to stress that the fact the formula (\ref{eq:Psi}) makes use of the position map $\omega_0$ ensures that if two tiles $T$ and $T'$ are close in $\sigma^\infty(C)$, then their images $\Psi(T)$ and $\Psi(T')$ will be close in $\Sstep$. In fact, it is easy to convince oneself that something like that should be true by having a look at Figure~\ref{pav-frac}, where three corresponding subsets of the tilings $\Ttop$, $\Tfrac$ and $\Tstep$ are given.

\begin{figure}[ht]
\centering
\myvcenter{\includegraphics[width=0.3\linewidth]{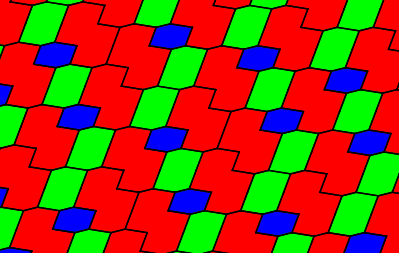}}
\hfil
\myvcenter{\includegraphics[width=0.3\linewidth]{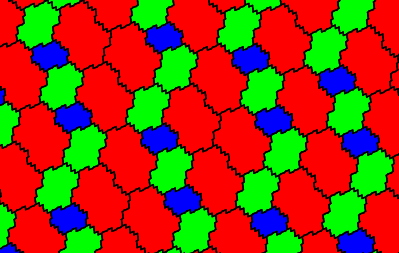}}
\hfil
\myvcenter{\includegraphics[width=0.3\linewidth]{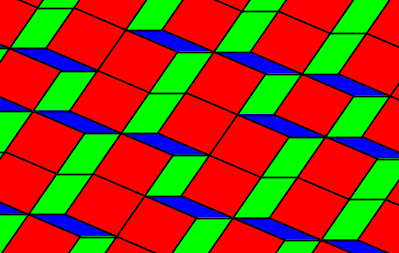}}
\caption{The three tilings  $\Ttop$, $\Tfrac$ and $\Tstep$ (from left to right).}
\label{pav-frac}
\end{figure}

On Figure~\ref{pav-frac}, it is also worth to notice that the underlying CW-complexes of $\Ttop$ and $\Tstep$ are not the same. Indeed, the valence of a vertex in $\Ttop$ is either 2 or 3, whereas the valence of a vertex in $\Tstep$ can be equal to 3, 4, 5 or 6. In that sense, the two tilings $\Tstep$ and $\Ttop$ are really different.

We have chosen to present our results on a specific substitution rather than in a general form
because it makes presentation clearer and it avoids many ``artificial'' technicalities.
Moreover, we do not know what a general answer to the above question may look like.
However, we give some insight about this general question in Section~\ref{conclu}.

\subsection{Comparison of some different notions of substitutions}

The word ``substitution'' is used in many different ways in the literature.
The list below reviews several such notions,
going from the most geometrical one (IFS) to the most combinatorial one (topological substitutions).
Indeed, as observed by Peyri\`ere \cite{Pey.86}, having a combinatorial description of substitutive tiling turns out to be very useful in many situations.
This list is not exhaustive, it only contains the notions of substitutions that we explicitly use in this article.
See~\cite{Fra08} for another survey about geometrical substitutions.

\paragraph{One-dimensional symbolic substitutions}
These substitutions are used to generated infinite one-dimensional words
which are studied mostly for their word-theoretical and dynamical properties.
An example is the Tribonacci substitution $1 \mapsto 12, 2 \mapsto 13, 3 \mapsto 1$
defined in Section~\ref{sec:subst}.
See~\cite{Pyt.02} for a classical reference.
This is the only notion of the present list which is only symbolic (not geometrical).

\paragraph{Self-affine substitutions (iterated function systems)}
Also known as \emph{substitution Delone sets}~\cite{LW03},
this notion is a particular class of iterated functions systems,
where it is required that the geometrical objects defined by the IFS
are compact sets which are the closure of their interior,
in such a way that tilings can be defined.
See Proposition~\ref{prop:rauzy} for an example of such a definition for the Tribonacci fractal.

\paragraph{Dual (or ``generalized'') substitutions}
These substitutions, introduced in~\cite{Arn.Ito.01}
can be seen as a discrete version of self-affine substitutions.
Instead of defining fractal tilings in a purely geometrical way (like with IFS),
these substitutions act on unions of faces of unit cubes located at integer coordinates.
We define the associated fractal sets and tilings by iterating the dual substitution
and by taking a Hausdorff limit of the (renormalized) unions of unit cube faces.
The fact that we deal with unit cube faces allows us to
exploit some fine combinatorial and topological properties of the resulting patterns.
This provides some powerful tools in the study of substitution dynamics and Rauzy fractal topology.
Dual substitutions are usually denoted by $\mathbf E_1^*(\sigma)$,
where $\sigma$ is a one-dimensional symbolic substitution,
See~\cite{CANT,ST09} for many references and results,
and Definition~\ref{def:eos} for the particular example studied in this article.

\paragraph{Local substitution rules}
This notion has been used to tackle combinatorial questions about
substitution dynamics~\cite{IO93,ABI02,Arn.Berth.Sieg.04,BBJS15}
and have also been studied in a more general context~\cite{Fer07,JK14}.
Their aim is to get a ``more combinatorial'' version of dual substitutions.
Instead of computing explicitely the coordinates of the image of each unit cube face
(like we do for dual substitutions),
we give some \emph{local rules} (or \emph{concatenation} rules)
for ``gluing together'' the images of two adjacent faces.
The map defined in Figure~\ref{fig:omega_two},~p.~\pageref{fig:omega_two}
is an example of such a substitution
(except that it is defined over topological tiles and not unit cube faces).

\paragraph{Topological substitutions}
Introduced in \cite{Beda.Hil.10},
topological substitutions do not make any use of geometry:
the tiles are topological disks (with no Euclidean shape),
the boundaries of which have a simplicial structure (made of vertices and edges).
It is a notion less geometrically rigid than the previous ones.
They act on CW-complexes,
and the ``gluing rules'' are more abstract and combinatorial than local substitution rules. A topological substitution generates a CW-complex homeomorphic to the plane. If this complex can be geometrized as a tiling of the plane, we say that the tiling is a topological substitutive tiling. Topological substitutions allowed for instance to prove that there is no substitutive primitive tiling of the hyperbolic plane,
even though an explicit example of a non-primitive topological substitution
which generates a tiling of the hyperbolic plane is given in \cite{Beda.Hil.10}.

In order to to distinguish this notion of substitution used the present article
from the other combinatorial notions discussed in this introduction,
we use the term \emph{topological substitution} instead of \emph{combinatorial substitution} 

The examples of topological substitutions given in the present article
(Figure~\ref{tribo-noms} and Figure~\ref{fig:exother}) are interesting
because they provide new examples of topological substitutive tilings,
which can be realized as (substitutive) tilings of the plane.

\paragraph{Other related notions}
There is another notion, elaborated by Fernique and Ollinger~\cite{FO10}
(and developped in details in the specific case of Tribonacci),
which lies between local substitution rules and topological substitutions.
For these so-called \emph{combinatorial substitutions}, the Euclidean shape of the tiles is specified,
and the matching rules are stated in terms of colors associated with
some subintervals on the boundaries of the tiles and their images.
We stress that, in that case, the Euclidean geometry is used both to give the shape of the tiles
and to specify that two tiles with same shape differ with a translation of the plane.

Purely combinatorial notions of substitutions have already been defined. For instance, 
Priebe-Frank~\cite{Fra03} introduced a very natural notion of (labelled) graph substitutions.
In the case of a substitutive tiling, this graph substitution has to be understood as a substitution on the dual graph to the tiling.
The main issue with this formalism is that there is no a priori control on the planarity of the graph obtained by iteration of the substitution, and thus in general the limit graph obtained by iteration can not be the dual graph to any tiling of the plane. Topological substitutions of \cite{Beda.Hil.10} remedy this problem.

Topological substitutions have also some worth to be met cousins: the so-called {\em subdivision rules}, introduced by Cannon, Floyd and Parry in \cite{Can.Flo.Par.01}. The natural context where these subdivision rules have hatched is the one of conformal geometry: on one hand, they can be seen as topological models for postcritically finite rational map of the Riemann sphere \cite{Can.Flo.Par.07}, on the other hand, they are likely useful to prove Cannon's conjecture for hyperbolic groups whose Gromov boundary is the 2-dimensional sphere as suggested by the results of \cite{Can.Swe.98}.
Nevertheless, subdivision rules can be also used to produce conformal\footnote{Here {\em conformal} means that the underlying geometry is the conformal geometry and not the Euclidean one. That is to say that the group $\Gamma$ defining the tiling (see Definition~\ref{def:tiling}) is not a subgroup of isometries of $\mathbb R^2$, but a subgroup of biholomorphisms of $\mathbb C$.
} substitutive tilings of the plane, see \cite{Bow.Ste.97,Ram.13}.
Even if, by iterating both a system of subdivision rules or a substitution, one get a 2-complex homeomorphic to the plane, these processes do differ in their nature: morally, in the case of subdivision rules the 2-complex is obtained as an inverse limit whereas in the case of a substitution it is obtained as a direct limit. It is not clear at all to the authors when, given a 2-complex obtained by one of two processes, one can also recover it using the other process.

\subsection{Organization of the article}

In Section~\ref{sectiling} we quickly review some usual facts about tilings. 
In Section~\ref{sect:tritop} we recall the general definition of a topological substitution
and we define the Tribonacci topological substitution we are interested in.
In Section~\ref{sect:frac} we recall the definition of  the Tribonacci dual substitution
and its associated IFS Rauzy fractal.
The link between the IFS and the topological substitution is finally studied in Section~\ref{sec:link}.
In Section~\ref{conclu} we describe how our results can be extended to some other Rauzy fractal tilings,
and we explain that finding a suitable topological substitution
has some dynamical implications for the underlying one-dimensional substitution.

\paragraph{Acknowledgements}
We thank the referee for a very careful reading of the paper and several useful suggestions.
This work was supported by the ANR through projects LAM ANR-10-JCJC-0110, QUASICOOL ANR-12-JS02-0011 and FAN ANR-12-IS01-0002.

\section{Tilings}\label{sectiling}
\subsection{Basic definitions}
In this section we recall standard notions on tiling in $\mathbb{R}^2$.
For the references about this material we refer the reader to \cite{Robin.04,Sol97,Beda.Hil.10,GS87}.
We denote by $\Gamma$ a subgroup of transformations of $\mathbb{R}^2$:
here, $\Gamma$ will be the group of translations of $\mathbb{R}^2$. We keep $\Gamma$ in the notation, just to have in mind that some classical tilings need rotations of tiles.

A {\bf tile} is a compact subset of $\mathbb R^2$ which is the closure of its interior
(in most of the basic examples, a tile is homeomorphic to a closed ball).
We denote by $\partial{T}$ the boundary of a tile,
i.e. $\partial{T}=T\setminus \mathring{T}$.
Let $\mathcal{A}$ be a finite set of {\bf labels}.
A {\bf labelled tile} is a pair $(T,a)$ where
$T$ is a tile and $a$ an element of $\mathcal{A}$.
Two labelled tiles $(T,a)$ and $(T',a')$ are {\bf equivalent} if
$a=a'$ and there exists a translation $g\in\Gamma$
such that $T'=gT$.
An equivalence class of labelled tiles is called a {\bf prototile}:
the class of $(T,a)$ is denoted by $[T,a]$, or simply by $[T]$ when the context is sufficiently clear. We will say that $(T,a)$ belongs to the prototile $[T,a]$. In some cases, one does not need the labelling
to distinguish different prototiles: for example if we consider a
family of prototiles such that the tiles in two different prototiles
are not isometric.

\begin{definition}\label{def:tiling}
A tiling $\mathbf{X}=(\mathbb{R}^2,\Gamma,\mathcal{P},\textsf{T})$ of the plane
modeled on a set of prototiles $\mathcal{P}$, is a set $\textsf{T}$ of tiles, each belonging to a prototile in
$\mathcal{P}$, such that:
\begin{itemize}
\item $\mathbb{R}^2=\displaystyle\bigcup_{T\in\textsf{T}}T,$
\item two distinct tiles of $\textsf{T}$ have disjoint interiors.
\end{itemize}
\end{definition}

A connected finite union of (labelled) tiles is called a (labelled) {\bf patch}.
Two finite patches are
{\bf equivalent} if they have the same number $k$ of tiles and these
tiles can be indexed $T_1,\dots,T_k$ and $T'_1,\dots,T'_k$,
such that there exists $g\in
\Gamma$ with $T'_i=g T_i$ for every $i\in\{1,\dots,k\}$.
Two labelled patches are {equivalent} if moreover $T_i,T'_i$
have same labelling for all $i\in \{1,\dots,k\}$.
An equivalence class of patches is called a {\bf protopatch} and denoted $[P]$ if $P$ is one of these patches.

The {\bf support} of a patch $P$, denoted by $\supp(P)$,
is the subset of $\mathbb{R}^2$ which
consists of points belonging to a tile of $P$.
A {\bf subpatch} of a patch $P$ is a patch which is a subset of the patch $P$.

Let $\mathbf{X}=(\mathbb{R}^2,\Gamma,\mathcal{P},\textsf{T})$ be a tiling, and
let $A$ be a subset of $\mathbb{R}^2$.
A patch $P$ {\bf occurs in $A$} if there exists
$g\in\Gamma$ such that for any tile $ T\in P$,
$gT$ is a tile of $\textsf{T}$ which is contained in $A$:
\[
gT \in \textsf{T} \ \text{ and } \ \supp(gT) \subseteq A.
\]
We note that any patch in the protopatch $[P]$ defined by $P$ occurs in $A$.
We say that the protopatch $[P]$ {\bf occurs} in $A$.
The {\bf language} of $\mathbf{X}$, denoted $\Lambda_\mathbf{X}$, is the set of protopatches of $\mathbf{X}$.

When all tiles of $\mathcal{P}$ are euclidean polygons, the
tiling is called a {\bf polygonal tiling}.

\subsection{Delone set defined by a tiling of the plane}
A Delone set in $\mathbb R^2$ is a set $\mathcal D$ of points such that there exists $r,R>0$ such that
every euclidean ball of radius $r$ contains at most one point of $\mathcal D$ and every euclidean ball of radius $R$ contains at least one point of $\mathcal D$. 

When a tiling of the plane $\mathbf{X}=(\mathbb{R}^2,\Gamma,\mathcal{P},\textsf{T})$ is modeled on a finite set of prototiles $\mathcal P$,
there is a standard way (among others) to derive a Delone set $\mathcal D$ from a tiling of the plane $\mathbf{X}=(\mathbb{R}^2,\Gamma,\mathcal{P},\textsf{T})$. We first choose a point in the interior of each prototile. This choice gives us a point $x_T$ in each tile $T$ of $\mathbf X$: $\mathcal D=\{x_t, T\in \textsf{T}\}$ is a Delone set.


\subsection{The $2$-complex defined by a tiling of the plane}
\label{section:complex}
Let $X$ be a $2$-dimensional CW-complex
(see, for instance, \cite{Hat.02} for basic facts about CW-complexes).
The 0-cells will be called vertices, the 1-cells edges
and the 2-cells faces.
The subcomplex
of $X$ which consists of cells of dimension at most $k\in\{0,1,2\}$ is denoted by $X^k$
(in particular $X^2=X$).
We denote by $|X^k|$ the number of $k$-cells in $X^k$. 

Let $\mathbf{X}=(\mathbb{R}^2,\Gamma,\mathcal{P},\textsf{T})$ be a tiling of the plane.
We suppose that the tiles are homeomorphic a a closed disc $\mathbb{D}^2$.
This tiling defines naturally a $2$-complex
$X$ in the following way.
The set $X^0$ of vertices of $X$ is the set of points in $\mathbb{R}^2$ which belong to
(at least) three tiles of $\mathcal{T}$. Each connected component of the set
$\bigcup_{T\in\textsf{T}}\partial{T}\setminus X^0$ is an open arc.
Any closed edge of $X$ is the closure of one of these arcs. 

Such an edge $e$ is glued to the endpoints $x,y\in X^0$ of the arc.
The set of faces of $X$ is the set of tiles of $\mathbf{X}$.
We remark that the boundary of a tile is a subcomplex of $X^1$ homeomorphic to the circle $\mathbb{S}^1$: this gives the gluing of the corresponding face on the 1-skeleton. 

Let $Y$ be a $2$-dimensional CW-complex homeomorphic to the plane $\mathbb{R}^2$.
A polygonal tiling $\mathbf{X}$ is a {\bf geometric realization} of $Y$ if the 2-complex
$X$ defined by $\mathbf{X}$ is isomorphic (as CW-complex) to $Y$.
In that case, each face of the complex $Y$ can be naturally labelled by the corresponding prototile of the tiling $\mathbf{X}$.
\section{The Tribonacci topological substitution}\label{sect:tritop}

Before giving the definition of the Tribonacci topological substitution in section~\ref{sect:deftoptribo}, we first recall some facts about (2-dimensional) topological substitutions in section~\ref{sect:toposub}.
These two sections can be read in parallel: along section~\ref{sect:toposub}, we illustrate the notions with examples referring to section~\ref{sect:deftoptribo}.

\subsection{Topological substitutions}
\label{sect:toposub}

The mathematical content of this section is essentially contained in\cite{Beda.Hil.10}: we include it here for completeness.
The vocabulary we will use in the present setting is often common to the one of tilings: 
the context is in general sufficient to prevent any ambiguity.

\subsubsection{General definition}\label{subset:top}

A {\bf topological $k$-gon} ($k\geq 3)$ is a $2$-dimensional CW-complex made of one face, 
$k$ edges and $k$ vertices, which is homeomorphic to a closed disc $\mathbb{D}^2$, 
and such that the $1$-skeleton is the boundary $\mathbb{S}^1$ of the closed disc. 
A {\bf topological polygon} is a topological $k$-gon for some $k\geq 3$.

We consider a finite set $\T=\{T_1,\dots,T_d\}$ of topological polygons.
The elements of $\mathcal{T}$ are called {\bf prototiles}, and $\T$ is called the {\bf set of prototiles}.
If $T_i$ is a $n_i$-gon,
we denote by $E_i=\{e_{1,i},\dots,e_{n_i,i}\}$ the set of edges of $T_i$.
In practice, we will need later to consider these $e_{n_k,i}$ as oriented edges: we first fix an orientation on the boundary of $T_i$, and equip the $e_{n_k,i}$ with the induced orientation. 
We set $E_i^{-1}=\{e_{1,i}^{-1},\dots,e_{n_i,i}^{-1}\}$ and $E_i^{\pm}=E_i\cup E_i^{-1}$,
where $e^{-1}$ denotes the edge $e$ equipped with the reverse orientation.

A {\bf patch} $P$ {\bf modeled} on $\T$ is a $2$-dimensional 
CW-complex homeomorphic to the closed disc 
$\mathbb{D}^2$ such that 
for each closed face $f$ of $P$, there exists a prototile $T_i\in\T$ and a homeomorphism 
$\tau_{f}:f\rightarrow T_i$ which respects the cellular structure. 
Then $T_i=\tau_f(f)$ is called the {\bf type} of the face $f$, and denoted by $\type(f)$.
The type of an edge $e$ of $T_i$, denoted by $\type(e)$, is $\tau(e)$.
An edge $e$ of $P$ is called a {\bf boundary edge} if it is contained
in the boundary $\mathbb{S}^1$ of the disc $\mathbb{D}^2\cong P$. Such a boundary
edge is contained in exactly one closed face of $P$.
An edge $e$ of $P$ which is not a boundary edge
is called an {\bf interior edge}.
An interior edge is contained in exactly
two closed faces of $P$.
In the following definition, and for the rest of this article, 
the symbol $\sqcup$ stands for the disjoint union.

\begin{definition}
A {\bf topological pre-substitution} is a triplet 
$(\T,\sigma(\T),\sigma)$ where:
\begin{enumerate}
	\item $\T=\{T_1,\dots,T_d\}$ is a set of prototiles, 
	\item $\sigma(\T)=\{\sigma(T_1),\dots,\sigma(T_d)\}$ 
	is a set of patches modeled on $\T$,
	\item 
	$\displaystyle \sigma: \bigsqcup_{i\in\{1,\dots,d\}} T_i 
	\rightarrow \bigsqcup_{i\in\{1,\dots,d\}} \sigma(T_i)$ 
	is a homeomorphism,
	which restricts to homeomorphisms 
	$T_i\rightarrow\sigma(T_i)$,
	such that the image of a vertex of $T_i$ is a vertex of 
        the boundary of $\sigma(T_i)$.
\end{enumerate}
\end{definition}

\begin{example}
In Figure \ref{tribo-noms}, we show one example of a pre-topological substitution defined on $3$ prototiles. This is the Tribonacci topological pre-substitution.
\end{example}

\paragraph{Compatible topological pre-substitution}

Let $\mathcal{T}=\{T_1,\dots, T_d\}$ be the set
of prototiles of $\sigma$, and let 
$E_i^{\pm}=E_i\cup E_i^{-1}$ be the set of oriented edges of $T_i$ ($i\in\{1,\dots,d\}$). We denote by $E^{\pm}$ the set of all oriented edges: $E^{\pm}=\sqcup_{i}E_i^{\pm}$ .

A pair $(e,e')\in E^\pm\times E^\pm$ is {\bf balanced} if
$\sigma(e)$ and $\sigma(e')$ have the same length
(= the number of edges in the edge path).
The {\bf flip} is the involution of $E^\pm\times E^\pm$
defined by $(e,e')\mapsto (e',e)$, and the {\bf reversion} is the involution of $E^\pm\times E^\pm$ defined by $(e,e')\mapsto (e^{-1},e'^{-1})$.
The quotient of $E^\pm\times E^\pm$ obtained by identifying 
a pair and its image by the flip and also a pair and its image by the reversion is denoted by $E_2$.
We denote by $[e,e']$ the image of a pair
$(e,e')\in E^\pm\times E^\pm$ in $E_2$.
Since the flip and the reversion preserve balanced pairs,
the notion of ``being balance'' is well defined
for elements of $E_2$.
The subset of $E_2$ which consists of balanced
elements is called the {\bf set of balanced pairs}, and
denoted by $\mathcal{B}$.
Let $[e,e']\in \mathcal{B}$ a balanced pair.
In other words, $\sigma(e)$ and $\sigma(e')$
are paths of edges which have same length say $p\geq 1$: 
$\sigma(e)=e_1\dots e_p$, $\sigma(e')=e'_1\dots e'_p$. Let $\varepsilon_i=\type(e_i)$ and let $\varepsilon'_i=\type(e'_i)$: $\varepsilon_i, \varepsilon'_i\in E^\pm$ for $i=1\dots p$. 
Then the $[\varepsilon_i,\varepsilon'_i]$ are called the {\bf descendants} of $[e,e']$.

\begin{example}
For the Tribonacci topological pre-substitution, consider for example the edges $e=B_{45}, e'=C_{76}$. By Figure \ref{tribo-noms}, we have 
$\sigma(e)=C_{34}C_{45}C_{56}, \sigma(e')=C_{10}C_{09}C_{98}$. Thus $[e,e']$ is a balanced pair. The descendants of this pair are $[C_{34},C_{10}], [C_{45},C_{09}], [C_{56},C_{98}]$.
\end{example}

Now, we consider a patch $P$ modeled on $\mathcal{T}$. 
An interior edge $e$ of $P$ defines an 
element $[\varepsilon,\varepsilon']$ of $E_2$. 
Indeed, let $f$ and $f'$ be the two faces adjacent to $e$ in $P$.
We denote by $\varepsilon=\tau_f(e)$ the edge of $\type(f)$
corresponding to $e$, and by $\varepsilon'=\tau_{f'}(e)$
the edge of $\type(f')$ corresponding to $e$.
The edge $e$ is said to be {\bf balanced} 
if $[\varepsilon,\varepsilon']$ is balanced.

We define, by induction on $p\in\mathbb{N}$, 
the notion of a {\bf $p$-compatible} topological pre-substitution $\sigma$.
To any $p$-compatible topological pre-substitution $\sigma$
we associate a new pre-substitution which will be denoted by
$\sigma^p$.

\begin{definition}\label{defpcomp}
\begin{enumerate}[label =\alph*)]
\item Any pre-substitution $(\T,\sigma(\T),\sigma)$ 
is {\bf $1$-compatible}. We set $\sigma^1=\sigma$.

\item A pre-substitution $(\T,\sigma(\T),\sigma)$ is said 
to be {\bf $p$-compatible} ($p\geq 2$) if:
\begin{enumerate}[label =\arabic*.]
\item $(\T,\sigma(\T),\sigma)$ is $(p-1)$-compatible 
\item for all $i\in\{1,\dots,d\}$, 
every interior edge $e$ of $\sigma^{p-1}(T_i)$
is balanced.
\end{enumerate}

\item We suppose now that $(\T,\sigma(\T),\sigma)$ is a 
$p$-compatible pre-substitution ($p\geq 2$). 
Then we define $\sigma^p(T_i)$ ($i\in\{1,\dots,d\}$) as
the patch obtained in the following way: 

We consider the collection of patches $\sigma(\type(f))$ for each
face $f$ of $\sigma^{p-1}(T_i)$. 
Then, if $f$ and $f'$ are two faces of $\sigma^{p-1}(T_i)$ adjacent 
along some edge $e$, we glue, edge to edge, 
$\sigma(\type(f))$ and $\sigma(\type(f'))$
along $\sigma(\tau_{f}(e))$ and $\sigma(\tau_{f'}(e))$. 
This is possible since the $p$-compatibility of $\sigma$
ensures that the edge $e$ is balanced.
The resulting patch $\sigma^p(T_i)$ is defined by:

$$\sigma^p(T_i)=\left.\left(\bigsqcup_{f \text{ face of } \sigma^{p-1}(T_i)}
\sigma(\type(f)) \right) \right/ \sim$$
where $\sim$ denotes the gluing.
We define $\sigma^p(\T)$ to be the set 
$\{\sigma^p(T_1),\dots,\sigma^p(T_d)\}$.
\end{enumerate}
\end{definition}

\begin{remark}
Definition~\ref{defpcomp} is recursive. Indeed, conditions $b)$ and $c)$ should be denoted $b_p)$ and $c_p)$ since they do depend on $p$. Then the definition should be read in the following order: $a)$, $b_2)$, $c_2)$, $\dots$, $b_p)$, $c_p)$, $b_{p+1})$, $c_{p+1})$, $\dots$ 
\end{remark}

The map $\sigma$ induces a natural map on the faces of each $\sigma^{p-1}(T_i)$ which factorizes to a map
$\sigma_{i,p}:\sigma^{p-1}(T_i) \rightarrow \sigma^{p}(T_i)$
thanks to the $p$-compatibility hypothesis:

\begin{equation*}
\begin{CD}
  \displaystyle\bigsqcup_{f \text{ face of } \sigma^{p-1}(T_i)}
\type(f)  @> \sigma >> \displaystyle\bigsqcup_{f \text{ face of } \sigma^{p-1}(T_i)}
\sigma(\type(f)) \\
@V{\sim}VV        @VV{\sim}V\\
 \sigma^{p-1}(T_i) @>>\sigma_{i,p}> \sigma^{p}(T_i)
\end{CD}
\end{equation*}

We note that $\sigma_{i,p}$ is a homeomorphism which
sends vertices to vertices.
Then we define the map
$$\sigma^{p}_i: T_i \rightarrow \sigma^{p}(T_i)$$
as the composition:
$\sigma^{p}_i= \sigma_{i,p}\circ\sigma_i^{p-1}$.
This is an homemorphism which
sends vertices to vertices.
Then $\sigma^p$ is naturally defined
such that the restriction of $\sigma^p$
on $T_i$ is $\sigma_i^p$.
We remark that $\sigma^1=\sigma$.
We have thus obtained a topological pre-substitution 
$(\T,\sigma^p(\T),{\sigma}^p)$. A topological pre-substitution is {\bf compatible} if it is $p$-compatible for every integer $p$.

\paragraph{Checking compatibility}

In this subsection we give an algorithm which decides
whether a pre-substitution is compatible. 

Suppose that $\sigma$ is $p$-compatible.
We define $W_p$ as the set of elements 
$[\varepsilon,\varepsilon']\in E_2$ such that there
exists $i\in\{1,\dots, d\}$, as well as two 
faces $f$ and $f'$ in $\sigma^p(T_i)$ glued along 
an edge $e$, such that
$\tau_f(e)=\varepsilon$ and $\tau_{f'}(e)=\varepsilon'$. 
The topological pre-substitution $\sigma$ is 
$(p+1)$-compatible if and only if  $W_p$ is contained in the 
set of balanced pairs $\mathcal{B}$:
$W_p\subseteq \mathcal{B}$. Then:

\begin{itemize}
\item either $W_p\nsubseteq \mathcal{B}$: the
algorithm stops, telling us that the substitution is
not compatible,
\item or $W_p\subseteq \mathcal{B}$: then
we define $V_p=V_{p-1}\cup W_p$.
\end{itemize}
By convention we settle $V_0=\emptyset$.

Suppose that $\sigma$ is compatible.
The sequence $(V_p)_{p\in\mathbb{N}}$ is an increasing
sequence (for the inclusion) of subsets of
the finite set $E_2$. Hence
there exists some $p_0\in\mathbb{N}$ such that
$V_{p_0+1}=V_{p_0}$ (and thus 
$V_{p}=V_{p_0}$ for all $p\geq p_0$).
The algorithm stops at step $p_0+1$
(where $p_0$ is the smallest integer
such that $V_{p_0+1}=V_{p_0}$),
telling that $\sigma$ is compatible.

The {\bf heredity graph of edges} of $\sigma$,
denoted $\mathcal{E}(\sigma)$, is defined in the following 
way. The set of vertices of $\mathcal{E}(\sigma)$
is $V_{p_0}$. There is an oriented edge from
vertex $[e,e']$ to vertex $[\epsilon,\epsilon']$
if $[\epsilon,\epsilon']$ is a descendant of $[e,e']$. 

\begin{example}
For the Tribonacci topological pre-substitution, consider one again the edges $e=B_{45}, e'=C_{76}$. We have shown in a previous example that $\sigma(e)=C_{34}C_{45}C_{56}, \sigma(e')=C_{10}C_{09}C_{98}$. Thus $[B_{45}, C_{76}]$ is a vertex of the heredity graph of edges. There are three edges which start from this vertex and go to the vertices defined by the balanced pairs -- see Figure~\ref{tribo-aretes} and the proof of Lemma~\ref{lemm:stable_tribo} for more details.
\end{example}

\paragraph{Core of a topological pre-substitution}

Let $P$ be a patch modeled on $\mathcal{T}=\{T_1,\dots,T_d\}$.
The {\bf thick boundary} $\B(P)$ of $P$ 
is the closed sub-complex of $P$ consisting of 
the closed faces which contain at least one vertex 
of the boundary $\partial P$ of $P$.
The {\bf core} $\core(P)$ of $P$ is 
the closure in $P$ of the complement
of $\B(P)$: in particular, $\core(P)$ is a closed subcomplex of $P$,
see Figure~\ref{fig:thick-core}.

\begin{figure}[!ht]
\centering
\includegraphics[width=5cm]{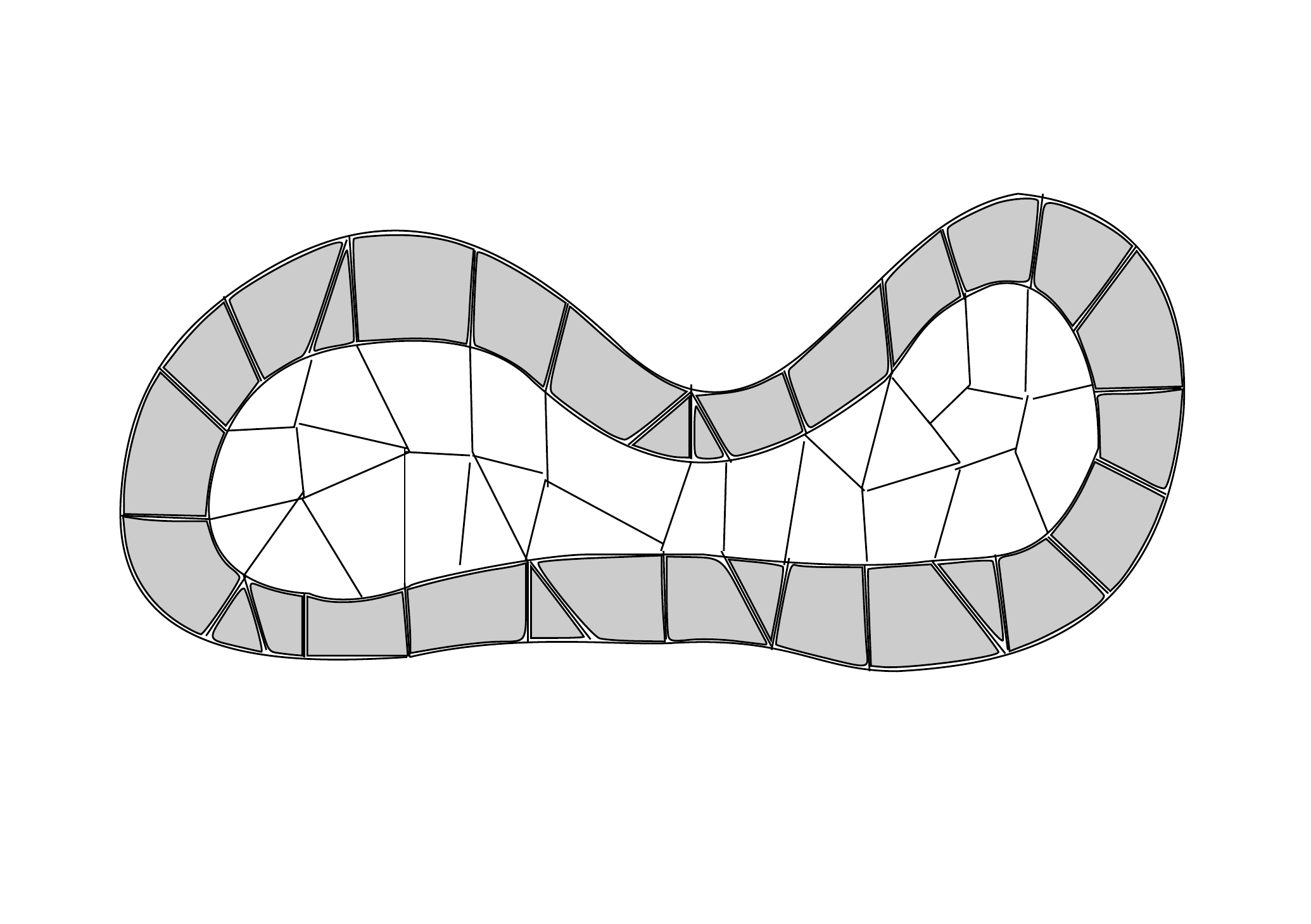}
\caption{The thick boundary is the grey subcomplex and the core is the white subcomplex.}
\label{fig:thick-core}
\end{figure}

A topological pre-substitution $(\T,\sigma(\T),\sigma)$ has the {\bf core property} if there exist $i\in\{1\dots d\}$ and $k\in\mathbb{N}$ such that the core 
of $\sigma^k(T_i)$ 
is non-empty.

\begin{example}
For the Tribonacci topological pre-substitution we show in Figure \ref{tribo-iter} that the cores of $\sigma(C), \sigma^2(C)$ are empty. But the core of $\sigma^3(C)$ is not empty.
\end{example}

\begin{definition}
A {\bf topological substitution} is a pre-substitution 
which is compatible and has the core property. 
\end{definition}

\subsubsection{Topological plane obtained by inflation}\label{sec:inflation}

Consider a tile $T\in\mathcal{T}$ such that the core of 
$\sigma(T)$ contains a face of type $T$.
Then, we can identify the tile  $T$ with a subcomplex of 
the core of $\sigma(T)$. 
By induction, $\sigma^k(T)$ is thus identified with a subcomplex 
of $\sigma^{k+1}(T)$ ($k\in\mathbb{N}$). 
We define $\sigma^\infty(T)$ as the increasing union:

$$\sigma^\infty(T)=\bigcup_{k=0}^\infty\sigma^k(T).$$
By construction, the complex $\sigma^\infty(T)$ 
is homeomorphic to $\mathbb{R}^2$.
(Indeed, denoting $\sigma^k(T)$ by $D_k$, one observes that 
$\sigma^\infty(T)$ is an increasing union of closed discs $D_k$
with $D_k$ contained in the interior $\text{Int}\; D_{k+1}$ of $D_{k+1}$, 
and with $D_{k+1}\smallsetminus \text{Int}\; D_{k}$ homemorphic to the annulus $S^1\times I$.
This allows to build an homeomorphism between $\sigma^\infty(T)$ and $\mathbb R^2$
-- see for instance \cite[exercise 3 p. 207]{Hir.94}.)
We say that such a complex is obtained {\bf by inflation from $\sigma$}. 
Moreover this complex can be labelled by the types of the topological polygons.
We notice that $\sigma$ induces an homeomorphism of $\sigma^\infty(T)$.

We denote by $\mathcal{P}_\sigma$ the {\bf set of patches in the complex $\sigma^\infty(T)$}.
We notice that $\sigma$ naturally defines a map $\mathcal{P}_\sigma\rightarrow \mathcal{P}_\sigma$,
that is still denoted by $\sigma$.
To be more precise, given a patch $P\in\mathcal{P}_\sigma$, there is some $k\in\mathbb{N}$ 
such that $P\subseteq \sigma^k(P)\subset \sigma^\infty(P)$, 
so that $\sigma(P) \subseteq \sigma^{k+1}(P)\subset \sigma^\infty(P)$: 
this patch $\sigma(P)\in\mathcal{P}_\sigma$ does not depend on the choice of $k$.

We denote by $\mathcal{T}_\sigma$ the {\bf set of tiles in the complex $\sigma^\infty(T)$}:
$\mathcal{T}_\sigma$ is a subset of $\mathcal{P}_\sigma$.
See Figure~\ref{fig:topotopo} for examples of such topological complexes generated by topological substitutions.

\begin{figure}[!ht]
\centering
\myvcenter{\includegraphics[height=5cm]{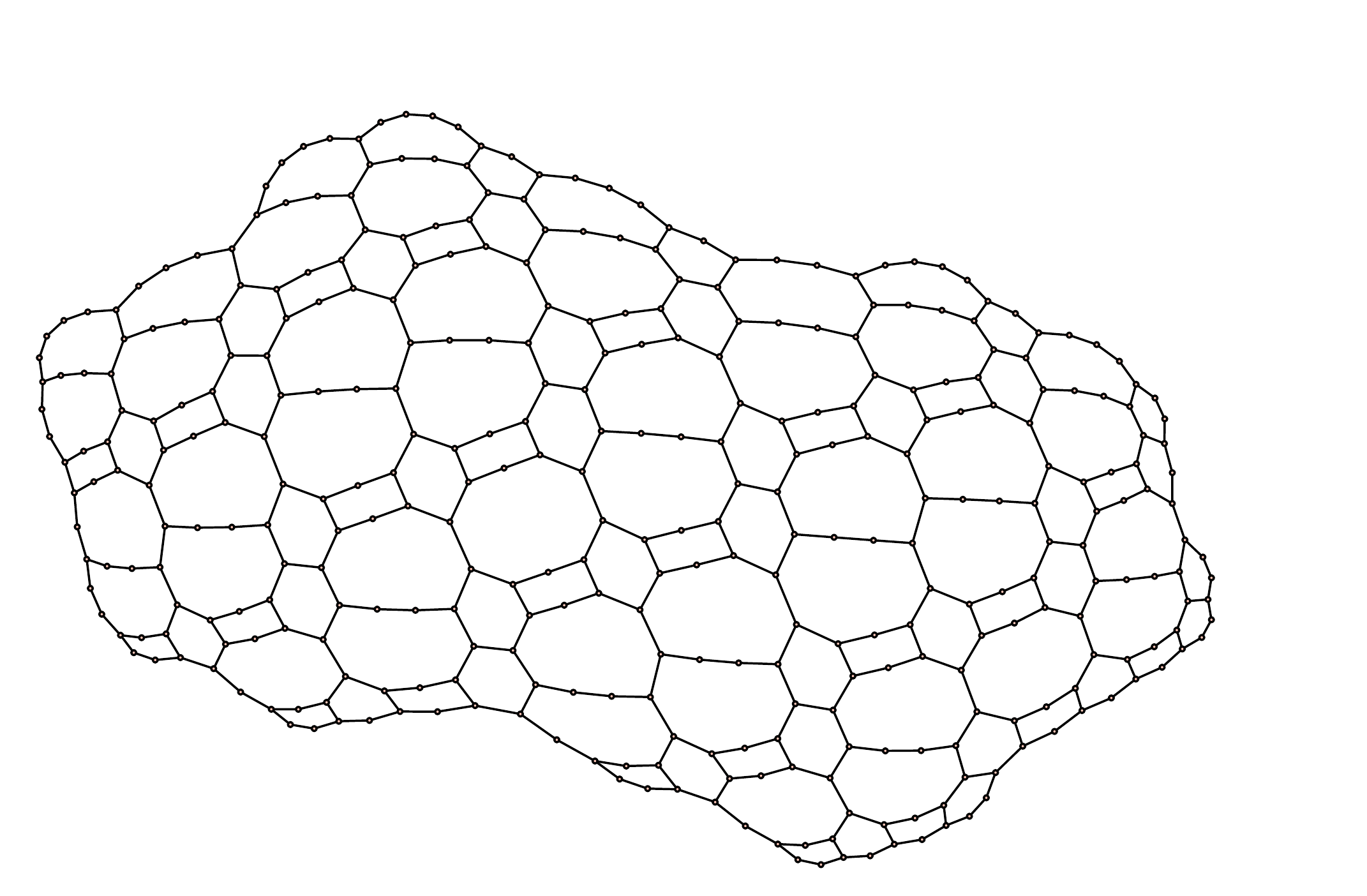}}
\hfil
\myvcenter{\includegraphics[height=5cm]{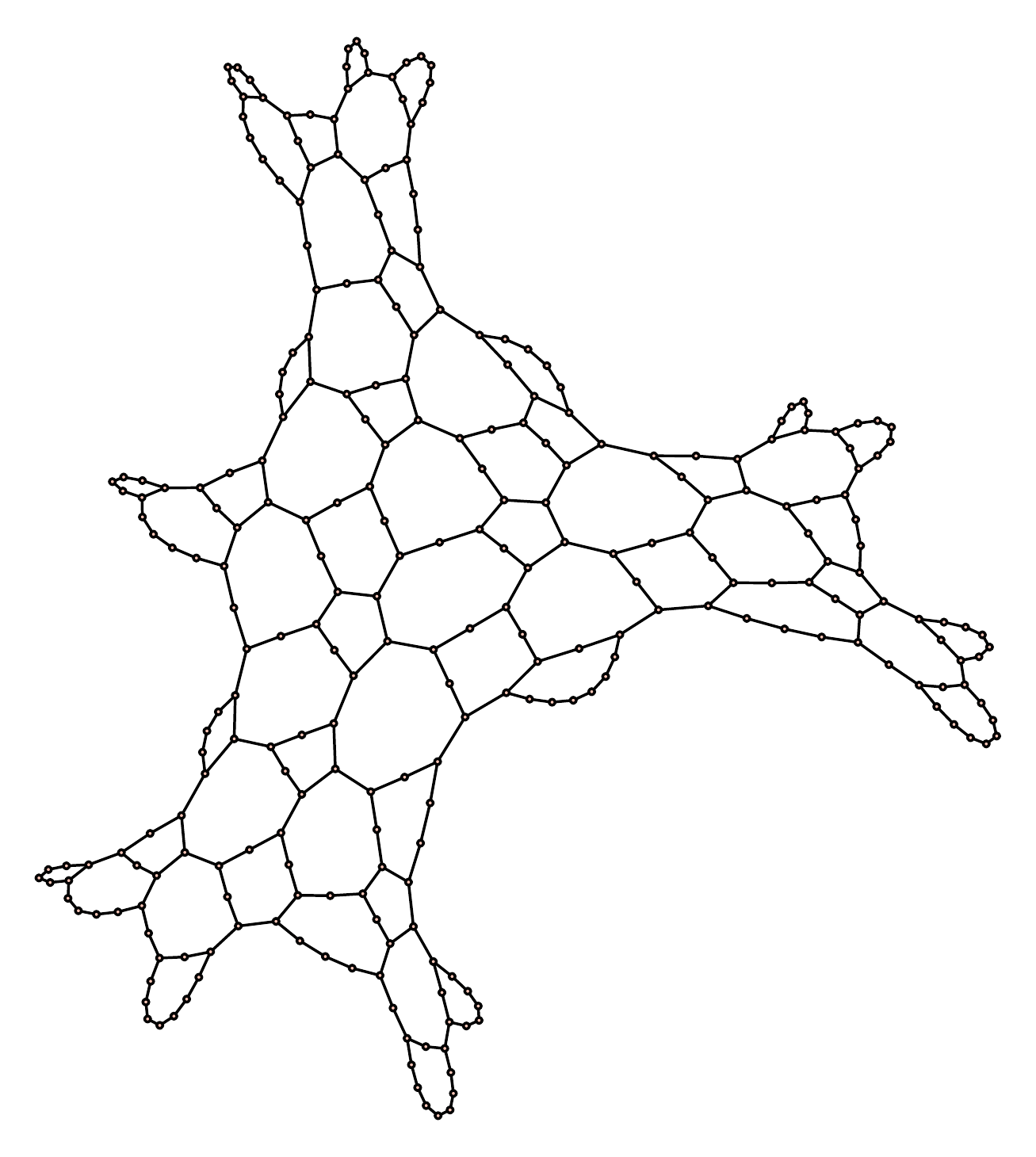}}
\caption{
The topological complexes associated with $\sigma^6(C)$ (left)
and $\tau^{10}(C)$ (right).
The definitions of $\sigma$ and $\tau$ are respectively given in
Figure~\ref{tribo-noms} (p.~\pageref{tribo-noms}) and Figure~\ref{fig:exother} (p.~\pageref{fig:exother}).
}
\label{fig:topotopo}
\end{figure}

\subsection{The Tribonacci topological substitution}
\label{sect:deftoptribo}

We first define a topological substitution $\sigma$.
Then we explain how to derive a tiling of $\mathbb{R}^2$
as a geometrical realization of the patches generated by~$\sigma$. The topological substitution is defined on Figure \ref{tribo-noms} and the first iterations on the polygon C are given in Figure \ref{tribo-iter}.

\subsubsection{Definition of the topological substitution}

We consider the Tribonacci topological pre-substitution $\sigma$ defined on Figure \ref{tribo-noms}.
There are three prototiles: two of them, $A$ and $B$, are hexagons, while the third one, $C$,
is a decagon.
The images of these prototiles (together with the labelling of the vertices and the images
of the vertices) are given on Figure \ref{tribo-noms}.
In practice, we will denote by $A_i$ the vertex $i$ of $A$, and by $A_{i(i+1)}$ the edge of $A$ joining
$A_i$ and $A_{i+1}$ (and so on for $B$ and $C$).

\begin{figure}[!ht]
\centering
\begin{tikzpicture}
[x={(-0.692820cm,-0.400000cm)}, y={(0.692820cm,-0.400000cm)}, z={(0.000000cm,0.800000cm)}]
\begin{scope}[shift={(-5,5,0)}]
\tuileC{C}{(2.5,-1.5,0)}{white}{0}{9}{8}{7}{6}{5}{4}{3}{2}{1}
\node at (0.5,-0.5,-0.75) {$\longmapsto$};
\tuileC{C}{(-1,2,0)}{white}{0}{9}{8}{7}{6}{5}{4}{3}{2}{1}
\tuileB{B}{(0,3,1)}{white}{0}{5}{4}{3}{2}{1}
\tuileA{A}{(-1,2,0)}{white}{0}{5}{4}{3}{2}{1}
\fill (1,1,0) circle (2.5pt) node[below] {\scriptsize $\sigma(2)$};
\fill (1,3,-1) circle (2.5pt) node[below] {\scriptsize $\sigma(1)$};
\fill (0,3,-1) circle (2.5pt) node[below] {\scriptsize $\sigma(0)$};
\fill (0,4,-1) circle (2.5pt) node[below] {\scriptsize $\sigma(9)$};
\fill (-1,4,-1) circle (2.5pt) node[below] {\scriptsize $\sigma(8)$};
\fill (-1,4,1) circle (2.5pt) node[above] {\scriptsize $\sigma(7)$};
\fill (-1,2,2) circle (2.5pt) node[above] {\scriptsize $\sigma(6)$};
\fill (0,2,2) circle (2.5pt) node[above] {\scriptsize $\sigma(5)$};
\fill (0,1,2) circle (2.5pt) node[above] {\scriptsize $\sigma(4)$};
\fill (1,1,2) circle (2.5pt) node[above] {\scriptsize $\sigma(3)$};
\begin{scope}[shift={(2.5,-2.5,0)}]
\fill (1,1,0) circle (2.5pt);
\fill (1,2,0) circle (2.5pt);
\fill (1,2,-1) circle (2.5pt);
\fill (1,3,-1) circle (2.5pt);
\fill (0,3,-1) circle (2.5pt);
\fill (0,3,1) circle (2.5pt);
\fill (0,2,1) circle (2.5pt);
\fill (0,2,2) circle (2.5pt);
\fill (0,1,2) circle (2.5pt);
\fill (1,1,2) circle (2.5pt);
\end{scope}
\end{scope}

\begin{scope}[shift={(0,0,-3)}]
\tuileB{B}{(2,1,0.25)}{white}{0}{5}{4}{3}{2}{1}
\node at (0.5,-0.5,-0.75) {$\longmapsto$};
\tuileC{C}{(0,1,0)}{white}{0}{9}{8}{7}{6}{5}{4}{3}{2}{1}
\fill (1,1,0) circle (2.5pt) node[below] {\scriptsize $\sigma(5)$};
\fill (1,3,-1) circle (2.5pt) node[below] {\scriptsize $\sigma(4)$};
\fill (0,3,-1) circle (2.5pt) node[below] {\scriptsize $\sigma(3)$};
\fill (0,3,1) circle (2.5pt) node[above] {\scriptsize $\sigma(2)$};
\fill (0,1,2) circle (2.5pt) node[above] {\scriptsize $\sigma(1)$};
\fill (1,1,2) circle (2.5pt) node[above] {\scriptsize $\sigma(0)$};
\begin{scope}[shift={(3,0,0)}]
\fill (1,1,0.25) circle (2.5pt);
\fill (0,1,0.25) circle (2.5pt);
\fill (0,1,2.25) circle (2.5pt);
\fill (0,0,2.25) circle (2.5pt);
\fill (1,0,2.25) circle (2.5pt);
\fill (1,0,0.25) circle (2.5pt);
\end{scope}
\end{scope}

\begin{scope}[shift={(0,0,2.5)}]
\tuileA{A}{(1,-1,0.5)}{white}{0}{5}{4}{3}{2}{1}
\node at (0.5,-0.5,-1) {$\longmapsto$};
\tuileB{B}{(-1,2,-1)}{white}{0}{5}{4}{3}{2}{1}
\fill (1,2,-1) circle (2.5pt) node[below] {\scriptsize $\sigma(4)$};
\fill (0,2,-1) circle (2.5pt) node[below] {\scriptsize $\sigma(3)$};
\fill (0,2,1) circle (2.5pt) node[above] {\scriptsize $\sigma(2)$};
\fill (0,1,1) circle (2.5pt) node[above] {\scriptsize $\sigma(1)$};
\fill (1,1,1) circle (2.5pt) node[above] {\scriptsize $\sigma(0)$};
\fill (1,1,-1) circle (2.5pt) node[below] {\scriptsize $\sigma(5)$};
\begin{scope}[shift={(2,-1,-0.5)}]
\fill (1,1,0) circle (2.5pt);
\fill (0,1,0) circle (2.5pt);
\fill (0,1,1) circle (2.5pt);
\fill (0,0,1) circle (2.5pt);
\fill (1,0,1) circle (2.5pt);
\fill (1,0,0) circle (2.5pt);
\end{scope}
\end{scope}
\end{tikzpicture}
\caption{The Tribonacci topological substitution.}
\label{tribo-noms}
\end{figure}
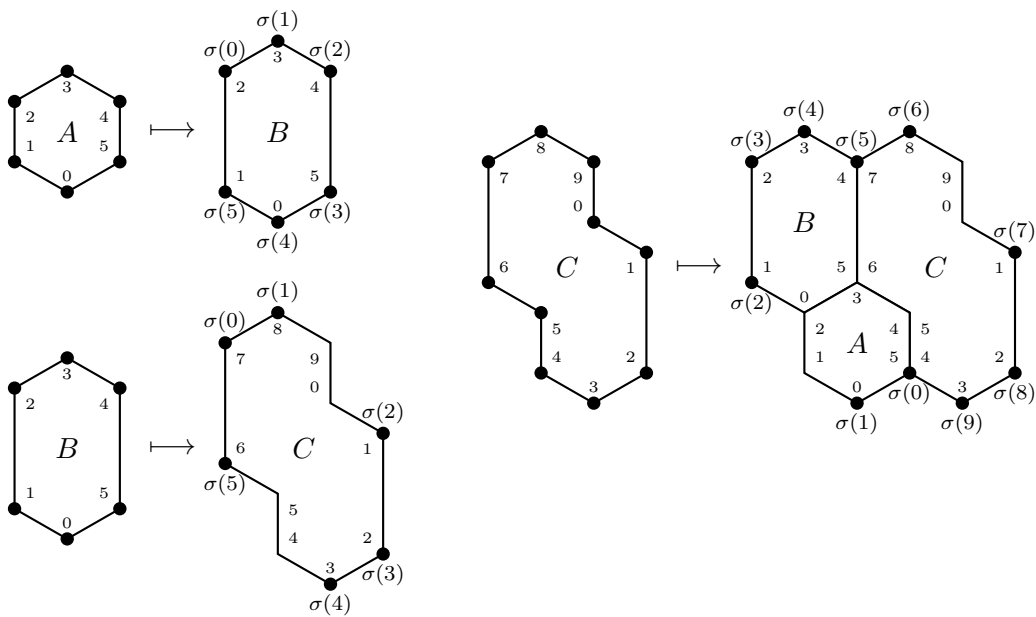

\begin{lemma}
\label{lemm:stable_tribo}
The topological pre-substitution $\sigma$ is a topological substitution.
\end{lemma}
\begin{proof}
Using the procedure described at the end of Subsection~\ref{subset:top} (in Paragraph ``Checking compatibility''),
we first check that $\sigma$ is compatible.
We start with the pair of edges that are glued in the images of $A$, $B$ and $C$. In fact, all these gluings occur in $\sigma(C)$: $[A_{45},C_{54}]$, $[A_{23},B_{05}]$, $[B_{45},C_{76}]$,  $[A_{43},C_{56}]$.

We focus now on $[B_{45},C_{76}]$. The image of the edge $B_{45}$ is the path of edges
$C_{34}C_{45}C_{56}$. The image of $C_{76}$ is $C_{10}C_{09}C_{98}$. Both have length $3$, and the gluing gives rise to the pairs of edges: 
$[C_{34},C_{10}]$, $[C_{45},C_{09}]$ and $[C_{56},C_{98}]$.

Carrying out the other pairs in  the same way, and iterating the process, we check that $\sigma$ is compatible. These computations are summed up in the heredity graph of edges, given in Figure~\ref{tribo-aretes}.

\begin{figure}[!ht]
\centering
$ \xymatrix{ 
[A_{45},C_{54}] \ar[d] & [A_{23},B_{05}] \ar[r] & [B_{45},C_{76}] \ar@/_0.5pc/[dl] \ar@/^0.5pc/[dr] \ar@/^1pc/[drr] & & \\
[B_{01},B_{43}] \ar@/_0.5pc/[dr] & [C_{56},C_{98}] \ar[d] &  & [C_{45},C_{09}] \ar[d] & [C_{34},C_{10}]  \ar@/^1pc/[dddl] \\
& [C_{78},C_{32}] \ar[d] & [C_{23},B_{32}] \ar@/_0.5pc/[dl] & [B_{34},C_{43}] \ar[l] & \\
& [C_{12},B_{21}] \ar@/_0.5pc/[dl] \ar@/^0.5pc/[dr] \ar[d] & & & \\
[B_{01},C_{98}]  \ar@/^1pc/[uur]  & [A_{12},C_{09}] \ar@/_1pc/[uurr] & [A_{01},C_{10}] \ar[r]  & [B_{23},A_{05}]  \ar@/_1pc/[ull] & \\ 
& [C_{12},C_{76}] \ar[u] \ar@/^0.5pc/[ul] \ar@/_0.5pc/[ur] & [B_{05},C_{78}] \ar[l] & [A_{43},C_{56}] \ar[l] &  \\
}$
\caption{The heredity graph of edges $\mathcal{E}(\sigma)$ of the Tribonacci substitution $\sigma$.}
\label{tribo-aretes}
\end{figure}
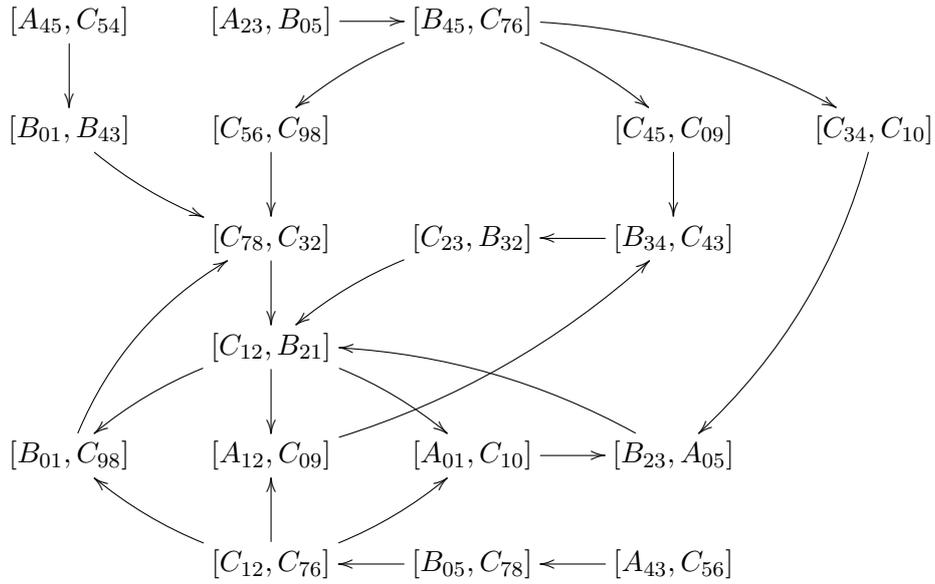

The core property is checked for $\sigma^3(C)$: we see on Figure~\ref{tribo-iter} that $\core(\sigma^3(C))\neq\emptyset$.
Hence $\sigma$ is a topological substitution.
\end{proof}

\begin{figure}[!ht]
\centering
\myvcenter{%
\begin{tikzpicture}
[x={(-0.562917cm,-0.325000cm)}, y={(0.562917cm,-0.325000cm)}, z={(0.000000cm,0.650000cm)}]
\tuileA{A}{(0,0,0)}{white}{0}{5}{4}{3}{2}{1}
\tuileB{B}{(0,0,0)}{white}{0}{5}{4}{3}{2}{1}
\tuileC{C}{(0,0,0)}{white}{0}{9}{8}{7}{6}{5}{4}{3}{2}{1}
\end{tikzpicture}}
    \hfil
\myvcenter{%
\begin{tikzpicture}
[x={(-0.562917cm,-0.325000cm)}, y={(0.562917cm,-0.325000cm)}, z={(0.000000cm,0.650000cm)}]
\tuileA{A}{(0,0,0)}{white}{0}{5}{4}{3}{2}{1}
\tuileB{B}{(0,0,0)}{white}{0}{5}{4}{3}{2}{1}
\tuileC{C}{(0,0,0)}{white}{0}{9}{8}{7}{6}{5}{4}{3}{2}{1}
\tuileB{B}{(-1,0,2)}{white}{0}{5}{4}{3}{2}{1}
\tuileC{C}{(-1,0,2)}{white}{0}{9}{8}{7}{6}{5}{4}{3}{2}{1}
\end{tikzpicture}}
    \hfil
\myvcenter{%
\begin{tikzpicture}
[x={(-0.562917cm,-0.325000cm)}, y={(0.562917cm,-0.325000cm)}, z={(0.000000cm,0.650000cm)}]
\tuileA{A}{(0,0,0)}{white}{0}{5}{4}{3}{2}{1}
\tuileB{B}{(0,0,0)}{white}{0}{5}{4}{3}{2}{1}
\tuileC{C}{(0,0,0)}{white}{0}{9}{8}{7}{6}{5}{4}{3}{2}{1}
\tuileB{B}{(-1,0,2)}{white}{0}{5}{4}{3}{2}{1}
\tuileC{C}{(-1,0,2)}{white}{0}{9}{8}{7}{6}{5}{4}{3}{2}{1}
\tuileA{A}{(0,2,-3)}{white}{0}{5}{4}{3}{2}{1}
\tuileB{B}{(0,2,-3)}{white}{0}{5}{4}{3}{2}{1}
\tuileC{C}{(0,2,-3)}{white}{0}{9}{8}{7}{6}{5}{4}{3}{2}{1}
\tuileC{C}{(-1,2,-1)}{white}{0}{9}{8}{7}{6}{5}{4}{3}{2}{1}
\end{tikzpicture}}
\caption{Iterating the Tribonacci topological substitution: $\sigma(C)$,~$\sigma^2(C)$~and~$\sigma^3(C)$.
}
\label{tribo-iter}
\end{figure}
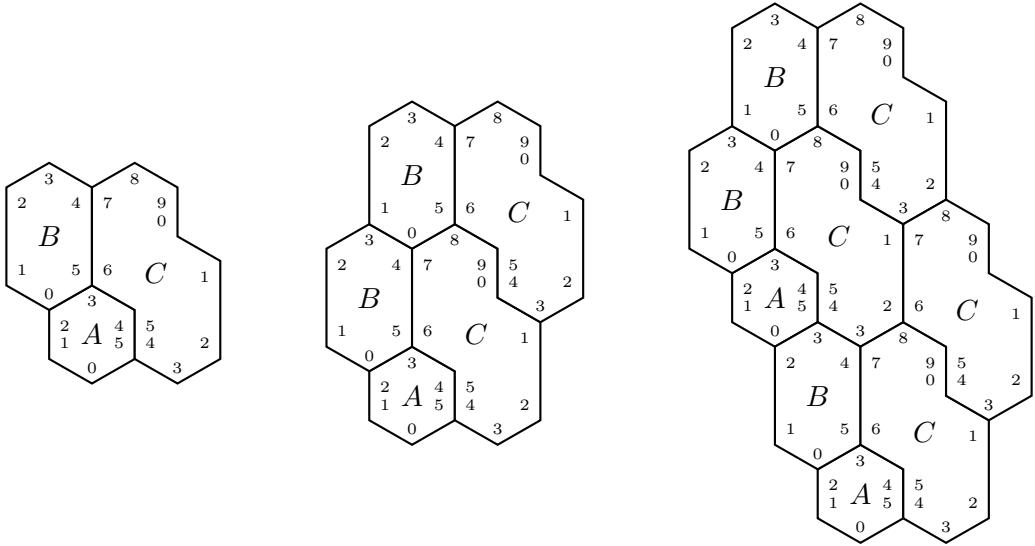

\subsubsection{Configurations at the vertices}

We denote by $V$ the set of vertices of the prototiles $A$, $B$, $C$.
The {\bf heredity graph of vertices} is an oriented
graph denoted by $\mathcal{V}(\sigma)$. 
The set of vertices of $\mathcal{V}(\sigma)$
is the set $V$. 
Let $T,T'\in\{A,B,C\}$, and let $v$ be a vertex of $T$ and $v'$ be a vertex of $T'$:
there is an oriented edge in $\mathcal{V}(\sigma)$ from $v$ to $v'$ 
if $\sigma(v)$ is a vertex of type $v'$ of a tile of type $T'$.

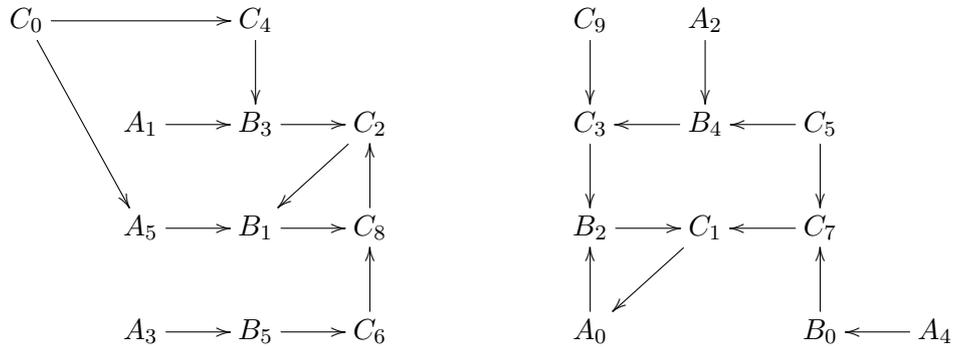
\begin{figure}[!ht]
\centering
$ \xymatrix{ 
C_{0}\ar[rr]\ar[ddr]&&C_4\ar[d]\\
&A_1\ar[r]&B_3\ar[r]&C_2\ar[dl]\\
&A_5\ar[r]&B_1\ar[r]&C_8\ar[u]&\\
&A_3\ar[r]&B_5\ar[r]&C_6\ar[u]&
}$
$ \xymatrix{ 
&C_9\ar[d]&A_2\ar[d]\\
&C_3\ar[d]&B_4\ar[l]&C_5\ar[d]\ar[l]\\
&B_2\ar[r]&C_1\ar[dl]&C_7\ar[l]\\
&A_0\ar[u]&&B_0\ar[u]&A_4\ar[l]
}$
\caption{The heredity graph of the vertices $\mathcal{V}(\sigma)$.}
\label{fig:vertex list}
\end{figure}

The graph $\mathcal{V}(\sigma)$ is given in Figure~\ref{fig:vertex list}.
It can be used to control the valences of the vertices in $\sigma^\infty(C)$
thanks to the following property proved in \cite{Beda.Hil.10}. A vertex $v\in V$ is a {\bf divided vertex} if there are
at least two oriented edges in $\mathcal{V}(\sigma)$ coming out of $v$. 
We denote by $V_{\mathcal D}$ the subset of $V$ which consists of all divided vertices.
The following properties are equivalent, see \cite{Beda.Hil.10}:
\begin{itemize}
\item The complex $\sigma^\infty(C)$ has bounded valence.
\item Every infinite oriented path in 
$\mathcal{V}(\sigma)$ crosses vertices of $V_{\mathcal D}$ only finitely many times.
\item The oriented cycles of $\mathcal{V}(\sigma)$ do not cross any vertex of $V_{\mathcal D}$.
\end{itemize}

\begin{lemma}\label{lem:val-bounded}
The valence of each vertex in $\sigma^\infty(C)$ is bounded.
\end{lemma}
\begin{proof}
We consider Figure~\ref{fig:vertex list}. Remark that $V_{\mathcal D}=\{C_0,C_5\}$. 
Moreover, neither $C_0$ nor $C_5$ is crossed by an oriented cycle of $\mathcal{V}(\sigma)$.
Hence the valence of the vertices of $\sigma^\infty(C)$ is bounded by the previous property.
\end{proof}

Now we introduce another graph, which is called
the {\bf configuration graph of vertices} and is denoted by 
$\mathcal{C}(\sigma)$.
We consider the equivalence relation on $k$-tuples of elements of $V$ ($k\in\mathbb{N}$) generated by:
$$(x_1,\dots,x_{k-1},x_k)\sim(x_k,x_1,\dots,x_{k-1}) \;\text{ and }\; (x_1,x_2,\dots,x_k)\sim(x_k,\dots,x_2,x_1).$$
Let $[x_1,\dots,x_k]$ denote the equivalence class of $(x_1,\dots,x_k)$. Let $K$ be the maximal valence of a vertex in $\sigma^\infty(C)$. Let $W$ be the set of equivalence classes of $k$-tuples with $2\leq k\leq K$.
A vertex $v$ in the interior of a patch $\sigma^n(C)$ $(n\geq 1)$ 
defines an element $[x_1,\dots,x_k]\in W$ 
(where $k$ is the valence of $v$). 
Indeed, the faces adjacent to $v$ are cyclically ordered, 
and $x_i$ is the type of the vertex of the $i$-th face which is glued on $v$. 

We define the oriented graph $\mathcal{C}(\sigma)$ as follows.
Let $W_0$ be the subset of $W$ defined by the vertices occuring in the interior of 
some $\sigma^n(C)$ for $n\geq 1$. 
An element of $W_0$ is called a {\bf vertex configuration}.
The set of vertices of $\mathcal{C}(\sigma)$ is $W_0$. 
For any $s\in W_0$, we choose some $T\in \mathcal{T}, n\geq 1$ and $v$ a vertex in the interior of $\sigma^n(T)$ which defines $s$. 
Let $s'$ the element of $W_0$ defined by $\sigma(v)$.
There is an oriented edge in $\mathcal{C}(\sigma)$ from $s$ to $s'$. 
We notice that this construction does not depend of the choice of $T$ and $n$.

In practice, to build the graph $\mathcal{C}(\sigma)$, we first remark that a vertex $v$ in the interior
of some $\sigma^n(C)$ ($n\geq1$) is 
either the image of a vertex in the interior of $\sigma^{n-1}(C)$,
or is in the interior of a path of edges which is the image of an interior edge of $\sigma^{n-1}(C)$.
Thus we first make the list of vertex configurations for:
\begin{itemize}
  \item vertices in the interior of the image of a tile: we get $[C_5,A_4]$ et $[C_6,A_3,B_5]$;
  \item vertices in the interior of the image of an interior edge.
  These ones can be derived from the vertices of $\mathcal{E}(\sigma)$
  with a least 2 outing edges. There are 3 such vertices in of $\mathcal{E}(\sigma)$:
  \begin{itemize}
    \item $[B_{45},C_{76}]$ which gives rise to vertex configurations $[C_{4},C_{0}]$ and $[C_{5},C_{9}]$,
    \item $[C_{12},B_{21}]$ and $[C_{12},C_{76}]$ 
    which gives rise to vertex configurations $[C_{4},C_{0}]$ and $[C_{5},C_{9}]$.
  \end{itemize}
\end{itemize}
Then we iteratively compute the vertex configurations obtained as the image under $\sigma^n$ of the vertex 
configurations in 
$\{[C_5,A_4], [C_6,A_3,B_5], [C_{4},C_{0}], [C_{5},C_{9}], C_{4},C_{0}], [C_{5},C_{9}]
\}$.
The graph $\mathcal{C}(\sigma)$ is represented on Figure~\ref{fig:configurations}.

\begin{figure}[ht]
\[
\xymatrix{ 
[C_5,A_4]
\ar[r] & 
[C_7,B_4,B_0]
\ar[r] & 
[C_1,C_3,C_7]
\ar[r] & 
[A_0,B_2,C_1]  
\ar@(ul,ur) \\
[C_5,C_9] 
\ar[r] & 
[C_7,B_4,C_3] 
\ar[r] & 
[C_1,C_3,B_2] 
\ar[ur] & 
[A_1,C_0] 
\ar[dl]  \\
[A_2,C_9,B_0] 
\ar[ur] & 
[B_1,B_3,C_2] 
\ar[dl] & 
[A_5,C_4,B_3] 
\ar[l] & 
[C_0,C_4] 
\ar[l]  \\
[B_1,C_8,C_2] 
\ar@(dl,dr) & 
[C_2,C_6,C_8] 
\ar[l] & 
[C_8,B_5,C_6] 
\ar[l] & 
[C_6,A_3,B_5] 
\ar[l]
}
\]
\caption{The configuration graph of the vertices $\mathcal C(\sigma)$.}
\label{fig:configurations}
\end{figure}
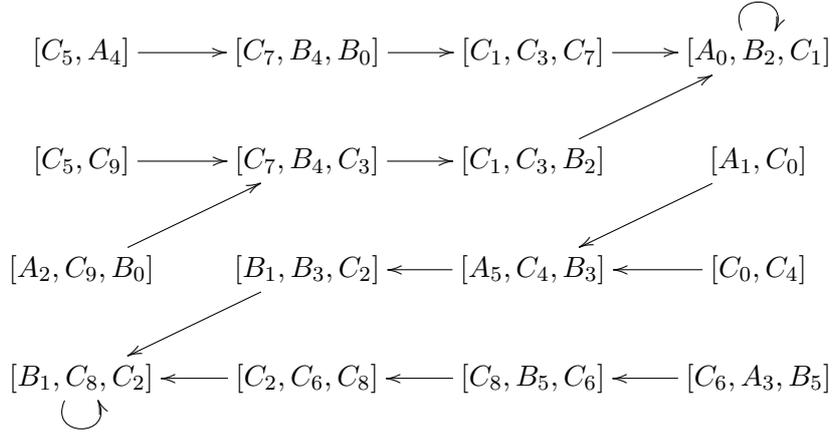

\subsubsection{A geometric realization of $\sigma^\infty(C)$}

\begin{proposition}
\label{prop:topotiling}
The complex $\sigma^\infty(C)$ can be realized as a tiling of $\mathbb{R}^2$. This tiling is denoted by $\Ttop$.
\end{proposition}
\begin{proof}
We first recall that a 2-dimensional CW-complex with hexagonal faces
such that each edge belongs to 2 faces and each vertex belongs to 3 faces
is isomorphic to the 2-dimensional CW-complex $X_{\text{hex}}$ defined by 
the tiling of $\mathbb{R}^2$ by regular euclidean hexagons.

Let $C_{\text{hex}}$ the patch made of two hexagonal faces
obtained by dividing $C$ along an edge between vertices
$C_0$ and $C_5$.
Given a patch $P$ made of tiles of types $A$, $B$ and $C$,
we build a patch $P_{\text{hex}}$ made of hexagonal faces 
by replacing the faces of type $C$ by $C_{\text{hex}}$.

We claim that $(\sigma^\infty(C))_{\text{hex}}=X_{\text{hex}}$.
Indeed, the vertex configurations of $\sigma^\infty(C)$ are given by
the vertices of the graph $\mathcal{C}$ (see Figure~\ref{fig:configurations}),
and for every of them, we check that $P=P_{\text{hex}}$ for the corresponding 
patch $P$.

Since $(\sigma^\infty(C))_{\text{hex}}=X_{\text{hex}}$, it is straightforward to derive
that $\sigma^\infty(C)$ can be geometrically realized as a tiling of $\mathbb{R}^2$,
where $A$ and $B$ are realized by regular hexagons, and $C$ by a decagon
obtained by gluing two regular hexagons along an edge.
\end{proof}

\begin{remark}\label{rem:valence-sommet}
In particular, we can now precise Lemma \ref{lem:val-bounded}: The valence of a vertex in $\sigma^\infty(C)$ is either $2$ or $3$.
\end{remark}

\subsection{The pointed topological substitution $\hat\sigma$}

Let $\PP$ be the set of pairs $(P,T)$ where $P$ is a patch in $\sigma^\infty(C)$ and $T$ is a tile in $\sigma^\infty(C)$.
We stress that $T$ need not lie in $P$.
Such a pair $(P,T)\in\PP$ is a {\bf pointed patch}, and $T$ is the {\bf base tile} of the pointed patch $(P,T)$. 

For our purposes, we need to consider a kind of ``pointed'' version $\hat\sigma$ of $\sigma$:
this will be a map
$$\begin{array}{rrcl}
\hat\sigma: & \PP & \rightarrow & \PP \\
 & (P,T) & \mapsto & (\sigma(P), \bt(T)).
\end{array}$$
To completely define $\hat\sigma$, it remains now to define the map $\bt$.
Let $T$ be a tile of $\sigma^\infty(C)$.
Then $\sigma(T)$ is a patch of $\sigma^\infty(C)$.
If $\sigma(T)$ is a tile (this is the case when $T$ is of type $A$ or $B$), 
then we simply set $\bt(T)=\sigma(T)$.
If $T$ is of type $C$, then we define $\bt(T)$ as the tile of type $C$ in $\sigma(T)$.

\begin{remark}
In the example this last choice is a bit arbitrary: in particular, we could have considered other versions of 
$\hat\sigma$ where, for a tile $T$ of type $C$, $\bt(T)$ would be the tile of type $A$ (or $B$) in $\sigma(T)$.
In that case, we would have to modify accordingly the definition of the position map $\omega_0$ of Section~\ref{sec:omega_0}.
\end{remark}


\section{The Tribonacci dual substitution and its fractal tilings}
\label{sect:frac}

We recall that a {\em substitution} is a morphism of the free monoid (of rank $d$).
There is a general construction introduced in \cite{IO93} and generalized by \cite{Arn.Ito.01} that associates to a substitution a so-called {\em dual substitution}.
To avoid to reproduce the general formalism of \cite{Arn.Ito.01}, we focus in sections~\ref{sect:dual_sub} and \ref{sec:subst}
on the {\em Tribonacci dual substitution} associated to the the Tribonacci substitution $1\mapsto 12$, $2\mapsto 13$, $3\mapsto 1$. In particular, $d=3$.

Dual substitutions act on facets of $\mathbb R^3$ (for the cellular decomposition of $\mathbb R^3$ given by $\mathbb Z^3$ translated copies of the unit cube) : the image of a facet is a set of facets. 
Hence, when iterating a dual substitution, one gets bigger sets of facets. A priori, there can be some overlaps: in that case, facets have to be count with some multiplicity, which leads to the notion of multiset of facets. 

The data that encode a multiset of facet is given by: the type (an element of the set $\{1,2,3\}$), the position (an element of $\mathbb Z^3$) and the signed multiplicity (an element of $\mathbb Z$) of each facet.
This setting is formalized in section~\ref{sec:multiset}, leading to the equivalent notions of {\em weight functions} and {\em multisets of facets}. We detail the equivalence of the two points of view, so that later in section~\ref{sec:link} we will swap from one point of view to the other one according to the context.

\subsection{Multisets of facets}\label{sec:multiset}

We denote by $(\mathbf e_1, \mathbf e_2, \mathbf e_3)$ the canonical basis of $\mathbb R^3$. In this article, this basis will be represented  as follows in the different figures.
\begin{center}
\begin{tikzpicture}[x={(-0.216506cm,-0.125000cm)}, y={(0.216506cm,-0.125000cm)}, z={(0.000000cm,0.250000cm)}]
\draw[thick, ->] (0,0,0)--(1,0,0);
\draw[thick, ->] (0,0,0)--(0,1,0);
\draw[thick, ->] (0,0,0)--(0,0,1);
\draw (1.7,-0.3,0) node{$\mathbf e_1$};
\draw (-0.3,1.7,0) node{$\mathbf e_2$};
\draw (0,0,1.6) node{$\mathbf e_3$};
\end{tikzpicture}
\end{center}
Let $\mathbf x \in \mathbb{Z}^3$ and let $i\in\{1,2,3\}$.
The \tdef{facet} $[\mathbf x,i]^*$ of \tdef{vector} $\mathbf x$ and \tdef{type} $i$ is a subset of $\mathbb R^3$ defined by:
\begin{align*}
\lbrack \mathbf x, 1 \rbrack^* & = \{\mathbf x + \lambda \mathbf e_2 + \mu \mathbf e_3 : \lambda,\mu \in [0,1] \} =
    \myvcenter{\begin{tikzpicture}
    [x={(-0.216506cm,-0.125000cm)}, y={(0.216506cm,-0.125000cm)}, z={(0.000000cm,0.250000cm)}]
    \definecolor{facecolor}{rgb}{0.8,0.8,0.8}
    \fill[thick, fill=facecolor, draw=black, shift={(0,0,0)}]
    (0, 0, 0) -- (0, 1, 0) -- (0, 1, 1) -- (0, 0, 1) -- cycle;
    \node[circle,fill=black,draw=black,minimum size=1.2mm,inner sep=0pt] at (0,0,0) {};
    \end{tikzpicture}} \\
\lbrack \mathbf x, 2 \rbrack^* & = \{\mathbf x + \lambda \mathbf e_1 + \mu \mathbf e_3 : \lambda,\mu \in [0,1] \} =
    \myvcenter{\begin{tikzpicture}
    [x={(-0.216506cm,-0.125000cm)}, y={(0.216506cm,-0.125000cm)}, z={(0.000000cm,0.250000cm)}]
    \definecolor{facecolor}{rgb}{0.8,0.8,0.8}
    \fill[thick, fill=facecolor, draw=black, shift={(0,0,0)}]
    (0, 0, 0) -- (0, 0, 1) -- (1, 0, 1) -- (1, 0, 0) -- cycle;
    \node[circle,fill=black,draw=black,minimum size=1.2mm,inner sep=0pt] at (0,0,0) {};
    \end{tikzpicture}} \\
\lbrack \mathbf x, 3 \rbrack^* & = \{\mathbf x + \lambda \mathbf e_1 + \mu \mathbf e_2 : \lambda,\mu \in [0,1] \} =
    \myvcenter{\begin{tikzpicture}
    [x={(-0.216506cm,-0.125000cm)}, y={(0.216506cm,-0.125000cm)}, z={(0.000000cm,0.250000cm)}]
    \definecolor{facecolor}{rgb}{0.8,0.8,0.8}
    \fill[thick, fill=facecolor, draw=black, shift={(0,0,0)}]
    (0, 0, 0) -- (1, 0, 0) -- (1, 1, 0) -- (0, 1, 0) -- cycle;
    \node[circle,fill=black,draw=black,minimum size=1.2mm,inner sep=0pt] at (0,0,0) {};
    \end{tikzpicture}}.
\end{align*}
On each of the previous pictures, the symbol $\bullet$ represents the endpoint of the vector $\mathbf x$.
We set $\mathcal{F}=\{[\mathbf x,i]^*,\mathbf x\in\mathbb Z^3, i\in\{1,2,3\}\}$.

Let $\mathcal W$ be the set of maps from $\mathcal{F}$ to $\mathbb Z_{\geq 0}$: such a map is called a \tdef{weight function}. A weight function $w\in \mathcal W$ gives a weight $w([\mathbf x,i]^*)\in\mathbb Z_{\geq 0}$ to any facet. Equipped with the addition of maps, $\mathcal W$ is a monoid.

A \tdef{multiset of facets} is a map $m:\mathbb Z^3 \rightarrow \mathbb Z_{\geq0}^3$.
We denote by $\mathcal M$ the set of {multisets of facets}.
The set $\mathcal M$, equipped with the addition of maps, is a monoid. 

Multisets of facets and weight functions are equivalent objects.
Indeed, a multiset $m\in\mathcal M$ defines a weight function $w_m\in \mathcal W$ by declaring that $w_m([\mathbf x,i]^*)$ is the $i$th coordinate of $m(\mathbf x)$.
The map $$\begin{array}{ccc}
\mathcal M&\rightarrow&\mathcal W\\
m&\mapsto&w_m
\end{array}$$ is an isomorphism of monoids. 
The inverse of the map is given by
$$\begin{array}{ccc}
\mathcal W&\rightarrow&\mathcal M\\
w&\mapsto& \bigg(\mathbf x\mapsto \Big(w([\mathbf x,1]^*),w([\mathbf x,2]^*),w([\mathbf x,3]^*)\Big)\bigg)
\end{array}$$

The group $\mathbb Z^3$ acts naturally on $\mathcal M$: if $\mathbf v\in\mathbb Z^3$, $m\in\mathcal M$ then 
$$m+\mathbf v: \mathbf x\mapsto m(\mathbf x-\mathbf v).$$

The {\bf support} $\supp(w)$ of a weight function $w$ is the union of facets which have positive weight:
$$\supp(w)=\bigcup_{w([\mathbf x,i]^*)>0} [\mathbf x,i]^* .$$
It is a subset of $\mathbb R^3$.
The {\bf support} $\supp(m)$ of a multiset of facets $m$ is the support of the corresponding weight function:
$$\supp(m)=\supp(w_m).$$

Let $\mathcal W^\circ\subset \mathcal W$ be the subset of weight functions which take values in $\{0,1\}$. We denote by $\mathcal M^\circ$ the corresponding subset of $\mathcal M$: a multiset of facets $m$ is in $\mathcal M^\circ$ if and only if for all $\mathbf x\in\mathbb Z^3$, the coordinates of $m(\mathbf x)$ are in $\{0,1\}$.

\begin{remark}\label{rem-support}
We notice that a multiset of facets in $\mathcal M^\circ$ (or a weight function in $\mathcal W^\circ$) is totally determined by its support.
\end{remark}

\subsection{Dual substitutions}\label{sect:dual_sub}

In this Section we quickly review a construction due to Arnoux-Ito \cite{Arn.Ito.01}
that associates to a unimodular substitution $s$ what is called a dual substitution
$\mathbf E_1^*(s)$. For details we refer to~\cite{Arn.Ito.01,CANT}.
In particular, this construction can be applied to the Tribonacci substitution to
lead the dual substitution $\EOSS$ defined below. The definition of a substitution will be given in Section \ref{sec:subst}.

Consider \[
\Ms = \begin{pmatrix}1&1&1\\1&0&0\\ 0&1&0 \end{pmatrix}.
\]
This matrix has characteristic polynomial $X^3-X^2-X-1$.
Its dominant eigenvalue $\beta$ is a Pisot number:
$\beta > 1$ and its conjugates $\alpha, \overline{\alpha} \in \mathbb C$ are such that $|\alpha|<1$.
The euclidean space $\mathbb R^3$ is hence decomposed as the direct sum of
the \tdef{expanding line} (spanned by the left $\beta$-eigenvector $\mathbf v_\beta$ of $\Ms$)
and the \tdef{contracting plane} $\mathcal P$ associated with
the complex eigenvalues $\alpha, \overline{\alpha}$.
Let $\pi_\beta : \mathbb R^3 \rightarrow \mathcal P$ be the projection on $\mathcal P$
along the line $\mathbb R \mathbf v_\beta$.
We denote by $\mathbf h : \mathcal P \rightarrow \mathcal P$
the restriction the action of $\Ms$ to $\mathcal P$,
which is contracting because $|\alpha| < 1$. Remark that $M_s,\mathbf h, \pi_\beta$ commute.

\begin{definition}
\label{def:eos}
We define
\[
\EOSS :
\begin{cases}
\lbrack\mathbf x, 1\rbrack^*
    \mapsto
    \mathbf M_s^{-1}\mathbf x + \big(\lbrack\mathbf 0,1\rbrack^* \cup \lbrack \mathbf 0,2\rbrack^* \cup \lbrack \mathbf 0,3\rbrack^*\big)\\
\lbrack\mathbf x, 2\rbrack^*
    \mapsto
    \mathbf M_s^{-1}\mathbf x + \lbrack\mathbf e_3, 1\rbrack^*\\
\lbrack\mathbf x, 3\rbrack^*
    \mapsto
    \mathbf M_s^{-1}\mathbf x + \lbrack\mathbf e_3, 2\rbrack^*.
\end{cases}
\]
\end{definition}

Alternatively $\EOSS$ can be defined using multisets as following:
\begin{itemize}
\item The image of $\lbrack\mathbf x, 1\rbrack^*$ by $\EOSS$ is the multiset $(\mathbb Z^3,m)$ where 
$$m(\mathbf y) = \begin{cases}
\mathbf 0 & \;\;\text{ if }\;\; \mathbf y\neq  \mathbf M_s^{-1}\mathbf x\\
(1,1,1) & \;\;\text{ if }\;\; \mathbf y= \mathbf M_s^{-1}\mathbf x.
 \end{cases}$$
\item The image of $\lbrack\mathbf x, 2\rbrack^*$ by $\EOSS$ is the multiset $(\mathbb Z^3,m)$ where 
$$m(\mathbf y) = \begin{cases}
\mathbf 0 & \;\;\text{ if }\;\; \mathbf y\neq  \mathbf M_s^{-1}\mathbf x\\
(0,0,1) & \;\;\text{ if }\;\; \mathbf y= \mathbf M_s^{-1}\mathbf x.
 \end{cases}$$

\item The image of $\lbrack\mathbf x, 3\rbrack^*$ by $\EOSS$ is the multiset $(\mathbb Z^3,m)$ where 
$$m(\mathbf y) = \begin{cases}
\mathbf 0 & \;\;\text{ if }\;\; \mathbf y\neq  \mathbf M_s^{-1}\mathbf x\\
(0,1,0) & \;\;\text{ if }\;\; \mathbf y= \mathbf M_s^{-1}\mathbf x.
 \end{cases}$$
\end{itemize}

We extend $\EOSS$ to $\mathcal M$ by declaring that the image of a union of faces
is the union of the images of these faces (the multiplicities of faces add up).
We also note for future application that for all $\mathbf x, \mathbf u\in\mathbb Z^3$,
and for all $i\in\{1,2,3\}$, we have
\begin{equation}\label{equE}
\EOSS\big([\mathbf x,i]^*+ \mathbf u\big)=\EOSS[\mathbf x,i]^*+ \mathbf M_s^{-1}\mathbf u\end{equation}

In practice, in order to simplify the notation, we represent $\EOSS$ by the following pictures
\begin{align*}
    \myvcenter{\begin{tikzpicture}
    [x={(-0.346410cm,-0.200000cm)}, y={(0.346410cm,-0.200000cm)}, z={(0.000000cm,0.400000cm)}]
    \definecolor{facecolor}{rgb}{0.8,0.8,0.8}
    \fill[thick, fill=facecolor, draw=black, shift={(0,0,0)}]
    (0, 0, 0) -- (0, 1, 0) -- (0, 1, 1) -- (0, 0, 1) -- cycle;
    \node[circle,fill=black,draw=black,minimum size=1.5mm,inner sep=0pt] at (0,0,0) {};
    \end{tikzpicture}}
& \ \mapsto \
    \myvcenter{\begin{tikzpicture}
    [x={(-0.346410cm,-0.200000cm)}, y={(0.346410cm,-0.200000cm)}, z={(0.000000cm,0.400000cm)}]
    \definecolor{facecolor}{rgb}{0.8,0.8,0.8}
    \fill[thick, fill=facecolor, draw=black, shift={(0,0,0)}]
    (0, 0, 0) -- (1, 0, 0) -- (1, 1, 0) -- (0, 1, 0) -- cycle;
    \fill[thick, fill=facecolor, draw=black, shift={(0,0,0)}]
    (0, 0, 0) -- (0, 1, 0) -- (0, 1, 1) -- (0, 0, 1) -- cycle;
    \fill[thick, fill=facecolor, draw=black, shift={(0,0,0)}]
    (0, 0, 0) -- (0, 0, 1) -- (1, 0, 1) -- (1, 0, 0) -- cycle;
    \node[circle,fill=black,draw=black,minimum size=1.5mm,inner sep=0pt] at (0,0,0) {};
    \end{tikzpicture}}
&
    \myvcenter{\begin{tikzpicture}
    [x={(-0.346410cm,-0.200000cm)}, y={(0.346410cm,-0.200000cm)}, z={(0.000000cm,0.400000cm)}]
    \definecolor{facecolor}{rgb}{0.8,0.8,0.8}
    \fill[thick, fill=facecolor, draw=black, shift={(0,0,0)}]
    (0, 0, 0) -- (0, 0, 1) -- (1, 0, 1) -- (1, 0, 0) -- cycle;
    \node[circle,fill=black,draw=black,minimum size=1.5mm,inner sep=0pt] at (0,0,0) {};
    \end{tikzpicture}}
& \ \mapsto \
    \myvcenter{\begin{tikzpicture}
    [x={(-0.346410cm,-0.200000cm)}, y={(0.346410cm,-0.200000cm)}, z={(0.000000cm,0.400000cm)}]
    \definecolor{facecolor}{rgb}{0.8,0.8,0.8}
    \fill[thick, fill=facecolor, draw=black, shift={(0,0,1)}]
    (0, 0, 0) -- (0, 1, 0) -- (0, 1, 1) -- (0, 0, 1) -- cycle;
    \node[circle,fill=black,draw=black,minimum size=1.5mm,inner sep=0pt] at (0,0,0) {};
    \draw[thick, densely dotted] (0,0,0) -- (0,0,1);
    \end{tikzpicture}}
&
    \myvcenter{\begin{tikzpicture}
    [x={(-0.346410cm,-0.200000cm)}, y={(0.346410cm,-0.200000cm)}, z={(0.000000cm,0.400000cm)}]
    \definecolor{facecolor}{rgb}{0.8,0.8,0.8}
    \fill[thick, fill=facecolor, draw=black, shift={(0,0,0)}]
    (0, 0, 0) -- (1, 0, 0) -- (1, 1, 0) -- (0, 1, 0) -- cycle;
    \node[circle,fill=black,draw=black,minimum size=1.5mm,inner sep=0pt] at (0,0,0) {};
    \end{tikzpicture}}
& \ \mapsto \
    \myvcenter{\begin{tikzpicture}
    [x={(-0.346410cm,-0.200000cm)}, y={(0.346410cm,-0.200000cm)}, z={(0.000000cm,0.400000cm)}]
    \definecolor{facecolor}{rgb}{0.8,0.8,0.8}
    \fill[thick, fill=facecolor, draw=black, shift={(0,0,1)}]
    (0, 0, 0) -- (0, 0, 1) -- (1, 0, 1) -- (1, 0, 0) -- cycle;
    \node[circle,fill=black,draw=black,minimum size=1.5mm,inner sep=0pt] at (0,0,0) {};
    \draw[thick, densely dotted] (0,0,0) -- (0,0,1);
    \end{tikzpicture}}
\end{align*}
where the black dots in the preimages stand for $\mathbf x$,
and the black dots in the images stand for $\mathbf M_s^{-1}\mathbf x$.
We denote the Euclidean scalar product of to vectors $\mathbf u, \mathbf v\in \mathbb R^3$ by
$\langle\mathbf u, \mathbf v\rangle$,
and we define $\mathcal U$ as the multiset of facets in $\mathcal M^\circ$ whose support is $[\mathbf 0,1]^*\cup[\mathbf 0,2]^*\cup[\mathbf 0,3]^*$ (see Remark~\ref{rem-support}).

\begin{proposition}[\cite{Arn.Ito.01,ABI02,Arn.Berth.Sieg.04}]
\label{prop:ai2}
$\;$
\begin{itemize}
\item For every integer $n$, $\EOSS^n(\mathcal U)$ belongs to $\mathcal M^\circ$, so it can be considered as a subset of $\mathbb R^3$ (see Remark~\ref{rem-support}).

\item For every integer $n$, $\EOSS^n(\mathcal U)$ is a subset of $\EOSS^{n+1}(\mathcal U)$. The increasing sequence of $\EOSS^n(\mathcal U)$ converges and we denote
$$\Sstep=\lim_n \EOSS^n(\mathcal U)=\bigcup_{n\in\mathbb N}\EOSS^n(\mathcal U).$$
This set $\Sstep$ is called the stepped surface.

\item Moreover: 
\[ \Sstep=\bigcup_{i\in\{1,2,3\}} \quad\bigcup_{ 
\substack{
\mathbf x \in \mathbb Z^3 \\
0 \leq \langle \mathbf x, \mathbf v_\beta \rangle < \langle \mathbf e_i, \mathbf v_\beta \rangle }}
[\mathbf x,i]^*.
\] 

\item The restriction $\pi_\beta:\Sstep\rightarrow\mathcal P$ of $\pi_\beta$ to $\Sstep$ is an homeomorphism.

\item This map induces a tiling of the plane $\mathcal P$: we denote this tiling by $\Tstep$. The set of tiles of $\Tstep$ is:

\[ \bigcup_{i\in\{1,2,3\}}
\{\pi_\beta([\mathbf 0,i]^*) + \pi_\beta(\mathbf x) : \mathbf x \in \mathbb Z^3,
0 \leq \langle \mathbf x, \mathbf v_\beta \rangle < \langle \mathbf e_i, \mathbf v_\beta \rangle \}.
\] 
\end{itemize}
\end{proposition}

We say that a vector $\mathbf x\in\mathbb Z^3$ lies in $\Sstep$ if there exists $i\in\{1,2,3\}$ such that $[\mathbf x,i]^*$ is a subset of $\Sstep$. Proposition~\ref{prop:ai2} implies that the set of vectors lying in $\Sstep$ is precisely $\mathcal V = \mathcal V_1 \cup \mathcal V_2 \cup \mathcal V_3$,
where
$$\mathcal V_i = \{ \mathbf x \in \mathbb Z^3 \;:\; 0 \leq \langle \mathbf x, \mathbf v_\beta \rangle < \langle \mathbf e_i, \mathbf v_\beta \rangle \}.$$

\begin{remark}\label{rem:D_i}
We set $\mathcal D_i=\pi_\beta(\mathcal V_i)$. Then $\mathcal D = \mathcal D_1 \cup \mathcal D_2 \cup \mathcal D_3$ is a Delone set in $\mathcal P$.
The tiling $\Tstep$ is obtained by putting in $\mathcal P$ a tile of type $i$ (i.e. a translated image of $\pi_\beta([\mathbf 0,i]^*)$) at each vector in $\mathcal D_i$.
\end{remark}

\subsection{Link between the tiling $\Tstep$ and the Rauzy fractal}\label{sec:subst}
In this section we recall basic facts about Rauzy fractals and substitutions~\cite{Pyt.02}.
We consider the free monoid $\{1,2,3\}^*$ and the \tdef{Tribonacci substitution}
$s : \{1,2,3\}^* \rightarrow \{1,2,3\}^*$, which is a morphism defined by
\[
s :
1 \mapsto 12 \quad
2 \mapsto 13 \quad
3 \mapsto 1.
\]

Denote by $u = 12131\ldots$ the infinite word on the alphabet $\{1,2,3\}$ such that $s(u)=u$.
In fact, for all $n\in\mathbb N$, $s^n(1)$ is a prefix of $s^{n+1}(1)$,
so that $u=\lim_{+\infty}s^n(1)$.
We denote by $u_i\in\{1,2,3\}$ the $i$-th letter in $u$:
$u = u_1u_2u_3\ldots$ with $u_1=1$, $u_2=2$, $u_3=3$.

Let us define $\Ms$ the \tdef{incidence matrix} of $s$:
its $i$th column vector is equal to $\mathbf P(s(i))$,
where $\mathbf P$ be the abelianization map from $\{1,2,3\}^*$ to $\mathbb{Z}^3$ defined by
$\mathbf P(w) = (|w|_1, |w|_2 ,|w|_3)$ and $|w|_i$ stands for the number of occurrences of $i$ in $w$.
Consistently with the notation of Section~\ref{sect:dual_sub}, we have:
\[
\Ms = \begin{pmatrix}1&1&1\\1&0&0\\ 0&1&0 \end{pmatrix}.
\]

In the following proposition, we go on using the notation introduced in Section~\ref{sect:dual_sub}.
In particular, the map $\mathbf h : \mathcal P \rightarrow \mathcal P$ is the restriction the action of $\Ms$ to contracting eigenplane $\mathcal P$.

\begin{proposition}[\cite{Rau.82, Can.Sieg.01, Arn.Ito.01}]
\label{prop:rauzy}
$\;$
\begin{itemize}
\item The sets
    \begin{itemize}
    \item $\mathcal{R}_i=\overline{\left\{\pi_\beta \big(\mathbf P(u_1\dots u_{j-1})\big) : j\in\mathbb N,\; u_j=i\right\}}$ for $i \in \{1,2,3\}$,
    \item $\mathcal{R}=\overline{\left\{\pi_\beta  \big(\mathbf P(u_1\dots u_{j-1}) \big), j\in\mathbb N\right\}} = \mathcal R_1 \cup \mathcal R_2 \cup \mathcal R_3$
    \end{itemize}
are compact subsets of $\mathcal P$ (where $\overline A$ denotes the closure of a subset $A$ in $\mathcal P$).

\item The $\mathcal{R}_i$'s are the solution of the following IFS:
\[
\begin{cases}
\mathcal{R}_1=\mathbf h\mathcal{R}_1\cup \mathbf h\mathcal{R}_2\cup \mathbf h\mathcal{R}_3 \\
\mathcal{R}_2=\mathbf h\mathcal{R}_1+\pi_\beta \mathbf (e_1) \\
\mathcal{R}_3=\mathbf h\mathcal{R}_2+\pi_\beta \mathbf (e_1).
\end{cases}
\]

\item There exists a tiling of the plane $\mathcal P$, that will be denoted by $\Tfrac$, whose set of tiles is:
\[ \bigcup_{i\in\{1,2,3\}}
\{\mathcal R_i + \pi_\beta(\mathbf x) : \mathbf x \in \mathbb Z^3,
0 \leq \langle \mathbf x, \mathbf v_\beta \rangle < \langle \mathbf e_i, \mathbf v_\beta \rangle \}.
\]

\end{itemize}
\end{proposition}

The set $\mathcal R$ is called the {\bf Rauzy fractal} of $s$
and $\mathcal R_1, \mathcal R_2, \mathcal R_3$ are the {\bf subtiles} of $\mathcal R$.

\begin{remark}\label{rem:Tfrac Ttop}
Comparing Proposition~\ref{prop:rauzy} and Proposition~\ref{prop:ai2},
we see that the positions of the tiles in $\Tfrac$ and $\Tstep$ are given by the same formula.
Indeed, the tiling $\Tfrac$ is obtained by putting in $\mathcal P$ a tile of type $i$ (i.e. a translated image of $\mathcal R_i$) at each vector in $\mathcal D_i$, where the sets $\mathcal D_i$ are precisely the ones of Remark~\ref{rem:D_i}.
This explicits the strong relation between the two tilings $\Tfrac$ and $\Tstep$.
\end{remark}

\section{Link between topological and dual substitutions}
\label{sec:link}

\subsection{The position map}
\label{sec:position}

In this section we define the \emph{position map} $\omega_0$ from the set $\mathfrak{P}$ of paths
of tiles in $\sigma^\infty(C)$ to $\mathbb{Z}^3$.
(See Sections~\ref{sec:add} and~\ref{sec:omega_0} below for precise definitions of $\mathfrak{P}$ and $\omega_0$.)

We use the term \emph{position map} because $\omega_0$ will be used
to give a geometric interpretation of the relative positions of two tiles in a common patch.
This geometric interpretation is given by a vector in $\mathbb{Z}^3$
(the \emph{position} of a tile with respect to another one).
This will simplify the work done in Section~\ref{sec:patchtostepped} where we associate geometric patches of stepped surfaces to abstract topological patches.

\subsubsection{Notation}

Let $\sigma$ be the Tribonacci topological substitution defined in Section \ref{sect:deftoptribo}. We denote by $P$ a patch of $\sigma^\infty(C)$, or possibly $P=\sigma^\infty(C)$.  
\begin{definition}\label{def:path of tiles}
Consider a positive integer $n$. A \tdef{path of tiles} $\gamma$ in $P$ is a sequence $T_0, \ldots, T_{n}$ of tiles of $P$
such that two consecutive tiles $T_i, T_{i+1}$ are different and share (at least) one common edge for all $i \in \{0, \ldots, n-1\}$. The integer $n+1$ is the {\bf length} of the path $\gamma=T_0, \ldots, T_{n}$. The set $\{T_0,\dots, T_n\}\subset P$ is called the {\bf support of} $\gamma$ in $P$.
When $T_0=T_{n}$,  $\gamma$ is a {\bf loop of tiles}. The integer $n$ is the {\bf length} of the loop $\gamma=T_0, \ldots, T_{n}$. 
The path $\gamma=T_0, \ldots, T_{n}$ and the path $\gamma'=T'_0, \ldots, T'_{m}$ can be {\bf concatenated} if $T'_0=T_n$.
The {\bf concatenation} of these paths is the path of tiles $\gamma\gamma'=T_0,\ldots, T_n, T'_1,\ldots, T'_m$. 
\end{definition}

Let $\gamma$  be a loop of tiles in $\sigma^\infty(C)$. 
Among the connected components of the complementary of the support of $\gamma$, there is exactly one, denoted by $\mathsf{C}_\infty(\gamma)$, which contains an infinite number of tiles. We denote by $\mathsf{C}_0(\gamma)$ the complement of $\mathsf{C}_\infty(\gamma)$: it is a patch, in particular it is homeomorphic to a disc. Alternatively $\mathsf{C}_0(\gamma)$ is the smallest subpatch of $\sigma^\infty(C)$ containing the support of $\gamma$.
We define the \tdef{area of $\gamma$} to be the number of tiles in $\mathsf{C}_0(\gamma)$:
$$\area(\gamma)=|\mathsf{C}_0(\gamma)|.$$

Now we define an equivalence relation on paths of tiles which will define a {\bf protopath of tiles}:
The path $\gamma=T_0, \ldots, T_{n}$ and the path $\gamma'=T'_0, \ldots, T'_{m}$ are equivalent if
\begin{itemize}
\item $m=n$
\item For every $i\in\{0,\dots,n\}$, $T_i$ and $T'_i$ have the same prototile type.
\item The gluing edges of $T_i$ and $T_{i+1}$ have the same type as the gluing edges of  $T'_i$ and $T'_{i+1}$.
\end{itemize}
In the same way we define the notion of {\bf protoloop}.
The notions of concatenation, area and length naturally extend to protopaths. 

\subsubsection{Additivity}
\label{sec:add}

Let $P$ be  patch in $\sigma^\infty(C)$. By definition, $P$ is homeomorphic to a disc and its boundary is homeomorphic to the circle $S^1$. 
Let $T$ be a tile in $P$, the {\bf wreath} of $T$ in $P$ is the subset of $P\smallsetminus\{T\}$ made of tiles that have at least one vertex in common with $T$. We denote it by $\wreath_P(T)$. A {\bf cut tile} of $P$ is a tile whose wreath in $P$ is not connected.

Let $T$ be a cut tile of $P$. Then $P\smallsetminus T$ has at least $2$ connected components, and each of these components is a patch. 

\begin{lemma}\label{lem:topo-couronne}
Let $P$ be a finite patch in $\sigma^\infty(C)$. 
There exists one tile of $P$ which is not a cut tile and has one edge in the boundary of $P$.
\end{lemma}
\begin{proof}
Pick a tile $T_0$ in $P$ which has a vertex in the boundary of $P$.
If $T_0$ is not a cut tile, we are done.
Otherwise, because $P$ is homeomorphic to a disk, $P\smallsetminus T_0$ has at least two connected components: we choose one of them that we denote by $P_1$.
Pick a tile $T_1$ in $P_1$ which has a vertex in the boundary of $P$.
If $T_1$ is not a cut tile, we are done.
Otherwise, $P\smallsetminus T_1$ has at least two connected components, and at least one of them is included in $P_1$:
we choose one of these ones, that we denote by $P_2$. If every tile with one edge in the boundary is a cut-tile we obtain an infinite number of nested connected components which is a contradiction with the fact that $P$ contains a finite number of tiles.
\end{proof}

Let $\mathfrak{P}$ be the set of protopaths of tiles in $\sigma^\infty(C)$. A map $\omega: \mathfrak{P}\rightarrow \mathbb{Z}^3$ is \tdef{additive} if for every $\gamma,\gamma'\in\mathfrak{P}$ that can be concatenated,
we have $\omega(\gamma\gamma')=\omega(\gamma)+\omega(\gamma')$.
Hence, an additive map $\omega: \mathfrak{P}\rightarrow \mathbb{Z}^3$ is uniquely defined by the image of the protopaths of length 2.

\begin{definition}
Let $\mathfrak{P}_0\subseteq\mathfrak{P}$ be the subset consisting of protoloops 
$\gamma=T_0,\dots,T_n$ (with $T_0=T_n$)
of tiles in $\sigma^\infty(C)$  such that
\begin{itemize}
  \item for every $i\in\{1,\dots,n-1\}$, $T_i\in\wreath_{\mathsf{C}_0(\gamma)}(T_0)$,
  \item for every $i\neq j\in\{1,\dots,n-1\}$, $T_i\neq T_j$.
\end{itemize}
We notice that $\mathfrak{P}_0$ contains all protoloops of lentgh 2.
\end{definition}

\begin{lemma}\label{lem:P0 fini}
The set $\mathfrak{P}_0$ is finite.
\end{lemma}
\begin{proof}
The valence of every vertex in $\sigma^\infty(C)$ is bounded (by 3, see Lemma \ref{lem:val-bounded}), we deduce that the cardinlity is finite.
\end{proof}

It is possible to produce an explicit list of the elements in $\mathfrak{P}_0$.
We detail below all the elements of $\mathfrak{P}_0$ of length 3.

\begin{center}
\myvcenter{\begin{tikzpicture}[x={(-0.216506cm,-0.125000cm)}, y={(0.216506cm,-0.125000cm)}, z={(0.000000cm,0.250000cm)}]
\tuileA{A}{(0,0,0)}{white}{}{}{}{}{}{}
\tuileB{B}{(0,0,0)}{white}{}{}{}{}{}{}
\tuileC{C}{(0,0,0)}{white}{}{}{}{}{}{}{}{}{}{}
\end{tikzpicture}} \
\myvcenter{\begin{tikzpicture}[x={(-0.216506cm,-0.125000cm)}, y={(0.216506cm,-0.125000cm)}, z={(0.000000cm,0.250000cm)}]
\tuileA{A}{(0,0,0)}{white}{}{}{}{}{}{}
\tuileB{B}{(0,2,-3)}{white}{}{}{}{}{}{}
\tuileC{C}{(2,-1,-2)}{white}{}{}{}{}{}{}{}{}{}{}
\end{tikzpicture}} \
\myvcenter{\begin{tikzpicture}[x={(-0.216506cm,-0.125000cm)}, y={(0.216506cm,-0.125000cm)}, z={(0.000000cm,0.250000cm)}]
\tuileA{A}{(0,0,0)}{white}{}{}{}{}{}{}
\tuileB{B}{(0,0,0)}{white}{}{}{}{}{}{}
\tuileC{C}{(2,-1,-2)}{white}{}{}{}{}{}{}{}{}{}{}
\end{tikzpicture}} \
\myvcenter{\begin{tikzpicture}[x={(-0.216506cm,-0.125000cm)}, y={(0.216506cm,-0.125000cm)}, z={(0.000000cm,0.250000cm)}]
\tuileA{A}{(0,-2,3)}{white}{}{}{}{}{}{}
\tuileB{B}{(0,0,0)}{white}{}{}{}{}{}{}
\tuileC{C}{(0,-2,3)}{white}{}{}{}{}{}{}{}{}{}{}
\end{tikzpicture}} \
\myvcenter{\begin{tikzpicture}[x={(-0.216506cm,-0.125000cm)}, y={(0.216506cm,-0.125000cm)}, z={(0.000000cm,0.250000cm)}]
\tuileB{B}{(0,0,0)}{white}{}{}{}{}{}{}
\tuileB{B}{(-1,0,2)}{white}{}{}{}{}{}{}
\tuileC{C}{(0,0,0)}{white}{}{}{}{}{}{}{}{}{}{}
\end{tikzpicture}}
\myvcenter{\begin{tikzpicture}[x={(-0.216506cm,-0.125000cm)}, y={(0.216506cm,-0.125000cm)}, z={(0.000000cm,0.250000cm)}]
\tuileB{B}{(0,0,0)}{white}{}{}{}{}{}{}
\tuileB{B}{(-1,0,2)}{white}{}{}{}{}{}{}
\tuileC{C}{(1,-3,3)}{white}{}{}{}{}{}{}{}{}{}{}
\end{tikzpicture}} \
\myvcenter{\begin{tikzpicture}[x={(-0.216506cm,-0.125000cm)}, y={(0.216506cm,-0.125000cm)}, z={(0.000000cm,0.250000cm)}]
\tuileB{B}{(-1,0,2)}{white}{}{}{}{}{}{}
\tuileC{C}{(0,0,0)}{white}{}{}{}{}{}{}{}{}{}{}
\tuileC{C}{(-1,0,2)}{white}{}{}{}{}{}{}{}{}{}{}
\end{tikzpicture}} \
\myvcenter{\begin{tikzpicture}[x={(-0.216506cm,-0.125000cm)}, y={(0.216506cm,-0.125000cm)}, z={(0.000000cm,0.250000cm)}]
\tuileB{B}{(-2,3,-1)}{white}{}{}{}{}{}{}
\tuileC{C}{(0,0,0)}{white}{}{}{}{}{}{}{}{}{}{}
\tuileC{C}{(-1,0,2)}{white}{}{}{}{}{}{}{}{}{}{}
\end{tikzpicture}}
\myvcenter{\begin{tikzpicture}[x={(-0.216506cm,-0.125000cm)}, y={(0.216506cm,-0.125000cm)}, z={(0.000000cm,0.250000cm)}]
\tuileB{B}{(-2,1,2)}{white}{}{}{}{}{}{}
\tuileC{C}{(0,0,0)}{white}{}{}{}{}{}{}{}{}{}{}
\tuileC{C}{(0,-2,3)}{white}{}{}{}{}{}{}{}{}{}{}
\end{tikzpicture}} \
\myvcenter{\begin{tikzpicture}[x={(-0.216506cm,-0.125000cm)}, y={(0.216506cm,-0.125000cm)}, z={(0.000000cm,0.250000cm)}]
\tuileB{B}{(0,0,0)}{white}{}{}{}{}{}{}
\tuileC{C}{(0,0,0)}{white}{}{}{}{}{}{}{}{}{}{}
\tuileC{C}{(0,-2,3)}{white}{}{}{}{}{}{}{}{}{}{}
\end{tikzpicture}} \
\myvcenter{\begin{tikzpicture}[x={(-0.216506cm,-0.125000cm)}, y={(0.216506cm,-0.125000cm)}, z={(0.000000cm,0.250000cm)}]
\tuileC{C}{(0,0,0)}{white}{}{}{}{}{}{}{}{}{}{}
\tuileC{C}{(-1,2,-1)}{white}{}{}{}{}{}{}{}{}{}{}
\tuileC{C}{(-1,0,2)}{white}{}{}{}{}{}{}{}{}{}{}
\end{tikzpicture}}
\myvcenter{\begin{tikzpicture}[x={(-0.216506cm,-0.125000cm)}, y={(0.216506cm,-0.125000cm)}, z={(0.000000cm,0.250000cm)}]
\tuileC{C}{(0,0,0)}{white}{}{}{}{}{}{}{}{}{}{}
\tuileC{C}{(-1,2,-1)}{white}{}{}{}{}{}{}{}{}{}{}
\tuileC{C}{(0,2,-3)}{white}{}{}{}{}{}{}{}{}{}{}
\end{tikzpicture}}
\end{center}

\begin{lemma}
\label{lem:loop:omega}
Let $\omega: \mathfrak{P}\rightarrow \mathbb{Z}^3$ an additive map such that
$\omega$ vanishes on the elements of $\mathfrak{P}_0$ . Then $\omega$ vanishes on every protoloop in $\sigma^\infty(C)$.
\end{lemma}
\begin{proof}
The proof is by induction on the area of the protoloop of tiles $\gamma$ in $\sigma^\infty(C)$.
According to Definition~\ref{def:path of tiles}, a loop of tiles has length a least 2,
and thus also area at least 2.
Moreover, a loop of tiles $\gamma$ with area 2 is the concatenation of a certain number of copies of  a same loop of tiles of length 2 $\gamma'=T_0,T_1,T_0$.
Since any loop of length 2 is in $\mathfrak{P}_0$, we get that $\omega(\gamma')=0$.
By additivity of $\omega$, we derive that $\omega(\gamma)=0$.

Suppose that $\omega$ vanishes on every loop in $\sigma^\infty(C)$ of area at most $k$.
Let $\gamma=T_0,\dots,T_n$ be a loop of area $k+1$.
By Lemma~\ref{lem:topo-couronne} there exists a tile $T$ in $\mathsf{C}_{0}(\gamma)$
which is not a cut-tile and has one edge in the boundary of $\mathsf{C}_{0}(\gamma)$.
The tile $T$ may occur several times in $\gamma$, and for each occurrence we will  successively act as follows.

Let $T_i=T$ be an occurrence of $T$ in $\gamma$.
\begin{itemize}
  \item If $T_{i-1}=T_{i+1}$, we set $\gamma'=T_0,\dots,T_{i-1}=T_{i+1},\dots,T_n$.
Then by additivity of $\omega$, 
$$\omega(\gamma)=\omega(\gamma')+\omega(T_{i-1},T_i)+\omega(T_{i},T_{i+1})
=\omega(\gamma').$$
  \item If $T_{i-1}\neq T_{i+1}\in\wreath_{\mathcal{C}_0(\gamma)}(T)$, see Figure~\ref{fig:additivity}.
  Since $T$ is not a cut tile of $\mathcal{C}_0(\gamma)$, 
  there exists a path of tiles $T_{i-1},T'_1,\dots, T'_d,T_{i+1}$
  in $\wreath_{\mathcal{C}_0(\gamma)}(T)$ joining $T_{i-1}$ and $T_{i+1}$.
  Then 
  $$\gamma''=T,T_{i-1},T'_1,\dots, T'_d,T_{i+1},T$$ 
  is a loop of tiles, 
  and since $T$ has at least an edge in the boundary of $\mathcal{C}_0(\gamma)$,
  we see that $\gamma''\in\mathfrak{P}_0$. In particular, $\omega(\gamma'')=0$.
  We set 
  $$\gamma'=T_0,\dots,T_{i-1},T'_1,\dots, T'_d,T_{i+1},\dots,T_n.$$
  Then by additivity of $\omega$, 
  $$\omega(\gamma)=\omega(\gamma')+\omega(\gamma'')=\omega(\gamma').$$ 
\end{itemize}  

After proceeding as above for each occurrence of $T$ in $\gamma$, 
we end up with a loop of tiles $\gamma_0$ such that
$\omega(\gamma)=\omega(\gamma_0)$
and $\area(\gamma_0)\leq k$ (since the support of $\gamma_0$ is included in
$\mathcal{C}_0(\gamma)\smallsetminus\{T\}$).
We conclude using the induction hypothesis on $\gamma_0$.
\end{proof}

\begin{figure}[ht]
\centering
\begin{tikzpicture}[x={(-0.5196cm,-0.3cm)}, y={(0.5196cm,-0.3cm)}, z={(0cm,0.6cm)}]
\fill[fill=black!20,draw=black,thick]
(3,-1,0) -- (-1,0,1) -- (-2,0,1) -- (-2,1,1) -- (-3,1,1) -- (-3,1,1) -- (-3,2,1) -- (-3,2,0) -- (-3,3,0) -- (-3,3,-1) -- (-2,3,-1) -- (-2,3,-2) -- (-1,3,-2) -- (-1,3,-3) -- (0,3,-3) -- (0,2,-3) -- (1,2,-3) -- (1,1,-3) -- (2,1,-3) -- (2,0,-3) -- (3,0,-3) -- (3,-1,-3) -- (4,-1,-3) -- (4,-2,-3) -- (4,-2,-2) -- (3,-2,-2) -- (3,-2,-1) -- cycle;
\tileA{}{}{}{}{}{}{}{(0,1,1)}{white}
\tileC{}{}{}{}{}{}{}{}{}{}{}{(1,0,2)}{white}
\tileC{}{}{}{}{}{}{}{}{}{}{}{(1,-1,0)}{white}
\tileA{}{}{}{}{}{}{}{(1,1,0)}{white}

\draw[line width=0.8mm, red, rounded corners, dotted, shift={(0,0.1,-0.1)}]
(2,-1,0) -- (1,0,0) -- (0,1,0) -- (-1,1,1) -- (-1,3,-1) -- (0,3,-2) -- (1,2,-2) -- (2,1,-2) -- (3,0,-2) -- (3,-1,-1) -- (2,-1,0);
\draw[line width=0.8mm, blue, rounded corners, dotted, shift={(0,-0.1,0.1)}]
(2,-1,0) -- (1,0,0) -- (0,1,0) -- (-1,1,1) -- (1,-1,1) -- cycle;
\node[blue,above] at (0,0,1) {$\gamma''$};
\node[red,right] at (-1,2,0) {$\gamma'$};
\node[left] at (-3,2,0.5) {$P$};
\end{tikzpicture}
\caption{Scheme of proof of Lemma \ref{lem:loop:omega} }
\label{fig:additivity}
\end{figure}
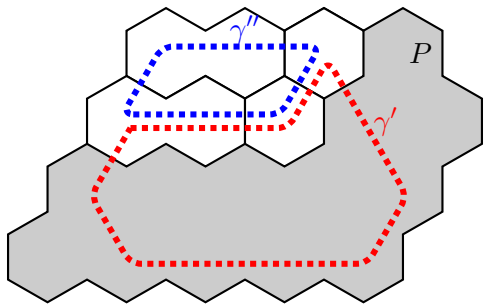

\subsubsection{The position map $\omega_0$}
\label{sec:omega_0}

\begin{definition}
\label{def:omega0}
First we define $\omega_0$ on the set of protopaths of length $2$ in $\sigma^\infty(C)$.
They form a finite set due to the heredity graph of edges.
In Figure~\ref{fig:omega_two} we explicitly give this set and define this list and define $\omega_0$ on it.
For each protopath $\gamma=(T_0,T_1)$ of Figure~\ref{fig:omega_two}, $T_0$ is the white tile.
Moreover, we set $\omega_0(T_1,T_0)=-\omega_0(T_0,T_1)$.

We are now ready to define the map $\omega_0:\mathfrak{P}\rightarrow \mathbb{Z}^3$.
For every protopath $\gamma=T_1,T_2,\dots,T_n$ of length $n\geq 2$, we set
\[
\omega_0(\gamma)=\sum_{i=1}^{n-1} \omega_0(T_{i},T_{i+1}).
\]

Finally, to make sure that $\omega_0$ vanishes on protoloops of $\sigma^\infty(C)$,
it remains to check that the map $\omega_0$ vanishes on the elements of $\mathfrak{P}_0$
(thanks to Lemma~\ref{lem:loop:omega}).
This a finite process, since Lemma~\ref{lem:P0 fini} ensures that $\mathfrak{P}_0$ is finite.
We detail below an instance of the kind of easy computations that have to be carried on:
\begin{align*}
\omega_0\left(
\myvcenter{\begin{tikzpicture}[x={(-0.216506cm,-0.125000cm)}, y={(0.216506cm,-0.125000cm)}, z={(0.000000cm,0.250000cm)}]
\tuileB{B}{(0,0,0)}{white}{}{}{}{}{}{}
\tuileC{C}{(2,-3,1)}{white}{}{}{}{}{}{}{}{}{}{}
\tuileC{C}{(1,-3,3)}{white}{}{}{}{}{}{}{}{}{}{}
\draw[->, draw=red, thick] (1.5,-2.6,0) -- (1,-2,2.5);
\draw[->, draw=red, thick] (1,-1.6,2.5) -- (1,-1,1.3);
\draw[->, draw=red, thick] (1,-1,1.1) -- (1.5,-2.2,0);
\end{tikzpicture}}%
\right)
    &=
\omega_0\left(
\myvcenter{\begin{tikzpicture}[x={(-0.216506cm,-0.125000cm)}, y={(0.216506cm,-0.125000cm)}, z={(0.000000cm,0.250000cm)}]
\tuileC{C}{(2,-3,1)}{white}{}{}{}{}{}{}{}{}{}{}
\tuileC{C}{(1,-3,3)}{black!20}{}{}{}{}{}{}{}{}{}{}
\end{tikzpicture}}
\right)
+
\omega_0\left(
\myvcenter{\begin{tikzpicture}[x={(-0.216506cm,-0.125000cm)}, y={(0.216506cm,-0.125000cm)}, z={(0.000000cm,0.250000cm)}]
\tuileB{B}{(0,0,0)}{black!20}{}{}{}{}{}{}
\tuileC{C}{(1,-3,3)}{white}{}{}{}{}{}{}{}{}{}{}
\end{tikzpicture}}
\right)
+
\omega_0\left(
\myvcenter{\begin{tikzpicture}[x={(-0.216506cm,-0.125000cm)}, y={(0.216506cm,-0.125000cm)}, z={(0.000000cm,0.250000cm)}]
\tuileB{B}{(0,0,0)}{white}{}{}{}{}{}{}
\tuileC{C}{(2,-3,1)}{black!20}{}{}{}{}{}{}{}{}{}{}
\end{tikzpicture}}
\right)
\\
    &= (-1,0,2) + (0,2,-3) + (1,-2,1) = (0,0,0).
\end{align*}
\end{definition}

\begin{figure}[ht!]
\begin{align*}
\myvcenter{\begin{tikzpicture}[x={(-0.216506cm,-0.125000cm)}, y={(0.216506cm,-0.125000cm)}, z={(0.000000cm,0.250000cm)}]
\tuileA{A}{(0,0,0)}{white}{}{}{}{}{}{}
\tuileB{B}{(0,0,0)}{black!20}{}{}{}{}{}{}
\end{tikzpicture}}
& \ \mapsto \
\svect{0}{-1}{1} &
\myvcenter{\begin{tikzpicture}[x={(-0.216506cm,-0.125000cm)}, y={(0.216506cm,-0.125000cm)}, z={(0.000000cm,0.250000cm)}]
\tuileA{A}{(0,0,0)}{white}{}{}{}{}{}{}
\tuileB{B}{(0,2,-3)}{black!20}{}{}{}{}{}{}
\end{tikzpicture}}
& \ \mapsto \
\svect{0}{1}{-2} &
\myvcenter{\begin{tikzpicture}[x={(-0.216506cm,-0.125000cm)}, y={(0.216506cm,-0.125000cm)}, z={(0.000000cm,0.250000cm)}]
\tuileA{A}{(0,0,0)}{white}{}{}{}{}{}{}
\tuileC{C}{(0,0,0)}{black!20}{}{}{}{}{}{}{}{}{}{}
\end{tikzpicture}}
& \ \mapsto \
\svect{-1}{0}{1} &
\myvcenter{\begin{tikzpicture}[x={(-0.216506cm,-0.125000cm)}, y={(0.216506cm,-0.125000cm)}, z={(0.000000cm,0.250000cm)}]
\tuileA{A}{(0,0,0)}{white}{}{}{}{}{}{}
\tuileC{C}{(2,-1,-2)}{black!20}{}{}{}{}{}{}{}{}{}{}
\end{tikzpicture}}
& \ \mapsto \
\svect{1}{-1}{-1} \\
\myvcenter{\begin{tikzpicture}[x={(-0.216506cm,-0.125000cm)}, y={(0.216506cm,-0.125000cm)}, z={(0.000000cm,0.250000cm)}]
\tuileB{B}{(0,0,0)}{white}{}{}{}{}{}{}
\tuileC{C}{(0,0,0)}{black!20}{}{}{}{}{}{}{}{}{}{}
\end{tikzpicture}}
& \ \mapsto \
\svect{-1}{1}{0} &
\myvcenter{\begin{tikzpicture}[x={(-0.216506cm,-0.125000cm)}, y={(0.216506cm,-0.125000cm)}, z={(0.000000cm,0.250000cm)}]
\tuileB{B}{(0,0,0)}{white}{}{}{}{}{}{}
\tuileC{C}{(2,-3,1)}{black!20}{}{}{}{}{}{}{}{}{}{}
\end{tikzpicture}}
& \ \mapsto \
\svect{1}{-2}{1} &
\myvcenter{\begin{tikzpicture}[x={(-0.216506cm,-0.125000cm)}, y={(0.216506cm,-0.125000cm)}, z={(0.000000cm,0.250000cm)}]
\tuileB{B}{(0,0,0)}{white}{}{}{}{}{}{}
\tuileB{B}{(-1,0,2)}{black!20}{}{}{}{}{}{}
\end{tikzpicture}}
& \ \mapsto \
\svect{-1}{0}{2} &
\myvcenter{\begin{tikzpicture}[x={(-0.216506cm,-0.125000cm)}, y={(0.216506cm,-0.125000cm)}, z={(0.000000cm,0.250000cm)}]
\tuileC{C}{(0,0,0)}{white}{}{}{}{}{}{}{}{}{}{}
\tuileC{C}{(-1,2,-1)}{black!20}{}{}{}{}{}{}{}{}{}{}
\end{tikzpicture}}
& \ \mapsto \
\svect{-1}{2}{-1} \\
\myvcenter{\begin{tikzpicture}[x={(-0.216506cm,-0.125000cm)}, y={(0.216506cm,-0.125000cm)}, z={(0.000000cm,0.250000cm)}]
\tuileC{C}{(0,0,0)}{white}{}{}{}{}{}{}{}{}{}{}
\tuileC{C}{(-1,0,2)}{black!20}{}{}{}{}{}{}{}{}{}{}
\end{tikzpicture}}
& \ \mapsto \
\svect{-1}{0}{2} &
\myvcenter{\begin{tikzpicture}[x={(-0.216506cm,-0.125000cm)}, y={(0.216506cm,-0.125000cm)}, z={(0.000000cm,0.250000cm)}]
\tuileC{C}{(0,0,0)}{white}{}{}{}{}{}{}{}{}{}{}
\tuileC{C}{(0,-2,3)}{black!20}{}{}{}{}{}{}{}{}{}{}
\end{tikzpicture}}
& \ \mapsto \
\svect{0}{-2}{3} &
\myvcenter{\begin{tikzpicture}[x={(-0.216506cm,-0.125000cm)}, y={(0.216506cm,-0.125000cm)}, z={(0.000000cm,0.250000cm)}]
\tuileB{B}{(0,0,0)}{white}{}{}{}{}{}{}
\tuileC{C}{(2,-1,-2)}{black!20}{}{}{}{}{}{}{}{}{}{}
\end{tikzpicture}}
& \ \mapsto \
\svect{1}{0}{-2} &
\myvcenter{\begin{tikzpicture}[x={(-0.216506cm,-0.125000cm)}, y={(0.216506cm,-0.125000cm)}, z={(0.000000cm,0.250000cm)}]
\tuileB{B}{(0,0,0)}{white}{}{}{}{}{}{}
\tuileC{C}{(0,-2,3)}{black!20}{}{}{}{}{}{}{}{}{}{}
\end{tikzpicture}}
& \ \mapsto \
\svect{-1}{-1}{3} \\
\myvcenter{\begin{tikzpicture}[x={(-0.216506cm,-0.125000cm)}, y={(0.216506cm,-0.125000cm)}, z={(0.000000cm,0.250000cm)}]
\tuileB{B}{(0,0,0)}{white}{}{}{}{}{}{}
\tuileC{C}{(1,0,-2)}{black!20}{}{}{}{}{}{}{}{}{}{}
\end{tikzpicture}}
& \ \mapsto \
\svect{0}{1}{-2} &
\myvcenter{\begin{tikzpicture}[x={(-0.216506cm,-0.125000cm)}, y={(0.216506cm,-0.125000cm)}, z={(0.000000cm,0.250000cm)}]
\tuileB{B}{(0,0,0)}{white}{}{}{}{}{}{}
\tuileC{C}{(1,-3,3)}{black!20}{}{}{}{}{}{}{}{}{}{}
\end{tikzpicture}}
& \ \mapsto \
\svect{0}{-2}{3}
\end{align*}
\caption{Definition of the map $\omega_0$ over protopaths of length $2$.
The orientation of the path is indicated using colors: the first tile is white
(see Definition~\ref{def:omega0}).}
\label{fig:omega_two}
\end{figure}
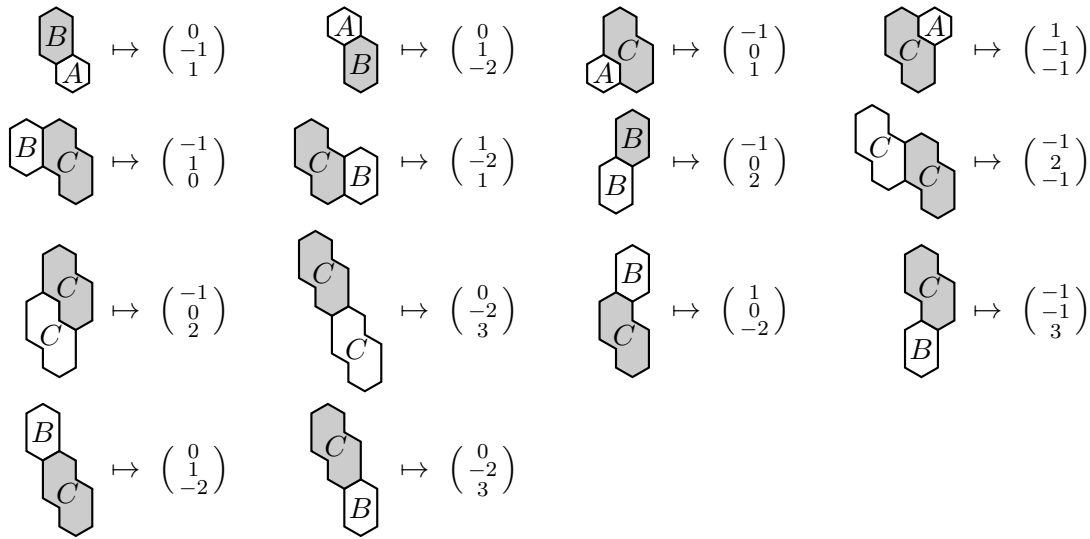

According to Lemma \ref{lem:loop:omega}, we thus obtain the following proposition.

\begin{proposition}\label{prop:omega}
The map $\omega_0:\mathfrak{P}\rightarrow \mathbb{Z}^3$ defined previously is additive,
and vanishes on each protoloop of tiles. \qed
\end{proposition}

\begin{remark}\label{rem:gamma_0}
Given $T$, $T'$ two tiles in $\sigma^\infty(C)$, we set
$$\omega_0(T,T')= \omega_0(\gamma)$$
where $\gamma$ is any path of tiles joining $T$ to $T'$.
Indeed, if $\gamma$, $\gamma'$ are two such paths, 
Proposition~\ref{prop:omega} ensures that $\omega_0(\gamma)= \omega_0(\gamma')$
since $\omega_0$ vanishes on the loop of tiles $\gamma'\gamma^{-1}$.
\end{remark}

The following proposition will be used afterwards.

\begin{proposition}\label{prop:M-omega}
Let $T,T'$ be tiles in $\sigma^\infty(C)$. Then
\begin{equation}\label{eq:M-omega}
\omega_0\big( \bt(T),\bt(T')  \big) = \mathbf M_s^{-1} \omega_0(T,T').
\end{equation}
\end{proposition}

\begin{proof}
Proposition~\ref{prop:omega} ensures that $\omega_0$ is additive.
It is thus sufficient to prove (\ref{eq:M-omega}) for adjacent tiles $T, T'$.
Moreover, $\omega_0(T,T')$ only depends on the protopath $(T,T')$ defined by the adjacent tiles $T, T'$.
There is only a finite number of protopaths of length 2 to consider: those which are listed on Figure~\ref{fig:omega_two}.
We detail below an instance of the kind of easy computations that have to be carried on.
Suppose that $(T,T')$ is the following protopath:
\[\myvcenter{\begin{tikzpicture}[x={(-0.216506cm,-0.125000cm)}, y={(0.216506cm,-0.125000cm)}, z={(0.000000cm,0.250000cm)}]
\tuileC{C}{(0,0,0)}{white}{}{}{}{}{}{}{}{}{}{}
\tuileC{C}{(-1,2,-1)}{black!20}{}{}{}{}{}{}{}{}{}{}
\draw[->, draw=red, thick] (0,0.5,0) -- (-1,2,-1);
\end{tikzpicture}}\;.\]
By definition of $\omega_0$ (Figure~\ref{fig:omega_two}) we have $\omega_0(T,T') = (-1,2,-1)$.
We compute $\omega_0( \bt(T),\bt(T')$ by inspecting the image of $(T,T')$ by $\sigma$:
\[\myvcenter{\begin{tikzpicture}[x={(-0.216506cm,-0.125000cm)}, y={(0.216506cm,-0.125000cm)}, z={(0.000000cm,0.250000cm)}]
\tuileC{C}{(0,0,0)}{white}{}{}{}{}{}{}{}{}{}{}
\tuileA{A}{(0,0,0)}{white}{}{}{}{}{}{}
\tuileB{B}{(0,0,0)}{white}{}{}{}{}{}{}
\tuileC{C}{(2,-1,-2)}{black!20}{}{}{}{}{}{}{}{}{}{}
\tuileA{A}{(2,-1,-2)}{black!20}{}{}{}{}{}{}
\tuileB{B}{(2,-1,-2)}{black!20}{}{}{}{}{}{}
\draw[->, draw=red, thick] (0,0,0) -- (1,-1,0);
\draw[->, draw=red, thick] (1,-1,0) -- (2,-1,-1.5);
\end{tikzpicture}}\;.\]
By choosing the path of length three above and by reading Figure~\ref{fig:omega_two},
we compute
\[
\omega_0(\bt(T),\bt(T')) = (1,-1,0) + (1,0,-2) = (2,-1,-2),
\]
so the proposition holds in this case because $\mathbf M_s^{-1} (-1,2,-1) = (2,-1,-2)$.
\end{proof}

\subsection{From the topological patches to the stepped surface}\label{sec:stepped}
\label{sec:patchtostepped}

Let $(P,T)$ a pointed patch formed by a patch $P$ of $\sigma^{\infty}(C)$ and a tile $T$ of $\sigma^{\infty}(C)$. We are going to associate to $(P,T)$ a multiset of facets $\varphi_0(P,T) \in\mathcal M$, see Section~\ref{sec:multiset}.
When $P$ is equal to the tile $T$, we simply denote the pointed patch $(T,T)$ by $T$.

\medskip

Let $ \mathcal{T}_\sigma$ denote the set of tiles of $\sigma^\infty(C)$, as in Section~\ref{sec:inflation}.
First, we define a map $\Phi: \mathcal{T}_\sigma \rightarrow \mathcal M$ so that two tiles of the same type have the same image. This map $\Phi$ is defined by setting:
\begin{align*}
\Phi\big(
\myvcenter{\begin{tikzpicture}
[x={(-0.216506cm,-0.125000cm)}, y={(0.216506cm,-0.125000cm)}, z={(0.000000cm,0.250000cm)}]
\tuileA{A}{(0,0,0)}{black!20}{}{}{}{}{}{}
\end{tikzpicture}}
\big)
    & = \myvcenter{\begin{tikzpicture}
[x={(-0.216506cm,-0.125000cm)}, y={(0.216506cm,-0.125000cm)}, z={(0.000000cm,0.250000cm)}]
\definecolor{facecolor}{rgb}{0.800,0.800,0.800}
\fill[thick, fill=facecolor, draw=black, shift={(0,0,0)}]
(0, 0, 0) -- (0, 0, 1) -- (1, 0, 1) -- (1, 0, 0) -- cycle;
\fill[thick, fill=facecolor, draw=black, shift={(0,0,0)}]
(0, 0, 0) -- (0, 1, 0) -- (0, 1, 1) -- (0, 0, 1) -- cycle;
\fill[thick, fill=facecolor, draw=black, shift={(0,0,0)}]
(0, 0, 0) -- (1, 0, 0) -- (1, 1, 0) -- (0, 1, 0) -- cycle;
\fill (0,0,0) circle (2.5pt);
\end{tikzpicture}}
    = \mathbf E([\mathbf 0,1]^*),
&
\Phi\Big(
\myvcenter{\begin{tikzpicture}
[x={(-0.216506cm,-0.125000cm)}, y={(0.216506cm,-0.125000cm)}, z={(0.000000cm,0.250000cm)}]
\tuileB{B}{(0,0,0)}{black!20}{}{}{}{}{}{}
\end{tikzpicture}}
\Big)
    & = \myvcenter{\begin{tikzpicture}
[x={(-0.216506cm,-0.125000cm)}, y={(0.216506cm,-0.125000cm)}, z={(0.000000cm,0.250000cm)}]
\definecolor{facecolor}{rgb}{0.800,0.800,0.800}
\fill[thick, fill=facecolor, draw=black, shift={(0,0,1)}]
(0, 0, 0) -- (0, 0, 1) -- (1, 0, 1) -- (1, 0, 0) -- cycle;
\fill[thick, fill=facecolor, draw=black, shift={(0,0,0)}]
(0, 0, 0) -- (0, 0, 1) -- (1, 0, 1) -- (1, 0, 0) -- cycle;
\fill[thick, fill=facecolor, draw=black, shift={(0,0,0)}]
(0, 0, 0) -- (0, 1, 0) -- (0, 1, 1) -- (0, 0, 1) -- cycle;
\fill[thick, fill=facecolor, draw=black, shift={(0,0,1)}]
(0, 0, 0) -- (0, 1, 0) -- (0, 1, 1) -- (0, 0, 1) -- cycle;
\fill[thick, fill=facecolor, draw=black, shift={(0,0,0)}]
(0, 0, 0) -- (1, 0, 0) -- (1, 1, 0) -- (0, 1, 0) -- cycle;
\fill (0,0,0) circle (2.5pt);
\end{tikzpicture}}
    = \mathbf E^2([\mathbf 0,1]^*),
&
\Phi\Big(\,
\myvcenter{\begin{tikzpicture}
[x={(-0.216506cm,-0.125000cm)}, y={(0.216506cm,-0.125000cm)}, z={(0.000000cm,0.250000cm)}]
\tuileC{C}{(0,0,0)}{black!20}{}{}{}{}{}{}{}{}{}{}
\end{tikzpicture}}\,
\Big)
    & = \myvcenter{\begin{tikzpicture}
[x={(-0.216506cm,-0.125000cm)}, y={(0.216506cm,-0.125000cm)}, z={(0.000000cm,0.250000cm)}]
\definecolor{facecolor}{rgb}{0.800,0.800,0.800}
\fill[thick, fill=facecolor, draw=black, shift={(0,0,1)}]
(0, 0, 0) -- (0, 0, 1) -- (1, 0, 1) -- (1, 0, 0) -- cycle;
\fill[thick, fill=facecolor, draw=black, shift={(0,0,1)}]
(0, 0, 0) -- (0, 1, 0) -- (0, 1, 1) -- (0, 0, 1) -- cycle;
\fill[thick, fill=facecolor, draw=black, shift={(0,1,0)}]
(0, 0, 0) -- (0, 1, 0) -- (0, 1, 1) -- (0, 0, 1) -- cycle;
\fill[thick, fill=facecolor, draw=black, shift={(0,1,-1)}]
(0, 0, 0) -- (1, 0, 0) -- (1, 1, 0) -- (0, 1, 0) -- cycle;
\fill[thick, fill=facecolor, draw=black, shift={(0,1,-1)}]
(0, 0, 0) -- (0, 0, 1) -- (1, 0, 1) -- (1, 0, 0) -- cycle;
\fill[thick, fill=facecolor, draw=black, shift={(0,0,0)}]
(0, 0, 0) -- (0, 0, 1) -- (1, 0, 1) -- (1, 0, 0) -- cycle;
\fill[thick, fill=facecolor, draw=black, shift={(0,1,-1)}]
(0, 0, 0) -- (0, 1, 0) -- (0, 1, 1) -- (0, 0, 1) -- cycle;
\fill[thick, fill=facecolor, draw=black, shift={(0,0,0)}]
(0, 0, 0) -- (1, 0, 0) -- (1, 1, 0) -- (0, 1, 0) -- cycle;
\fill[thick, fill=facecolor, draw=black, shift={(0,0,0)}]
(0, 0, 0) -- (0, 1, 0) -- (0, 1, 1) -- (0, 0, 1) -- cycle;
\fill (0,0,0) circle (2.5pt);
\end{tikzpicture}}
    = \mathbf E^3([\mathbf 0,1]^*).
\end{align*}
Alternatively: 
$$\Phi(A)=\mathbf E ^3([\mathbf 0,3]^*)+\mathbf e_3-\mathbf e_1,\quad \Phi(B)=\mathbf E^3([\mathbf 0,2]^*)+\mathbf e_2-\mathbf e_1,\quad \Phi(C)=\mathbf E^3([\mathbf 0,1]^*).$$

In the pictures representing multisets, the symbol $\bullet$ indicates the origin of $\mathbb R^3$. 
For instance the image of $A$ 
is the multiset $m:\mathbb Z^3\rightarrow\mathbb Z_{\geq 0}^3$ defined by
$$m(\mathbf y) = \begin{cases}
\mathbf 0 & \;\;\text{ if }\;\; \mathbf y\neq (1,0,-1) \\
(1,1,1) & \;\;\text{ if }\;\; \mathbf y= (1,0,-1).
\end{cases}$$

For a patch $P$ and a tile $T$, we consider
$T'$ another tile  and $\gamma$ a path of tiles from $T$ to $T'$. 
By Proposition \ref{prop:omega}, the vector $\omega_0(\gamma)$ only depends on $T$ and $T'$, and not on the choice of the path $\gamma$. Thus we denote it by $\omega_0(T,T')$.

\begin{definition}\label{def:phi_0}
Let $P$ be a patch of $\sigma^\infty(C)$, and let $T$ be a tile of $\sigma^\infty(C)$.
The multiset of facets $\varphi_0(P,T)\in\mathcal{M}$ is defined by:
$$\varphi_0(P,T)=\displaystyle\sum_{T'\in P} \Big(\Phi(T')+\omega_0(T,T')\Big).$$
\end{definition}

\begin{remark}\label{rem:Phi varphi_0}
By definition, we notice that $\varphi_0(T,T)=\Phi(T)$ for every tile $T$ in $\sigma^\infty(C)$.
\end{remark}

The next two lemmas state useful properties of the map $\varphi_0$.
By definition of $\varphi_0$ and by additivity of $\omega_0$, we derive immediately the following lemma.

\begin{lemma}\label{lem:changement tuile de base}
Let $P$ be a patch of $\sigma^\infty(C)$, and let $T$, $T'$ be tiles of $\sigma^\infty(C)$. 
Then we have $\varphi_0(P,T)=\varphi_0(P,T')+\omega_0(T,T')$.
\end{lemma}

Let $P_1$, $P_2$ be patches in $\sigma^\infty(C)$.
We denote by $P_1\cap P_2$ the (possibly empty) patch in $\sigma^\infty(C)$ made of tiles belonging to both $P_1$ and $P_2$: this is the standard definition of ``intersection of patches''.

\begin{lemma}
\label{lem:partition}
Let $P_1$, $P_2$ be patches of $\sigma^\infty(C)$, 

and let $T$ be a tile of $\sigma^\infty(C)$. 
\begin{itemize}
\item If $P_1$ and $P_2$ have no tile in common, then $\varphi_0(P_1\cup P_2,T)=\varphi_0(P_1,T)+\varphi_0(P_2,T)$.
\item In the general situation we have
$\varphi_0(P_1\cup P_2,T)=\varphi_0(P_1,T)+\varphi_0(P_2,T)-\varphi_0(P_1\cap P_2,T)$.
\end{itemize}
\end{lemma}
\begin{proof}
The second point is a direct consequence of the first one. By definition,
and because $P_1,P_2$ have no tile in common, we have
\begin{align*}
\varphi_0(P_1\cup P_2,T)
    &= \sum_{T'\in P_1\cup P_2} \Big(\Phi(T')+\omega_0(T,T')\Big) \\
    &= \sum_{T'\in P_1} \Big(\Phi(T')+\omega_0(T,T')\Big)+\displaystyle\sum_{T'\in P_2} \Big(\Phi(T')+\omega_0(T,T')\Big) \\
    &= \varphi_0(P_1\cup P_2,T)=\varphi_0(P_1,T)+\varphi_0( P_2,T).
\end{align*}
\end{proof}

\subsection{Commutation between $\sigma, \mathbf E, \varphi_0$ }

\begin{proposition}\label{propcle}
\label{prop:phi_commute}
Let $P$ be a simply connected patch of $\sigma^\infty(C)$ and $T$ a tile.
We have
\begin{equation}\label{eq:commutation}
\varphi_0\circ\hat\sigma(P,T) = \mathbf E\circ\varphi_0(P,T).
\end{equation}
\end{proposition}
In the previous formula, we formally consider that the map $\EOSS$ acts on multisets.
\begin{proof}
A direct verification shows that for every tile $T$ in $\sigma^\infty(C)$, we have
\begin{equation}\label{eq:Phi}
\varphi_0\circ\hat\sigma(T,T) = \mathbf E\circ\varphi_0(T,T),
\end{equation}
as detailed on the the following diagrams (where the base tile of a patch is the white tile).
\begin{center}
\begin{tikzpicture}
%
\draw (0,0) node (a) {%
\myvcenter{\begin{tikzpicture}
[x={(-0.216506cm,-0.125000cm)}, y={(0.216506cm,-0.125000cm)}, z={(0.000000cm,0.250000cm)}]
\tuileA{A}{(0,0,0)}{white
}{}{}{}{}{}{}
\end{tikzpicture}}%
};
\draw (2,0) node (b) {%
\myvcenter{\begin{tikzpicture}
[x={(-0.216506cm,-0.125000cm)}, y={(0.216506cm,-0.125000cm)}, z={(0.000000cm,0.250000cm)}]
\tuileB{B}{(0,0,0)}{white
}{}{}{}{}{}{}
\end{tikzpicture}}%
};
\draw (0,-2) node (c) {%
\myvcenter{\begin{tikzpicture}
[x={(-0.216506cm,-0.125000cm)}, y={(0.216506cm,-0.125000cm)}, z={(0.000000cm,0.250000cm)}]
\definecolor{facecolor}{rgb}{0.800,0.800,0.800}
\fill[thick, fill=facecolor, draw=black, shift={(0,0,0)}]
(0, 0, 0) -- (0, 0, 1) -- (1, 0, 1) -- (1, 0, 0) -- cycle;
\fill[thick, fill=facecolor, draw=black, shift={(0,0,0)}]
(0, 0, 0) -- (0, 1, 0) -- (0, 1, 1) -- (0, 0, 1) -- cycle;
\fill[thick, fill=facecolor, draw=black, shift={(0,0,0)}]
(0, 0, 0) -- (1, 0, 0) -- (1, 1, 0) -- (0, 1, 0) -- cycle;
\fill (0,0,0) circle (2.5pt);
\end{tikzpicture}}%
};
\draw (2,-2) node (d) {%
\myvcenter{\begin{tikzpicture}
[x={(-0.216506cm,-0.125000cm)}, y={(0.216506cm,-0.125000cm)}, z={(0.000000cm,0.250000cm)}]
\definecolor{facecolor}{rgb}{0.800,0.800,0.800}
\fill[thick, fill=facecolor, draw=black, shift={(0,0,1)}]
(0, 0, 0) -- (0, 0, 1) -- (1, 0, 1) -- (1, 0, 0) -- cycle;
\fill[thick, fill=facecolor, draw=black, shift={(0,0,0)}]
(0, 0, 0) -- (0, 0, 1) -- (1, 0, 1) -- (1, 0, 0) -- cycle;
\fill[thick, fill=facecolor, draw=black, shift={(0,0,0)}]
(0, 0, 0) -- (0, 1, 0) -- (0, 1, 1) -- (0, 0, 1) -- cycle;
\fill[thick, fill=facecolor, draw=black, shift={(0,0,1)}]
(0, 0, 0) -- (0, 1, 0) -- (0, 1, 1) -- (0, 0, 1) -- cycle;
\fill[thick, fill=facecolor, draw=black, shift={(0,0,0)}]
(0, 0, 0) -- (1, 0, 0) -- (1, 1, 0) -- (0, 1, 0) -- cycle;
\fill (0,0,0) circle (2.5pt);
\end{tikzpicture}}%
};
\draw[|->, decoration={markings,mark=at position 1 with {\arrow[scale=1.7]{>}}}, postaction={decorate}, shorten >=0.4pt ]
(a) -- node [above] {$\hat\sigma$} (b);
\draw[|->, decoration={markings,mark=at position 1 with {\arrow[scale=1.7]{>}}}, postaction={decorate}, shorten >=0.4pt ]
(c) -- node [above] {$\mathbf E$} (d);
\draw[|->, decoration={markings,mark=at position 1 with {\arrow[scale=1.7]{>}}}, postaction={decorate}, shorten >=0.4pt ]
(a) -- node [left] {$\varphi_0$
} (c);
\draw[|->, decoration={markings,mark=at position 1 with {\arrow[scale=1.7]{>}}}, postaction={decorate}, shorten >=0.4pt ]
(b) -- node [right] {$\varphi_0$
} (d);
\end{tikzpicture}
\hfil
\begin{tikzpicture}
\draw (0,0) node (a) {%
\myvcenter{\begin{tikzpicture}
[x={(-0.216506cm,-0.125000cm)}, y={(0.216506cm,-0.125000cm)}, z={(0.000000cm,0.250000cm)}]
\tuileB{B}{(0,0,0)}{white
}{}{}{}{}{}{}
\end{tikzpicture}}%
};
\draw (2,0) node (b) {%
\myvcenter{\begin{tikzpicture}
[x={(-0.216506cm,-0.125000cm)}, y={(0.216506cm,-0.125000cm)}, z={(0.000000cm,0.250000cm)}]
\tuileC{C}{(0,0,0)}{white
}{}{}{}{}{}{}{}{}{}{}
\end{tikzpicture}}%
};
\draw (0,-2) node (c) {%
\myvcenter{\begin{tikzpicture}
[x={(-0.216506cm,-0.125000cm)}, y={(0.216506cm,-0.125000cm)}, z={(0.000000cm,0.250000cm)}]
\definecolor{facecolor}{rgb}{0.800,0.800,0.800}
\fill[thick, fill=facecolor, draw=black, shift={(0,0,1)}]
(0, 0, 0) -- (0, 0, 1) -- (1, 0, 1) -- (1, 0, 0) -- cycle;
\fill[thick, fill=facecolor, draw=black, shift={(0,0,0)}]
(0, 0, 0) -- (0, 0, 1) -- (1, 0, 1) -- (1, 0, 0) -- cycle;
\fill[thick, fill=facecolor, draw=black, shift={(0,0,0)}]
(0, 0, 0) -- (0, 1, 0) -- (0, 1, 1) -- (0, 0, 1) -- cycle;
\fill[thick, fill=facecolor, draw=black, shift={(0,0,1)}]
(0, 0, 0) -- (0, 1, 0) -- (0, 1, 1) -- (0, 0, 1) -- cycle;
\fill[thick, fill=facecolor, draw=black, shift={(0,0,0)}]
(0, 0, 0) -- (1, 0, 0) -- (1, 1, 0) -- (0, 1, 0) -- cycle;
\fill (0,0,0) circle (2.5pt);
\end{tikzpicture}}%
};
\draw (2,-2) node (d) {%
\myvcenter{\begin{tikzpicture}
[x={(-0.216506cm,-0.125000cm)}, y={(0.216506cm,-0.125000cm)}, z={(0.000000cm,0.250000cm)}]
\definecolor{facecolor}{rgb}{0.800,0.800,0.800}
\fill[thick, fill=facecolor, draw=black, shift={(0,0,1)}]
(0, 0, 0) -- (0, 0, 1) -- (1, 0, 1) -- (1, 0, 0) -- cycle;
\fill[thick, fill=facecolor, draw=black, shift={(0,0,1)}]
(0, 0, 0) -- (0, 1, 0) -- (0, 1, 1) -- (0, 0, 1) -- cycle;
\fill[thick, fill=facecolor, draw=black, shift={(0,1,0)}]
(0, 0, 0) -- (0, 1, 0) -- (0, 1, 1) -- (0, 0, 1) -- cycle;
\fill[thick, fill=facecolor, draw=black, shift={(0,1,-1)}]
(0, 0, 0) -- (1, 0, 0) -- (1, 1, 0) -- (0, 1, 0) -- cycle;
\fill[thick, fill=facecolor, draw=black, shift={(0,1,-1)}]
(0, 0, 0) -- (0, 0, 1) -- (1, 0, 1) -- (1, 0, 0) -- cycle;
\fill[thick, fill=facecolor, draw=black, shift={(0,0,0)}]
(0, 0, 0) -- (0, 0, 1) -- (1, 0, 1) -- (1, 0, 0) -- cycle;
\fill[thick, fill=facecolor, draw=black, shift={(0,1,-1)}]
(0, 0, 0) -- (0, 1, 0) -- (0, 1, 1) -- (0, 0, 1) -- cycle;
\fill[thick, fill=facecolor, draw=black, shift={(0,0,0)}]
(0, 0, 0) -- (1, 0, 0) -- (1, 1, 0) -- (0, 1, 0) -- cycle;
\fill[thick, fill=facecolor, draw=black, shift={(0,0,0)}]
(0, 0, 0) -- (0, 1, 0) -- (0, 1, 1) -- (0, 0, 1) -- cycle;
\fill (0,0,0) circle (2.5pt);
\end{tikzpicture}}%
};
\draw[|->, decoration={markings,mark=at position 1 with {\arrow[scale=1.7]{>}}}, postaction={decorate}, shorten >=0.4pt ]
(a) -- node [above] {$\hat\sigma$} (b);
\draw[|->, decoration={markings,mark=at position 1 with {\arrow[scale=1.7]{>}}}, postaction={decorate}, shorten >=0.4pt ]
(c) -- node [above] {$\mathbf E$} (d);
\draw[|->, decoration={markings,mark=at position 1 with {\arrow[scale=1.7]{>}}}, postaction={decorate}, shorten >=0.4pt ]
(a) -- node [left] {$\varphi_0$
} (c);
\draw[|->, decoration={markings,mark=at position 1 with {\arrow[scale=1.7]{>}}}, postaction={decorate}, shorten >=0.4pt ]
(b) -- node [right] {$\varphi_0$
} (d);
\end{tikzpicture}
\hfil
\begin{tikzpicture}
\draw (0,0) node (a) {%
\myvcenter{\begin{tikzpicture}
[x={(-0.216506cm,-0.125000cm)}, y={(0.216506cm,-0.125000cm)}, z={(0.000000cm,0.250000cm)}]
\tuileC{C}{(0,0,0)}{white
}{}{}{}{}{}{}{}{}{}{}
\end{tikzpicture}}%
};
\draw (2,0) node (b) {%
\myvcenter{\begin{tikzpicture}
[x={(-0.216506cm,-0.125000cm)}, y={(0.216506cm,-0.125000cm)}, z={(0.000000cm,0.250000cm)}]
\tuileA{A}{(0,0,0)}{white
}{}{}{}{}{}{}
\tuileB{B}{(0,0,0)}{black!20}{}{}{}{}{}{}
\tuileC{C}{(0,0,0)}{black!20}{}{}{}{}{}{}{}{}{}{}
\end{tikzpicture}}%
};
\draw (0,-2) node (c) {%
\myvcenter{\begin{tikzpicture}
[x={(-0.216506cm,-0.125000cm)}, y={(0.216506cm,-0.125000cm)}, z={(0.000000cm,0.250000cm)}]
\definecolor{facecolor}{rgb}{0.800,0.800,0.800}
\fill[thick, fill=facecolor, draw=black, shift={(0,0,1)}]
(0, 0, 0) -- (0, 0, 1) -- (1, 0, 1) -- (1, 0, 0) -- cycle;
\fill[thick, fill=facecolor, draw=black, shift={(0,0,1)}]
(0, 0, 0) -- (0, 1, 0) -- (0, 1, 1) -- (0, 0, 1) -- cycle;
\fill[thick, fill=facecolor, draw=black, shift={(0,1,0)}]
(0, 0, 0) -- (0, 1, 0) -- (0, 1, 1) -- (0, 0, 1) -- cycle;
\fill[thick, fill=facecolor, draw=black, shift={(0,1,-1)}]
(0, 0, 0) -- (1, 0, 0) -- (1, 1, 0) -- (0, 1, 0) -- cycle;
\fill[thick, fill=facecolor, draw=black, shift={(0,1,-1)}]
(0, 0, 0) -- (0, 0, 1) -- (1, 0, 1) -- (1, 0, 0) -- cycle;
\fill[thick, fill=facecolor, draw=black, shift={(0,0,0)}]
(0, 0, 0) -- (0, 0, 1) -- (1, 0, 1) -- (1, 0, 0) -- cycle;
\fill[thick, fill=facecolor, draw=black, shift={(0,1,-1)}]
(0, 0, 0) -- (0, 1, 0) -- (0, 1, 1) -- (0, 0, 1) -- cycle;
\fill[thick, fill=facecolor, draw=black, shift={(0,0,0)}]
(0, 0, 0) -- (1, 0, 0) -- (1, 1, 0) -- (0, 1, 0) -- cycle;
\fill[thick, fill=facecolor, draw=black, shift={(0,0,0)}]
(0, 0, 0) -- (0, 1, 0) -- (0, 1, 1) -- (0, 0, 1) -- cycle;
\fill (0,0,0) circle (2.5pt);
\end{tikzpicture}}%
};
\draw (2,-2) node (d) {%
\myvcenter{\begin{tikzpicture}
[x={(-0.216506cm,-0.125000cm)}, y={(0.216506cm,-0.125000cm)}, z={(0.000000cm,0.250000cm)}]
\definecolor{facecolor}{rgb}{0.800,0.800,0.800}
\fill[thick, fill=facecolor, draw=black, shift={(1,0,-1)}]
(0, 0, 0) -- (0, 0, 1) -- (1, 0, 1) -- (1, 0, 0) -- cycle;
\fill[thick, fill=facecolor, draw=black, shift={(1,0,-1)}]
(0, 0, 0) -- (0, 1, 0) -- (0, 1, 1) -- (0, 0, 1) -- cycle;
\fill[thick, fill=facecolor, draw=black, shift={(1,0,-1)}]
(0, 0, 0) -- (1, 0, 0) -- (1, 1, 0) -- (0, 1, 0) -- cycle;
\fill[thick, fill=facecolor, draw=black, shift={(1,-1,0)}]
(0, 0, 0) -- (0, 1, 0) -- (0, 1, 1) -- (0, 0, 1) -- cycle;
\fill[thick, fill=facecolor, draw=black, shift={(1,-1,1)}]
(0, 0, 0) -- (0, 1, 0) -- (0, 1, 1) -- (0, 0, 1) -- cycle;
\fill[thick, fill=facecolor, draw=black, shift={(1,-1,0)}]
(0, 0, 0) -- (0, 0, 1) -- (1, 0, 1) -- (1, 0, 0) -- cycle;
\fill[thick, fill=facecolor, draw=black, shift={(1,-1,1)}]
(0, 0, 0) -- (0, 0, 1) -- (1, 0, 1) -- (1, 0, 0) -- cycle;
\fill[thick, fill=facecolor, draw=black, shift={(1,-1,0)}]
(0, 0, 0) -- (1, 0, 0) -- (1, 1, 0) -- (0, 1, 0) -- cycle;
\fill[thick, fill=facecolor, draw=black, shift={(0,0,1)}]
(0, 0, 0) -- (0, 0, 1) -- (1, 0, 1) -- (1, 0, 0) -- cycle;
\fill[thick, fill=facecolor, draw=black, shift={(0,0,1)}]
(0, 0, 0) -- (0, 1, 0) -- (0, 1, 1) -- (0, 0, 1) -- cycle;
\fill[thick, fill=facecolor, draw=black, shift={(0,1,0)}]
(0, 0, 0) -- (0, 1, 0) -- (0, 1, 1) -- (0, 0, 1) -- cycle;
\fill[thick, fill=facecolor, draw=black, shift={(0,1,-1)}]
(0, 0, 0) -- (1, 0, 0) -- (1, 1, 0) -- (0, 1, 0) -- cycle;
\fill[thick, fill=facecolor, draw=black, shift={(0,1,-1)}]
(0, 0, 0) -- (0, 0, 1) -- (1, 0, 1) -- (1, 0, 0) -- cycle;
\fill[thick, fill=facecolor, draw=black, shift={(0,0,0)}]
(0, 0, 0) -- (0, 0, 1) -- (1, 0, 1) -- (1, 0, 0) -- cycle;
\fill[thick, fill=facecolor, draw=black, shift={(0,1,-1)}]
(0, 0, 0) -- (0, 1, 0) -- (0, 1, 1) -- (0, 0, 1) -- cycle;
\fill[thick, fill=facecolor, draw=black, shift={(0,0,0)}]
(0, 0, 0) -- (1, 0, 0) -- (1, 1, 0) -- (0, 1, 0) -- cycle;
\fill[thick, fill=facecolor, draw=black, shift={(0,0,0)}]
(0, 0, 0) -- (0, 1, 0) -- (0, 1, 1) -- (0, 0, 1) -- cycle;
\fill (0,0,0) circle (2.5pt);
\end{tikzpicture}}%
};
\draw[|->, decoration={markings,mark=at position 1 with {\arrow[scale=1.7]{>}}}, postaction={decorate}, shorten >=0.4pt ]
(a) -- node [above] {$\hat\sigma$} (b);
\draw[|->, decoration={markings,mark=at position 1 with {\arrow[scale=1.7]{>}}}, postaction={decorate}, shorten >=0.4pt ]
(c) -- node [above] {$\mathbf E$} (d);
\draw[|->, decoration={markings,mark=at position 1 with {\arrow[scale=1.7]{>}}}, postaction={decorate}, shorten >=0.4pt ]
(a) -- node [left] {$\varphi_0$
} (c);
\draw[|->, decoration={markings,mark=at position 1 with {\arrow[scale=1.7]{>}}}, postaction={decorate}, shorten >=0.4pt ]
(b) -- node [right] {$\varphi_0$
} (d);
\end{tikzpicture}
\end{center}

We now establish the relation~(\ref{eq:commutation}) when $P$ is a tile $T'$.
Let $T,T'$ be tiles of $\sigma^\infty(C)$. 
Recall that, by definition of the pointed substitution $\hat\sigma$,
we have $\hat\sigma(T',T) = (\sigma(T'),\bt(T))$.
By Lemma~\ref{lem:changement tuile de base} and Proposition~\ref{prop:M-omega},
we have
\begin{align*}
\varphi_0(\sigma(T'),\bt(T)) &= \varphi_0(\sigma(T'),\bt(T')) + \omega_0(\bt(T),\bt(T')) \\
    &= \varphi_0(\hat\sigma(T',T')) + \mathbf M_s^{-1} \omega_0(T,T').
\end{align*}
Using relation (\ref{eq:Phi}) and Equation~\ref{equE}, we get that:
\begin{align*}
\varphi_0(\hat\sigma(T',T)) &= \EOSS (\varphi_0(T',T')) +  \mathbf M_s^{-1} \omega_0(T,T') \\  
  &= \EOSS (\varphi_0(T',T') + \omega_0(T,T') ).
\end{align*}
Using again Lemma~\ref{lem:changement tuile de base}, we get:
\begin{equation}\label{eq:commutation'}
\varphi_0(\sigma(T'),\bt(T)) = \EOSS (\varphi_0(T',T) ).
\end{equation}

We now prove the relation~(\ref{eq:commutation}) in full generality.
Let $P$ be a patch in $\sigma^\infty(C)$ and let $T$ be a tile of $\sigma^\infty(C)$. 
Since 
$$\sigma(P)=\bigcup_{T''\in P}\sigma(T''),$$ 
Lemma~\ref{lem:partition} ensures that
\begin{align*}
\varphi_0(\hat\sigma(P,T)) & = \varphi_0(\sigma(P),\bt(T)) \\
  &= \displaystyle\sum_{T''\in P} \varphi_0( \sigma(T''),\bt(T) ) \\
  &= \displaystyle\sum_{T''\in P} \varphi_0( \hat\sigma(T'',T) ).
\end{align*}
We conclude by using relation~(\ref{eq:commutation'}) and the additivity of the map $\EOSS$:
\begin{align*}
\varphi_0(\hat\sigma(P,T)) &= \displaystyle\sum_{T''\in P} \EOSS\Big( \varphi_0(T'',T) \Big)\\  
  &= \displaystyle  \EOSS\left(\sum_{T''\in P} \varphi_0(T'',T) \right)\\  
  &= \displaystyle  \EOSS(\varphi_0(P,T) ).
\end{align*}
\end{proof}

\subsection{From $\sigma^\infty(C)$ to $\Sstep$}

\subsubsection{The map induced by $\varphi_0$}\label{subset:varphi_0}\label{sec:varphi_0}

Proposition \ref{propcle} implies that for every integer $n$,
$$\varphi_0\circ\hat\sigma^n(C,C)=\EOSS^n\circ\varphi_0(C,C)=\EOSS^n(\mathcal U).$$
In what follows, it is convenient to identify a multiset and its support as explained in Remark~\ref{rem-support}.
Since $\EOSS^n(\mathcal U)$ converges to the stepped surface $\Sstep=\bigcup_{n\in\mathbb N}\EOSS^n(\mathcal U)$ (see Proposition \ref{prop:ai2}) and since $\sigma^n(C)$ converges to $\sigma^\infty(C)$, the map $\varphi_0$ induces a map, still denoted by $\varphi_0$ such that:

\begin{itemize}
\item for every tile $T$ of $\sigma^\infty(C)$, $\varphi_0(T,C)$ is a subset of $\Sstep$,
\item $\Sstep=\displaystyle\bigcup_{T \text{ tile of } \sigma^\infty(C)}\varphi_0(T,C)$.
\end{itemize}

\begin{lemma}\label{lem:Psi}
For every tile $T$ of $\sigma^\infty(C)$, there exists a unique facet $[\mathbf x,i]^*$ in $\Sstep$ such that 
$$\varphi_0(T,C)=\EOSS^3([\mathbf x,i]^*).$$
\end{lemma}

\begin{proof}
First, we recall that Lemma 3 of \cite{Arn.Ito.01} states that the map $\EOSS$ is ``injective'': precisely, if $\EOSS([\mathbf x,i]^*)$ and $\EOSS([\mathbf x',i']^*)$ have a facet in common, then $\mathbf x=\mathbf x'$ and $i=i'$.
This provides the unicity of the facet $[\mathbf x,i]^*$ in the lemma (if it exists).

We know that $\varphi_0(T,C)=\Phi(T)+\omega_0(T,C)$. By definition of $\Phi$, there exists an integer $i$ such that $\varphi_0(T,C)=\EOSS^3([0,i]^*)+\mathbf u_i+\omega_0(T,C)$ with $$\mathbf u_1=\mathbf e_3-\mathbf e_1=\mathbf M_s^{-3} (-2\mathbf e_1-\mathbf e_2), \mathbf u_2=\mathbf e_2-\mathbf e_1=\mathbf M_s^{-3}(-\mathbf e_1), \mathbf u_3=\mathbf 0.$$
That is to say: 
\begin{equation*}
\varphi_0(T,C)=\EOSS^3([\mathbf x,i]^*)\quad \text{with}\quad 
\mathbf x=\mathbf M_s^3(\omega_0(T,c)+\mathbf u_i).
\end{equation*}

We claim that $[\mathbf x,i]^*$ lies in $\Sstep$. Indeed since $\Sstep=\EOSS^3(\Sstep)$, we know that there exists facets of $\Sstep$ which images by $\EOSS^3$ cover $\varphi_0(T,C)$. By Lemma 3 of \cite{Arn.Ito.01} again, we conclude that the 
facet $[\mathbf x,i]^*$ is lying in $\Sstep$.
\end{proof}

\subsubsection{The map $\Psi$}\label{subset:psi}

We are now in position to define a bijection $\Psi$ from the set of tiles of $\sigma^\infty(C)$ to the set of facets of $\Sstep$:
$$\Psi(T)=[\mathbf x,i]^*,\quad\text{where}\quad \varphi_0(T)=\EOSS^3([\mathbf x,i]^*).$$ 

Moreover, since $\Phi(C)=\EOSS^3([\mathbf 0,1]^*)$, $\Phi(B)=\EOSS^3([\mathbf 0,2]^*)$ and $\Phi(A)=\EOSS^3([\mathbf 0,3]^*),$
we see that 
$$\type(T)=A \Leftrightarrow \type(\Psi(T))=3,$$
$$\type(T)=B \Leftrightarrow \type(\Psi(T))=2,$$
$$\type(T)=C \Leftrightarrow \type(\Psi(T))=1.$$

We set $\theta(A)=3$, $\theta(B)=2$, $\theta(C)=1$. We summarize the previous discussion in the following proposition:

\begin{theorem}\label{theo:Psi}
The map $\Psi$ defined, for every tile $T$ of $\sigma^\infty(C)$, by:
$$\Psi(T)=[\mathbf M_s^3(\omega_0(T,C)+\mathbf u_{\type(T)}),\theta(\type(T))]^*$$
is a bijection from the set of tiles of $\sigma^\infty(C)$ to the set of facets of $\Sigma_{step}$.
\end{theorem}

\subsection{Link between two tilings $\Ttop$ and $\Tstep$}

\subsubsection{Theorem~\ref{theo:Psi} revisited}

We recall that, according to Proposition~\ref{prop:topotiling}, $\sigma^\infty(C)$ can be realized as the tiling $\Ttop$,
and according to Proposition~\ref{prop:ai2} the tiling $\Tstep$ is the ``image'' of the stepped surface $\Sstep$ by the projection $\pi_\beta:\Sstep\rightarrow \mathcal P$.
Hence, Theorem~\ref{theo:Psi} explains exactly how the two tilings $\Ttop$ and $\Tstep$ are related. In particular:
\begin{itemize}
\item The map $\Psi$ sends tiles of the same type to tiles of the same type.
\item In $\Ttop$, the way to locate a tile $T$ with respect to another tile $T'$, via the identification with $\sigma^\infty(C)$, is by using the position $\omega_0(T,T')$. In $\Tstep$, the corresponding tiles $\Psi(T)$ and $\Psi(T')$ then will differ from the vector $\pi_\beta(\omega_0(T,T'))$.
\end{itemize}

\subsubsection{Via the Delone set $\mathcal D=\cup_{i\in\{1,2,3\}}\mathcal D_i$}

We would like to explicit the link between $\Ttop$ and $\Tstep$ in the same spirit of what we explain in Remark~\ref{rem:Tfrac Ttop} and Remark~\ref{rem:D_i}.
For that, we first notice that, alternatively, $\sigma^\infty(C)$ can be geometrized as follow.
According to Lemma~\ref{lem:Psi}, the set $\{\pi_\beta ( \varphi_0(T,C) )  \;|\; T \text{ tile of } \sigma^\infty(C) \}$ tiles the plane $\mathcal P$, and the resulting tiling is a geometric realization of $\sigma^\infty(C)$.
The definition of $\Psi$ in Section~\ref{sec:stepped} gives us a base point in $\Phi(A)$, $\Phi(B)$ and $\Phi(C)$, and consequently, gives rise to a base point $\mathbf x_T$ in each $\varphi_0(T,C)$.
The set $\{\pi_\beta(\mathbf x_T) \;|\; T \text{ tile of } \sigma^\infty(C)\}$ is a Delone set in $\mathcal P$. However, it is not equal to the set $\mathcal D$ of Remark~\ref{rem:D_i}: 
indeed, it is equal to $\mathbf M_s^{-3} \mathcal D$, see the proof of Lemma~\ref{lem:Psi}.

This leads us to do what follows.
For each tile $T$ of $\sigma^\infty(C)$, we consider
$$T_\textsf{geo} = \pi_\beta ( \mathbf M_s^3 ( \varphi_0(T,C) ) ) \subset \mathcal P.$$
The set $\{T_\textsf{geo}  \;|\; T \text{ tile of } \sigma^\infty(C) \}$ tiles the plane $\mathcal P$. This tilling is again a geometric realization of $\sigma^\infty(C)$, and we denote it by $\mathcal T'_\textsf{top}$. 
By construction, $\pi_\beta(\mathbf M_s^3 \mathbf x_T)$ lies in $T_\textsf{geo}$, and the set $\{ \pi_\beta(\mathbf M_s^3 \mathbf x_T) \;|\; T \text{ tile of } \sigma^\infty(C) \}$ is precisely the Delone set $\mathcal D$. 
Moreover, Theorem~\ref{theo:Psi} ensures that the subset consisting of vectors $\pi_\beta(\mathbf M_s^3 \mathbf x_T)$ with $\theta(\type(T))=i$ ($i\in\{1,2,3\}$) is precisely the set $\mathcal D_i$ of  Remark~\ref{rem:D_i}.

To sum up: 
the tiling $\mathcal T'_\textsf{top}$ is obtained by putting in $\mathcal P$ a tile of type $i$ (i.e. a translated image of $T_\textsf{geo}$ with $\theta(\type(T))=i$) at each vector in $\mathcal D_i$.
This explicits the strong relation between the two tilings $\mathcal T'_\textsf{top}$ and $\Tstep$, 
and thus also with $\Tfrac$ via Remark~\ref{rem:Tfrac Ttop}.

\subsubsection{The map $\Psi$ as a ``two-dimensional sliding block code''}

It is worth to remark that, by definition of $\Psi$, and because the position map $\omega_0$ is additive, there is an elementary way to rebuild $\Tstep$ from $\Ttop$ just by looking at local configurations of tiles. 
We give in Figure~\ref{fig:Psiblock} all the information needed for doing that.

\begin{figure}[h!]
\begin{align*}
\myvcenter{\begin{tikzpicture}[x={(-0.216506cm,-0.125000cm)}, y={(0.216506cm,-0.125000cm)}, z={(0.000000cm,0.250000cm)}]
\tuileA{A}{(0,0,0)}{white}{}{}{}{}{}{}
\tuileB{B}{(0,0,0)}{black!20}{}{}{}{}{}{}
\end{tikzpicture}}
& \ \leftrightarrow \
\myvcenter{\begin{tikzpicture}
[x={(-0.216506cm,-0.125000cm)}, y={(0.216506cm,-0.125000cm)}, z={(0.000000cm,0.250000cm)}]
\definecolor{facecolor}{rgb}{0.8,0.8,0.8}
\fill[thick, fill=facecolor, draw=black, shift={(0,0,0)}]
(0, 0, 0) -- (0, 0, 1) -- (1, 0, 1) -- (1, 0, 0) -- cycle;
\definecolor{facecolor}{rgb}{1,1,1}
\fill[thick, fill=facecolor, draw=black, shift={(0,0,0)}]
(0, 0, 0) -- (1, 0, 0) -- (1, 1, 0) -- (0, 1, 0) -- cycle;
\end{tikzpicture}}
&
\myvcenter{\begin{tikzpicture}[x={(-0.216506cm,-0.125000cm)}, y={(0.216506cm,-0.125000cm)}, z={(0.000000cm,0.250000cm)}]
\tuileA{A}{(0,0,0)}{white}{}{}{}{}{}{}
\tuileB{B}{(0,2,-3)}{black!20}{}{}{}{}{}{}
\end{tikzpicture}}
& \ \leftrightarrow \
\myvcenter{\begin{tikzpicture}
[x={(-0.216506cm,-0.125000cm)}, y={(0.216506cm,-0.125000cm)}, z={(0.000000cm,0.250000cm)}]
\definecolor{facecolor}{rgb}{0.8,0.8,0.8}
\fill[thick, fill=facecolor, draw=black, shift={(0,0,0)}]
(0, 0, 0) -- (0, 0, 1) -- (1, 0, 1) -- (1, 0, 0) -- cycle;
\definecolor{facecolor}{rgb}{1,1,1}
\fill[thick, fill=facecolor, draw=black, shift={(0,-1,1)}]
(0, 0, 0) -- (1, 0, 0) -- (1, 1, 0) -- (0, 1, 0) -- cycle;
\end{tikzpicture}}
&
\myvcenter{\begin{tikzpicture}[x={(-0.216506cm,-0.125000cm)}, y={(0.216506cm,-0.125000cm)}, z={(0.000000cm,0.250000cm)}]
\tuileA{A}{(0,0,0)}{white}{}{}{}{}{}{}
\tuileC{C}{(0,0,0)}{black!20}{}{}{}{}{}{}{}{}{}{}
\end{tikzpicture}}
& \ \leftrightarrow \
\myvcenter{\begin{tikzpicture}
[x={(-0.216506cm,-0.125000cm)}, y={(0.216506cm,-0.125000cm)}, z={(0.000000cm,0.250000cm)}]
\definecolor{facecolor}{rgb}{1,1,1}
\fill[thick, fill=facecolor, draw=black, shift={(0,0,0)}]
(0, 0, 0) -- (1, 0, 0) -- (1, 1, 0) -- (0, 1, 0) -- cycle;
\definecolor{facecolor}{rgb}{0.8,0.8,0.8}
\fill[thick, fill=facecolor, draw=black, shift={(0,0,0)}]
(0, 0, 0) -- (0, 1, 0) -- (0, 1, 1) -- (0, 0, 1) -- cycle;
\end{tikzpicture}}
&
\myvcenter{\begin{tikzpicture}[x={(-0.216506cm,-0.125000cm)}, y={(0.216506cm,-0.125000cm)}, z={(0.000000cm,0.250000cm)}]
\tuileA{A}{(0,0,0)}{white}{}{}{}{}{}{}
\tuileC{C}{(2,-1,-2)}{black!20}{}{}{}{}{}{}{}{}{}{}
\end{tikzpicture}}
& \ \leftrightarrow \
\myvcenter{\begin{tikzpicture}
[x={(-0.216506cm,-0.125000cm)}, y={(0.216506cm,-0.125000cm)}, z={(0.000000cm,0.250000cm)}]
\definecolor{facecolor}{rgb}{1,1,1}
\fill[thick, fill=facecolor, draw=black, shift={(-1,0,1)}]
(0, 0, 0) -- (1, 0, 0) -- (1, 1, 0) -- (0, 1, 0) -- cycle;
\definecolor{facecolor}{rgb}{0.8,0.8,0.8}
\fill[thick, fill=facecolor, draw=black, shift={(0,0,0)}]
(0, 0, 0) -- (0, 1, 0) -- (0, 1, 1) -- (0, 0, 1) -- cycle;
\end{tikzpicture}} \\
\myvcenter{\begin{tikzpicture}[x={(-0.216506cm,-0.125000cm)}, y={(0.216506cm,-0.125000cm)}, z={(0.000000cm,0.250000cm)}]
\tuileB{B}{(0,0,0)}{white}{}{}{}{}{}{}
\tuileC{C}{(0,0,0)}{black!20}{}{}{}{}{}{}{}{}{}{}
\end{tikzpicture}}
& \ \leftrightarrow \
\myvcenter{\begin{tikzpicture}
[x={(-0.216506cm,-0.125000cm)}, y={(0.216506cm,-0.125000cm)}, z={(0.000000cm,0.250000cm)}]
\definecolor{facecolor}{rgb}{1,1,1}
\fill[thick, fill=facecolor, draw=black, shift={(0,0,0)}]
(0, 0, 0) -- (0, 0, 1) -- (1, 0, 1) -- (1, 0, 0) -- cycle;
\definecolor{facecolor}{rgb}{0.8,0.8,0.8}
\fill[thick, fill=facecolor, draw=black, shift={(0,0,0)}]
(0, 0, 0) -- (0, 1, 0) -- (0, 1, 1) -- (0, 0, 1) -- cycle;
\end{tikzpicture}}
&
\myvcenter{\begin{tikzpicture}[x={(-0.216506cm,-0.125000cm)}, y={(0.216506cm,-0.125000cm)}, z={(0.000000cm,0.250000cm)}]
\tuileB{B}{(0,0,0)}{white}{}{}{}{}{}{}
\tuileC{C}{(2,-3,1)}{black!20}{}{}{}{}{}{}{}{}{}{}
\end{tikzpicture}}
& \ \leftrightarrow \
\myvcenter{\begin{tikzpicture}
[x={(-0.216506cm,-0.125000cm)}, y={(0.216506cm,-0.125000cm)}, z={(0.000000cm,0.250000cm)}]
\definecolor{facecolor}{rgb}{1,1,1}
\fill[thick, fill=facecolor, draw=black, shift={(0,0,0)}]
(0, 0, 0) -- (0, 0, 1) -- (1, 0, 1) -- (1, 0, 0) -- cycle;
\definecolor{facecolor}{rgb}{0.8,0.8,0.8}
\fill[thick, fill=facecolor, draw=black, shift={(1,-1,0)}]
(0, 0, 0) -- (0, 1, 0) -- (0, 1, 1) -- (0, 0, 1) -- cycle;
\end{tikzpicture}}
&
\myvcenter{\begin{tikzpicture}[x={(-0.216506cm,-0.125000cm)}, y={(0.216506cm,-0.125000cm)}, z={(0.000000cm,0.250000cm)}]
\tuileB{B}{(0,0,0)}{white}{}{}{}{}{}{}
\tuileB{B}{(-1,0,2)}{black!20}{}{}{}{}{}{}
\end{tikzpicture}}
& \ \leftrightarrow \
\myvcenter{\begin{tikzpicture}
[x={(-0.216506cm,-0.125000cm)}, y={(0.216506cm,-0.125000cm)}, z={(0.000000cm,0.250000cm)}]
\definecolor{facecolor}{rgb}{1,1,1}
\fill[thick, fill=facecolor, draw=black, shift={(0,0,0)}]
(0, 0, 0) -- (0, 0, 1) -- (1, 0, 1) -- (1, 0, 0) -- cycle;
\definecolor{facecolor}{rgb}{0.8,0.8,0.8}
\fill[thick, fill=facecolor, draw=black, shift={(0,0,1)}]
(0, 0, 0) -- (0, 0, 1) -- (1, 0, 1) -- (1, 0, 0) -- cycle;
\end{tikzpicture}}
&
\myvcenter{\begin{tikzpicture}[x={(-0.216506cm,-0.125000cm)}, y={(0.216506cm,-0.125000cm)}, z={(0.000000cm,0.250000cm)}]
\tuileC{C}{(0,0,0)}{white}{}{}{}{}{}{}{}{}{}{}
\tuileC{C}{(-1,2,-1)}{black!20}{}{}{}{}{}{}{}{}{}{}
\end{tikzpicture}}
& \ \leftrightarrow \
\myvcenter{\begin{tikzpicture}
[x={(-0.216506cm,-0.125000cm)}, y={(0.216506cm,-0.125000cm)}, z={(0.000000cm,0.250000cm)}]
\definecolor{facecolor}{rgb}{0.8,0.8,0.8}
\fill[thick, fill=facecolor, draw=black, shift={(0,1,0)}]
(0, 0, 0) -- (0, 1, 0) -- (0, 1, 1) -- (0, 0, 1) -- cycle;
\definecolor{facecolor}{rgb}{1,1,1}
\fill[thick, fill=facecolor, draw=black, shift={(0,0,0)}]
(0, 0, 0) -- (0, 1, 0) -- (0, 1, 1) -- (0, 0, 1) -- cycle;
\end{tikzpicture}} \\
\myvcenter{\begin{tikzpicture}[x={(-0.216506cm,-0.125000cm)}, y={(0.216506cm,-0.125000cm)}, z={(0.000000cm,0.250000cm)}]
\tuileC{C}{(0,0,0)}{white}{}{}{}{}{}{}{}{}{}{}
\tuileC{C}{(-1,0,2)}{black!20}{}{}{}{}{}{}{}{}{}{}
\end{tikzpicture}}
& \ \leftrightarrow \
\myvcenter{\begin{tikzpicture}
[x={(-0.216506cm,-0.125000cm)}, y={(0.216506cm,-0.125000cm)}, z={(0.000000cm,0.250000cm)}]
\definecolor{facecolor}{rgb}{1,1,1}
\fill[thick, fill=facecolor, draw=black, shift={(0,0,0)}]
(0, 0, 0) -- (0, 1, 0) -- (0, 1, 1) -- (0, 0, 1) -- cycle;
\definecolor{facecolor}{rgb}{0.8,0.8,0.8}
\fill[thick, fill=facecolor, draw=black, shift={(0,0,1)}]
(0, 0, 0) -- (0, 1, 0) -- (0, 1, 1) -- (0, 0, 1) -- cycle;
\end{tikzpicture}}
&
\myvcenter{\begin{tikzpicture}[x={(-0.216506cm,-0.125000cm)}, y={(0.216506cm,-0.125000cm)}, z={(0.000000cm,0.250000cm)}]
\tuileC{C}{(0,0,0)}{white}{}{}{}{}{}{}{}{}{}{}
\tuileC{C}{(0,-2,3)}{black!20}{}{}{}{}{}{}{}{}{}{}
\end{tikzpicture}}
& \ \leftrightarrow \
\myvcenter{\begin{tikzpicture}
[x={(-0.216506cm,-0.125000cm)}, y={(0.216506cm,-0.125000cm)}, z={(0.000000cm,0.250000cm)}]
\definecolor{facecolor}{rgb}{1,1,1}
\fill[thick, fill=facecolor, draw=black, shift={(0,0,0)}]
(0, 0, 0) -- (0, 1, 0) -- (0, 1, 1) -- (0, 0, 1) -- cycle;
\definecolor{facecolor}{rgb}{0.8,0.8,0.8}
\fill[thick, fill=facecolor, draw=black, shift={(0,-1,1)}]
(0, 0, 0) -- (0, 1, 0) -- (0, 1, 1) -- (0, 0, 1) -- cycle;
\end{tikzpicture}}
&
\myvcenter{\begin{tikzpicture}[x={(-0.216506cm,-0.125000cm)}, y={(0.216506cm,-0.125000cm)}, z={(0.000000cm,0.250000cm)}]
\tuileB{B}{(0,0,0)}{white}{}{}{}{}{}{}
\tuileC{C}{(2,-1,-2)}{black!20}{}{}{}{}{}{}{}{}{}{}
\end{tikzpicture}}
& \ \leftrightarrow \
\myvcenter{\begin{tikzpicture}
[x={(-0.216506cm,-0.125000cm)}, y={(0.216506cm,-0.125000cm)}, z={(0.000000cm,0.250000cm)}]
\definecolor{facecolor}{rgb}{1,1,1}
\fill[thick, fill=facecolor, draw=black, shift={(-1,0,1)}]
(0, 0, 0) -- (0, 0, 1) -- (1, 0, 1) -- (1, 0, 0) -- cycle;
\definecolor{facecolor}{rgb}{0.8,0.8,0.8}
\fill[thick, fill=facecolor, draw=black, shift={(0,0,0)}]
(0, 0, 0) -- (0, 1, 0) -- (0, 1, 1) -- (0, 0, 1) -- cycle;
\end{tikzpicture}}
&
\myvcenter{\begin{tikzpicture}[x={(-0.216506cm,-0.125000cm)}, y={(0.216506cm,-0.125000cm)}, z={(0.000000cm,0.250000cm)}]
\tuileB{B}{(0,0,0)}{white}{}{}{}{}{}{}
\tuileC{C}{(0,-2,3)}{black!20}{}{}{}{}{}{}{}{}{}{}
\end{tikzpicture}}
& \ \leftrightarrow \
\myvcenter{\begin{tikzpicture}
[x={(-0.216506cm,-0.125000cm)}, y={(0.216506cm,-0.125000cm)}, z={(0.000000cm,0.250000cm)}]
\definecolor{facecolor}{rgb}{1,1,1}
\fill[thick, fill=facecolor, draw=black, shift={(0,0,0)}]
(0, 0, 0) -- (0, 0, 1) -- (1, 0, 1) -- (1, 0, 0) -- cycle;
\definecolor{facecolor}{rgb}{0.8,0.8,0.8}
\fill[thick, fill=facecolor, draw=black, shift={(0,-1,1)}]
(0, 0, 0) -- (0, 1, 0) -- (0, 1, 1) -- (0, 0, 1) -- cycle;
\end{tikzpicture}} \\
\myvcenter{\begin{tikzpicture}[x={(-0.216506cm,-0.125000cm)}, y={(0.216506cm,-0.125000cm)}, z={(0.000000cm,0.250000cm)}]
\tuileB{B}{(0,0,0)}{white}{}{}{}{}{}{}
\tuileC{C}{(1,0,-2)}{black!20}{}{}{}{}{}{}{}{}{}{}
\end{tikzpicture}}
& \ \leftrightarrow \
\myvcenter{\begin{tikzpicture}
[x={(-0.216506cm,-0.125000cm)}, y={(0.216506cm,-0.125000cm)}, z={(0.000000cm,0.250000cm)}]
\definecolor{facecolor}{rgb}{1,1,1}
\fill[thick, fill=facecolor, draw=black, shift={(0,0,1)}]
(0, 0, 0) -- (0, 0, 1) -- (1, 0, 1) -- (1, 0, 0) -- cycle;
\definecolor{facecolor}{rgb}{0.8,0.8,0.8}
\fill[thick, fill=facecolor, draw=black, shift={(0,0,0)}]
(0, 0, 0) -- (0, 1, 0) -- (0, 1, 1) -- (0, 0, 1) -- cycle;
\end{tikzpicture}}
&
\myvcenter{\begin{tikzpicture}[x={(-0.216506cm,-0.125000cm)}, y={(0.216506cm,-0.125000cm)}, z={(0.000000cm,0.250000cm)}]
\tuileB{B}{(0,0,0)}{white}{}{}{}{}{}{}
\tuileC{C}{(1,-3,3)}{black!20}{}{}{}{}{}{}{}{}{}{}
\end{tikzpicture}}
& \ \leftrightarrow \
\myvcenter{\begin{tikzpicture}
[x={(-0.216506cm,-0.125000cm)}, y={(0.216506cm,-0.125000cm)}, z={(0.000000cm,0.250000cm)}]
\definecolor{facecolor}{rgb}{1,1,1}
\fill[thick, fill=facecolor, draw=black, shift={(-1,1,-1)}]
(0, 0, 0) -- (0, 0, 1) -- (1, 0, 1) -- (1, 0, 0) -- cycle;
\definecolor{facecolor}{rgb}{0.8,0.8,0.8}
\fill[thick, fill=facecolor, draw=black, shift={(0,0,0)}]
(0, 0, 0) -- (0, 1, 0) -- (0, 1, 1) -- (0, 0, 1) -- cycle;
\end{tikzpicture}}
\end{align*}
\caption{$\Psi$ and $\Psi^{-1}$ as a two-dimensional sliding block code.}
\label{fig:Psiblock}
\end{figure}
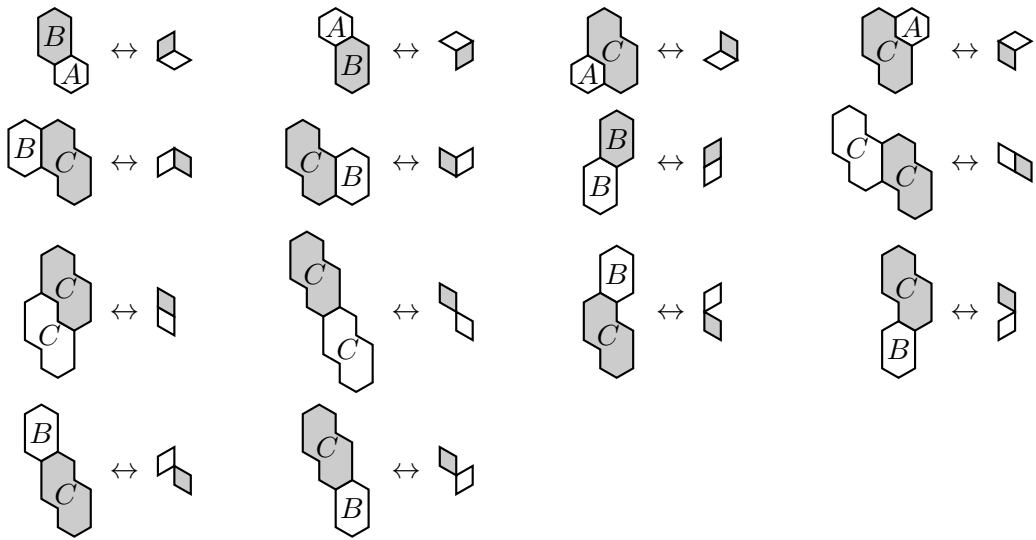

In practice, to construct $\Tstep$ from $\Ttop$, we can ignore the formula of Theorem~\ref{theo:Psi}
and simply use the recipe given in Figure~\ref{fig:Psiblock}. We choose a tile $T$ in $\Ttop$, and put in the plane a tile of $\Tstep$ of same type as $T$. Then we choose a tile $T'$ adjacent to $T$ and use Figure~\ref{fig:Psiblock} to place correctly the corresponding tile $\Psi(T')$ relatively to $\Psi(T)$. And we go on, inductively rebuilding $\Tstep$.

This can also be done to define $\Psi^{-1}$, see Figure~\ref{fig:Psiblock}.

This point of view on $\Psi$ remind us of the so-called \emph{sliding block code} in classical symbolic dynamics, see for instance~\cite{LM95}. In that spirit, $\psi$ could be called a \emph{two-dimensional sliding block code}.

\section{Concluding remarks}
\label{conclu}

\paragraph{Towards more general results}

The results presented in this article are specifically about the tilings associated
with the Tribonacci substitution and its associated Rauzy fractal tiling.
We have been able to get such results for some other examples of Pisot substitutions,
such as the one shown in Figure~\ref{fig:exother}.

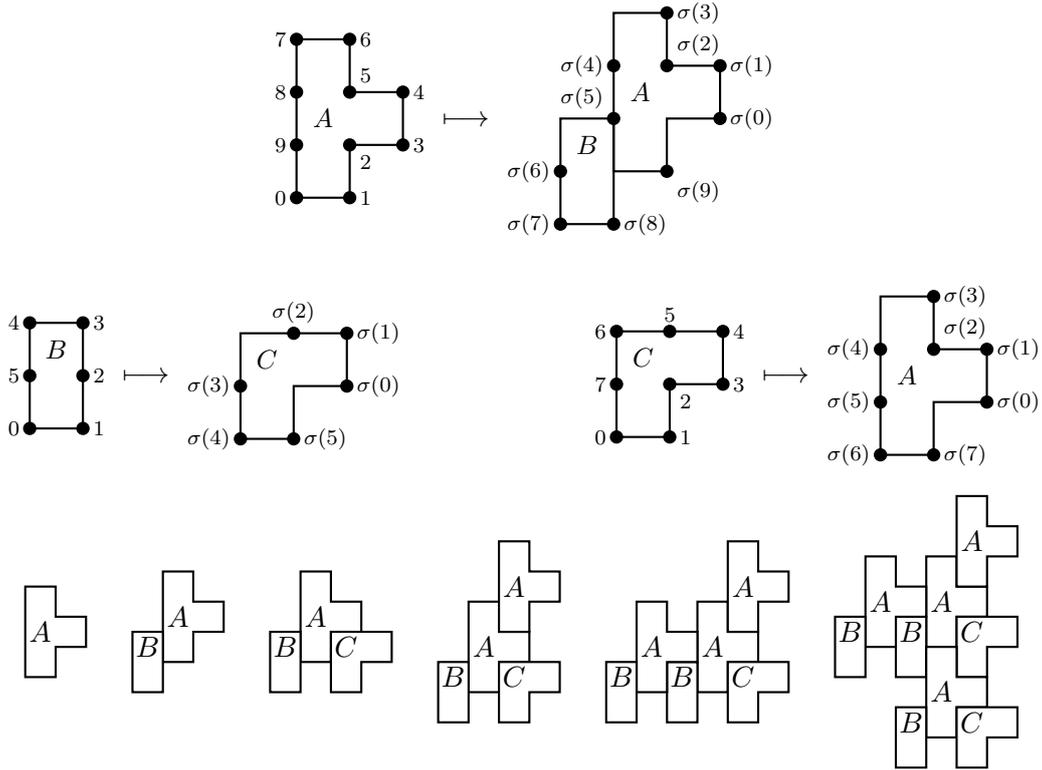
\begin{figure}[h!]
\centering
\[
\myvcenter{%
\begin{tikzpicture}[x={(7mm,0mm)},y={(0mm,7mm)}]
\draw[thick,black]
(0, 0) -- (1, 0) -- (1, 1) -- (2, 1) -- (2, 2) -- (1, 2) -- (1, 3) -- (0, 3) -- (0, 2) -- (0, 1) -- cycle;
\fill (0,0) circle (2.5pt) node[left] {\scriptsize$0$};
\fill (1,0) circle (2.5pt) node[right] {\scriptsize$1$};
\fill (1,1) circle (2.5pt) node[below right] {\scriptsize$2$};
\fill (2,1) circle (2.5pt) node[right] {\scriptsize$3$};
\fill (2,2) circle (2.5pt) node[right] {\scriptsize$4$};
\fill (1,2) circle (2.5pt) node[above right] {\scriptsize$5$};
\fill (1,3) circle (2.5pt) node[right] {\scriptsize$6$};
\fill (0,3) circle (2.5pt) node[left] {\scriptsize$7$};
\fill (0,2) circle (2.5pt) node[left] {\scriptsize$8$};
\fill (0,1) circle (2.5pt) node[left] {\scriptsize$9$};
\node at (0.5,1.5) {\small $A$};
\end{tikzpicture}}
    \longmapsto
\myvcenter{%
\begin{tikzpicture}[x={(7mm,0mm)},y={(0mm,7mm)}]
\draw[thick,black]
(0, 0) -- (1, 0) -- (1, 1) -- (2, 1) -- (2, 2) -- (1, 2) -- (1, 3) -- (0, 3) -- (0, 2) -- (0, 1) -- cycle;
\draw[thick,black]
(0, 0) -- (0, 1) -- (-1, 1) -- (-1, 0) -- (-1, -1) -- (0, -1) -- cycle;
\fill (1,0) circle (2.5pt) node[below right] {\scriptsize$\sigma(9)$};
\fill (2,1) circle (2.5pt) node[right] {\scriptsize$\sigma(0)$};
\fill (2,2) circle (2.5pt) node[right] {\scriptsize$\sigma(1)$};
\fill (1,2) circle (2.5pt) node[above right] {\scriptsize$\sigma(2)$};
\fill (1,3) circle (2.5pt) node[right] {\scriptsize$\sigma(3)$};
\fill (0,2) circle (2.5pt) node[left] {\scriptsize$\sigma(4)$};
\fill (0,1) circle (2.5pt) node[above left] {\scriptsize$\sigma(5)$};
\fill (-1,0) circle (2.5pt) node[left] {\scriptsize$\sigma(6)$};
\fill (-1,-1) circle (2.5pt) node[left] {\scriptsize$\sigma(7)$};
\fill (0,-1) circle (2.5pt) node[right] {\scriptsize$\sigma(8)$};
\node at (0.5,1.5) {\small $A$};
\node at (-0.5,0.5) {\small $B$};
\end{tikzpicture}}
\]
\[
\myvcenter{%
\begin{tikzpicture}[x={(7mm,0mm)},y={(0mm,7mm)}]
\draw[thick,black]
(0, 0) -- (1, 0) -- (1, 1) -- (1, 2) -- (0, 2) -- (0, 1) -- cycle;
\fill (0,0) circle (2.5pt) node[left] {\scriptsize$0$};
\fill (1,0) circle (2.5pt) node[right] {\scriptsize$1$};
\fill (1,1) circle (2.5pt) node[right] {\scriptsize$2$};
\fill (1,2) circle (2.5pt) node[right] {\scriptsize$3$};
\fill (0,2) circle (2.5pt) node[left] {\scriptsize$4$};
\fill (0,1) circle (2.5pt) node[left] {\scriptsize$5$};
\node at (0.5,1.5) {\small $B$};
\end{tikzpicture}}
    \longmapsto
\myvcenter{%
\begin{tikzpicture}[x={(7mm,0mm)},y={(0mm,7mm)}]
\draw[thick,black]
(0, 0) -- (1, 0) -- (1, 1) -- (2, 1) -- (2, 2) -- (1, 2) -- (0, 2) -- (0, 2) -- cycle;
\fill (0,0) circle (2.5pt) node[left] {\scriptsize$\sigma(4)$};
\fill (1,0) circle (2.5pt) node[right] {\scriptsize$\sigma(5)$};
\fill (2,1) circle (2.5pt) node[right] {\scriptsize$\sigma(0)$};
\fill (2,2) circle (2.5pt) node[right] {\scriptsize$\sigma(1)$};
\fill (1,2) circle (2.5pt) node[above] {\scriptsize$\sigma(2)$};
\fill (0,1) circle (2.5pt) node[left] {\scriptsize$\sigma(3)$};
\node at (0.5,1.5) {\small $C$};
\end{tikzpicture}}
\qquad
\qquad
\qquad
\myvcenter{%
\begin{tikzpicture}[x={(7mm,0mm)},y={(0mm,7mm)}]
\draw[thick,black]
(0, 0) -- (1, 0) -- (1, 1) -- (2, 1) -- (2, 2) -- (1, 2) -- (0, 2) -- (0, 2) -- cycle;
\fill (0,0) circle (2.5pt) node[left] {\scriptsize$0$};
\fill (1,0) circle (2.5pt) node[right] {\scriptsize$1$};
\fill (1,1) circle (2.5pt) node[below right] {\scriptsize$2$};
\fill (2,1) circle (2.5pt) node[right] {\scriptsize$3$};
\fill (2,2) circle (2.5pt) node[right] {\scriptsize$4$};
\fill (1,2) circle (2.5pt) node[above] {\scriptsize$5$};
\fill (0,2) circle (2.5pt) node[left] {\scriptsize$6$};
\fill (0,1) circle (2.5pt) node[left] {\scriptsize$7$};
\node at (0.5,1.5) {\small $C$};
\end{tikzpicture}}
    \longmapsto
\myvcenter{%
\begin{tikzpicture}[x={(7mm,0mm)},y={(0mm,7mm)}]
\draw[thick,black]
(0, 0) -- (1, 0) -- (1, 1) -- (2, 1) -- (2, 2) -- (1, 2) -- (1, 3) -- (0, 3) -- (0, 2) -- (0, 1) -- cycle;
\fill (0,0) circle (2.5pt) node[left] {\scriptsize$\sigma(6)$};
\fill (1,0) circle (2.5pt) node[right] {\scriptsize$\sigma(7)$};
\fill (2,1) circle (2.5pt) node[right] {\scriptsize$\sigma(0)$};
\fill (2,2) circle (2.5pt) node[right] {\scriptsize$\sigma(1)$};
\fill (1,2) circle (2.5pt) node[above right] {\scriptsize$\sigma(2)$};
\fill (1,3) circle (2.5pt) node[right] {\scriptsize$\sigma(3)$};
\fill (0,2) circle (2.5pt) node[left] {\scriptsize$\sigma(4)$};
\fill (0,1) circle (2.5pt) node[left] {\scriptsize$\sigma(5)$};
\node at (0.5,1.5) {\small $A$};
\end{tikzpicture}}
\]

\myvcenter{\begin{tikzpicture}[x={(4mm,0mm)},y={(0mm,4mm)}]
\sqtA{(0,0)}
\end{tikzpicture}}
\quad
\myvcenter{\begin{tikzpicture}[x={(4mm,0mm)},y={(0mm,4mm)}]
\sqtA{(0,0)}
\sqtB{(-1,-1)}
\end{tikzpicture}}
\quad
\myvcenter{\begin{tikzpicture}[x={(4mm,0mm)},y={(0mm,4mm)}]
\sqtA{(0,0)}
\sqtB{(-1,-1)}
\sqtC{(1,-1)}
\end{tikzpicture}}
\quad
\myvcenter{\begin{tikzpicture}[x={(4mm,0mm)},y={(0mm,4mm)}]
\sqtA{(0,0)}
\sqtB{(-1,-1)}
\sqtC{(1,-1)}
\sqtA{(1,2)}
\end{tikzpicture}}
\quad
\myvcenter{\begin{tikzpicture}[x={(4mm,0mm)},y={(0mm,4mm)}]
\sqtA{(0,0)}
\sqtB{(-1,-1)}
\sqtC{(1,-1)}
\sqtA{(1,2)}
\sqtA{(-2,0)}
\sqtB{(-3,-1)}
\end{tikzpicture}}
\quad
\myvcenter{\begin{tikzpicture}[x={(4mm,0mm)},y={(0mm,4mm)}]
\sqtA{(0,0)}
\sqtB{(-1,-1)}
\sqtC{(1,-1)}
\sqtA{(1,2)}
\sqtA{(-2,0)}
\sqtB{(-3,-1)}
\sqtA{(0,-3)}
\sqtB{(-1,-4)}
\sqtC{(1,-4)}
\end{tikzpicture}}
\caption{Definition of the topological substitution $\tau$ obtained from the dual substitution associated with $1 \mapsto 13, 2 \mapsto 1, 3 \mapsto 2$ (top).
Six iterations from the tile $A$ are shown (bottom).}
\label{fig:exother}
\end{figure}

We describe how we derived the topological substitution $\tau$
from the dual substitution $\EOS(t)$ associated with the symbolic substitution
$t : 1 \mapsto 13, 2 \mapsto 1, 3 \mapsto 2$.
\begin{enumerate}
\item
Start with a single facet $[\mathbf 0,1]^*$
and compute $\EOS(t)^k([\mathbf 0,1]^*)$ with $k$ large enough,
in such a way that the patch $\EOS(t)^k([\mathbf 0,1]^*)$ contains
every possible neighboring couples of facets.
This is shown in Figure~\ref{fig:dualtotau} (left).

\item\label{item2}
Compute some more iterates by $\EOS(t)$
to ``inflate'' the tiles from single facets to patches of facets (metatiles).
This is shown in Figure~\ref{fig:dualtotau} (center),
where each metatile has the same color as its single-facet preimage.
We must iterate $\EOS(t)$ sufficiently many times ($3$ times in this case),
so that every intersection between two tiles is either empty
or consists of edges (single points are not allowed).

\item\label{item3}
Iterate $\EOS(t)$ one more time to ``read'' how the metatiles should be substituted.
This is where we extract the information to define the topological substitution $\tau$
in two steps:
    \begin{enumerate}
    \item We define the image of each tile by noticing that
        the image tiles are either one of the other metatiles,
        or a union of two metatiles:
    \[
    \myvcenter{\begin{tikzpicture}[x={(4mm,0mm)}, y={(0mm, 4mm)}]
    \fill[thick,fill=white,draw=black] (0,0)--(1,0)--(1,1)--(2,1)--(2,2)--(1,2)--(1,3)--(0,3)--cycle;
    \draw (0.5,0.5) node{\small$2$};
    \draw (0.5,1.5) node{\small$1$};
    \draw (1.5,1.5) node{\small$3$};
    \draw (0.5,2.5) node{\small$1$};
    \end{tikzpicture}}
    \mapsto
    \myvcenter{\begin{tikzpicture}[x={(4mm,0mm)}, y={(0mm, 4mm)}]
    \fill[thick,fill=white,draw=black] (0,0)--(1,0)--(1,1)--(2,1)--(2,2)--(1,2)--(1,3)--(0,3)--cycle;
    \draw (0.5,0.5) node{\small$2$};
    \draw (0.5,1.5) node{\small$1$};
    \draw (1.5,1.5) node{\small$3$};
    \draw (0.5,2.5) node{\small$1$};
    \fill[thick,fill=white,draw=black] (-1,-1)--(0,-1)--(0,1)--(-1,1)--cycle;
    \draw (-0.5,-0.5) node{\small$2$};
    \draw (-0.5,0.5) node{\small$1$};
    \end{tikzpicture}}
    \qquad
    \myvcenter{\begin{tikzpicture}[x={(4mm,0mm)}, y={(0mm, 4mm)}]
    \fill[thick,fill=white,draw=black] (-1,-1)--(0,-1)--(0,1)--(-1,1)--cycle;
    \draw (-0.5,-0.5) node{\small$2$};
    \draw (-0.5,0.5) node{\small$1$};
    \end{tikzpicture}}
    \mapsto
    \myvcenter{\begin{tikzpicture}[x={(4mm,0mm)}, y={(0mm, 4mm)}]
    \fill[thick,fill=white,draw=black] (0,0)--(1,0)--(1,1)--(2,1)--(2,2)--(0,2)--cycle;
    \draw (0.5,0.5) node{\small$2$};
    \draw (0.5,1.5) node{\small$1$};
    \draw (1.5,1.5) node{\small$3$};
    \end{tikzpicture}}
    \qquad
    \myvcenter{\begin{tikzpicture}[x={(4mm,0mm)}, y={(0mm, 4mm)}]
    \fill[thick,fill=white,draw=black] (0,0)--(1,0)--(1,1)--(2,1)--(2,2)--(0,2)--cycle;
    \draw (0.5,0.5) node{\small$2$};
    \draw (0.5,1.5) node{\small$1$};
    \draw (1.5,1.5) node{\small$3$};
    \end{tikzpicture}}
    \mapsto
    \myvcenter{\begin{tikzpicture}[x={(4mm,0mm)}, y={(0mm, 4mm)}]
    \fill[thick,fill=white,draw=black] (0,0)--(1,0)--(1,1)--(2,1)--(2,2)--(1,2)--(1,3)--(0,3)--cycle;
    \draw (0.5,0.5) node{\small$2$};
    \draw (0.5,1.5) node{\small$1$};
    \draw (1.5,1.5) node{\small$3$};
    \draw (0.5,2.5) node{\small$1$};
    \end{tikzpicture}}
    \]
    \item We define the boundaries' images by comparing the common edges between two adjacent metatiles
        and the common edges between their images in Figure~\ref{fig:dualtotau} (center and right).
    \end{enumerate}
\end{enumerate}

\begin{figure}[h!]
\centering
\myvcenter{\includegraphics[width=0.3\linewidth]{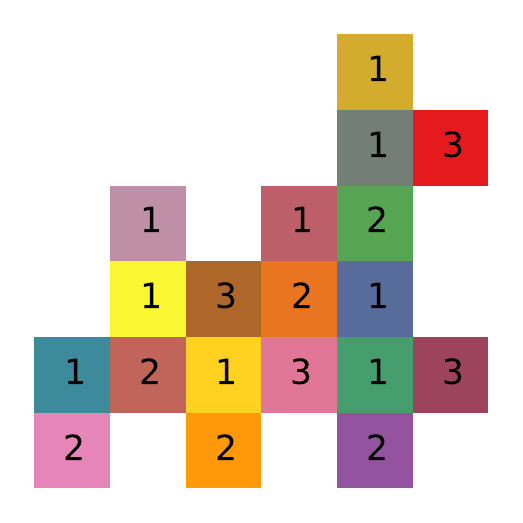}}
\hfil
\myvcenter{\includegraphics[width=0.3\linewidth]{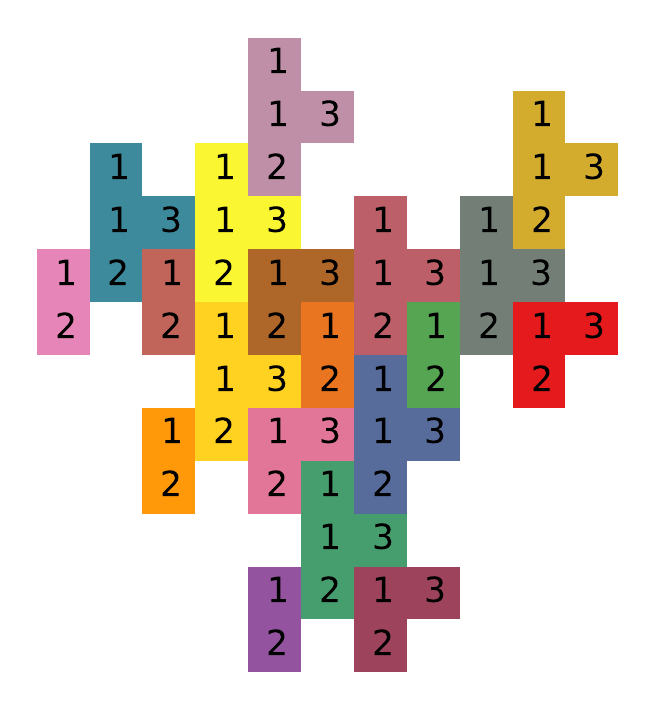}}
\hfil
\myvcenter{\includegraphics[width=0.3\linewidth]{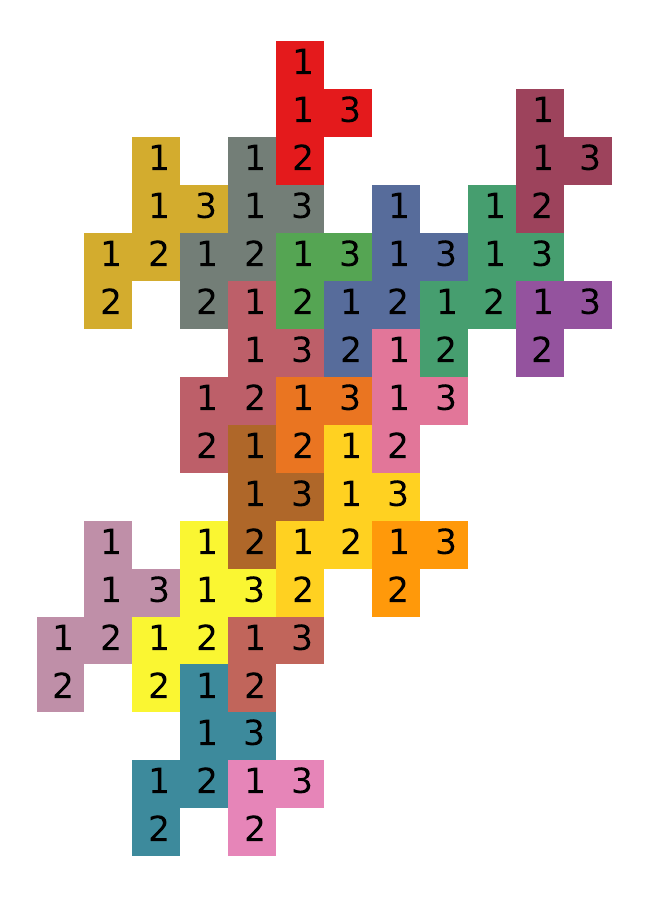}}
\caption{Deriving the topological substitution $\tau$ from the dual substitution $\EOS(t)$.
From left to right:
$\EOS(t)^7([\mathbf 0,1]^*)$,
$\EOS(t)^{10}([\mathbf 0,1]^*)$ and
$\EOS(t)^{11}([\mathbf 0,1]^*)$.
}
\label{fig:dualtotau}
\end{figure}

\paragraph{Limits of this approach}
Despite the fact that Rauzy fractals tilings have finite local complexity~\cite{ABBLS},
the method described above is not guaranteed to work in general for
an arbitrary dual substitution. The main problem is that in many cases,
the topology of the patterns produced by the dual substitution can be complicated
(disconnected or not simply connected, for example).
This can cause Step~\ref{item2} above to fail.

These difficulties are linked with some questions about the dynamics of the underlying Pisot substitution.
Indeed, it can be proved that the underlying Pisot substitution
has pure discrete spectrum if and only if the patterns generated by
its associated dual substitution contain arbitrarily large balls~\cite{CANT}.
This property is difficult to check for dual substitutions,
but easy to check for topological substitutions (see the core property in Section~\ref{sect:tritop}).
See~\cite{ABBLS} for more information about the Pisot conjecture and its different formulations.

Another possible approach to the original question raised in the introduction would be,
given an IFS with a topologically complicated attractor,
to construct another IFS with a topologically simpler attractor which gives a similar tiling,
and then apply the method described above.
For example, the tilings associated with the Tribonacci substitution and the
``flipped Tribonacci'' substitution $1\mapsto12, 2\mapsto31, 3\mapsto1$
are closely related: the tile positions are equal (but the neighbor relations change),
even though the topology of the flipped Tribonacci fractal is complicated.
However we do not know if this feasible in general, even in the case of Rauzy fractals.

\bibliographystyle{alpha}
\bibliography{biblio}

\end{document}